\documentclass[reqno]{amsart}

\usepackage{hyperref}
\usepackage{a4wide}
\usepackage{cleveref}
\usepackage{eucal}
\usepackage{newtxtext}
\usepackage{newtxmath}
\usepackage{tikz}
\usepackage{tikz-cd}
\usepackage{subcaption}
\usepackage[T1]{fontenc}
\usepackage{mathtools}
\numberwithin{equation}{section}

\newtheorem{thm}{Theorem}[section]
\newtheorem{lem}[thm]{Lemma}
\newtheorem{prop}[thm]{Proposition}
\newtheorem{cor}[thm]{Corollary}
\theoremstyle{definition}
\newtheorem{exam}[thm]{Example}
\newtheorem{defn}[thm]{Definition}

\newtheorem{question}[thm]{Question}

\newtheorem{remark}[thm]{Remark}

\AddToHook{env/lem/begin}{\crefalias{thm}{lem}}
\AddToHook{env/prop/begin}{\crefalias{thm}{prop}}
\AddToHook{env/cor/begin}{\crefalias{thm}{cor}}
\AddToHook{env/exam/begin}{\crefalias{thm}{exam}}
\AddToHook{env/defn/begin}{\crefalias{thm}{defn}}
\AddToHook{env/conj/begin}{\crefalias{thm}{conj}}
\AddToHook{env/question/begin}{\crefalias{thm}{question}}
\AddToHook{env/problem/begin}{\crefalias{thm}{problem}}
\AddToHook{env/remark/begin}{\crefalias{thm}{remark}}

\crefname{lem}{Lemma}{Lemmas}
\crefname{thm}{Theorem}{Theorems}
\crefname{prop}{Proposition}{Propositions}
\crefname{question}{Question}{Questions}
\crefname{defn}{Definition}{Definitions}
\crefname{conj}{Conjecture}{Conjectures}
\crefname{figure}{Figure}{Figures}
\crefname{cor}{Corollary}{Corollaries} 
\crefformat{equation}{(#2#1#3)}
\Crefformat{equation}{Equation #2(#1)#3}

\newcommand\SHE{\operatorname{SHE}}
\newcommand\HE{\operatorname{HE}}

\DeclareMathOperator*\conv{\scalebox{1.5}{\raisebox{-0.2ex}{\(\ast\)}}}
\DeclareMathOperator{\crit}{Crit}

\newcommand\loops{\ell}
\newcommand\TC{\sfC^{\mathsf{top}}}
\newcommand\Grape{\mathcal{G}\mathsf{rape}}
\newcommand\Graph{\mathcal{G}}

\newcommand\Emb{\mathcal{E}\mathsf{mb}}
\newcommand\ess{\mathsf{ess}}
\newcommand\cell{\mathsf{cell}}

\newcommand\bbF{\mathbb{F}}

\newcommand\bbR{\mathbb{R}}
\newcommand\bbS{\mathbb{S}}
\newcommand\bbZ{\mathbb{Z}}
\newcommand\bft{\mathbf{t}}
\newcommand\bfx{\mathbf{x}}
\newcommand\sfC{\mathsf{C}}

\newcommand\sfT{\mathsf{T}}
\newcommand\sfa{\mathsf{a}}

\newcommand\sfe{\mathsf{e}}

\newcommand\sfg{\mathsf{g}}
\newcommand\sfh{\mathsf{h}}

\newcommand\sfv{\mathsf{v}}
\newcommand\sfw{\mathsf{w}}

\newcommand\cP{\mathcal{P}}
\newcommand\cR{\mathcal{R}}
\newcommand\cD{\mathcal{D}}

\newcommand{\vertex}{\mathsf{\thevertex}\stepcounter{vertex}}

\renewcommand\vec[1]{\mathbf{#1}}
\newcommand\norm[1]{\lVert #1 \rVert}

\newcommand{\grape}[2][0]{
\draw[thick, fill] (#2.center) circle(2pt) -- +({30+#1}:0.5) +(0,0) -- +({-30+#1}:0.5);
\draw[thick] (#2.center) +({30+#1}:0.5) arc ({120+#1}:{-120+#1}:{0.5 / sqrt(3)});
}
\newcommand{\dgrape}[2][0]{
\draw[thick, fill] (#2.center) circle(1.5pt) -- +({-30+#1}:0.5) circle (1.5pt);
\draw[thick, fill] (#2.center) +({30+#1}:0.5) circle (1.5pt) +({#1}:{1.5/sqrt(3)}) circle (1.5pt);
\draw[thick, dashed, fill] (#2.center) -- +({30+#1}:0.5);
\draw[thick] (#2.center) +({30+#1}:0.5) arc ({120+#1}:{-120+#1}:{0.5 / sqrt(3)});
}
\newcommand{\ddgrape}[2][0]{
\draw (#2.center) -- +({-30+#1}:0.5) arc ({-120+#1}:{120+#1}:{0.5 / sqrt(3)});
}

\title[Hilbert polynomials of configuration spaces]{Hilbert polynomials of configuration spaces over graphs of circumference at most 1}
\author{Byung Hee An}
\address{\parbox{\linewidth}{Department of Mathematics Education, Kyungpook National University, Daegu, Korea\\
Department of STEM Education, North Carolina State University, NC, United States, Visiting Scholar}}
\email{anbyhee@knu.ac.kr, ban2@ncsu.edu}

\author{Jang Soo Kim}
\address{Department of Mathematics,
Sungkyunkwan University (SKKU), Suwon, Gyeonggi-do 16419, South Korea}
\email{jangsookim@skku.edu}
\keywords{Configuration spaces, Circumference of a graph, Hilbert polynomials}
\subjclass[2020]{Primary: 20F36, 55R80; Secondary: 05C10, 13E15}

\begin{document}

\begin{abstract}
The \( k \)-configuration space \( B_k\Gamma \) of a topological space \( \Gamma \) is the space of sets of \( k \) distinct points in \( \Gamma \). In this paper, we consider the case where \( \Gamma \) is a graph of circumference at most \(1\). We show that for all \( k\ge0 \), the \( i \)-th Betti number of \( B_k\Gamma \) is given by a polynomial \(P_\Gamma^i(k)\) in \( k \), called the Hilbert polynomial of \( \Gamma \). We find an expression for the Hilbert polynomial \(P_\Gamma^i(k)\) in terms of those coming from the canonical \(1\)-bridge decomposition of \( \Gamma \). We also give a combinatorial description of the coefficients of \(P_\Gamma^i(k)\).
\end{abstract}

\maketitle
\tableofcontents

\section{Introduction}
\label{sec:introduction}

Given a graph \(\Gamma=(V,E)\), the (\emph{unordered})
\emph{\(k\)-configuration space} \(B_k\Gamma\) on \(\Gamma\) is the
topological space consisting of sets of \(k\) distinct points in
\(\Gamma\):
\[
B_k\Gamma = \left\{
\{x_1,\dots, x_k\}\subseteq \Gamma: x_i\neq x_j \mbox{ for all \( i\ne j \)}
\right\}.
\]
The space \(B_k\Gamma\) was firstly introduced by Ghrist and Abrams
\cite{Ghrist2001,Abrams2000}, and its fundamental group is called the \emph{graph
  braid group}.

Graph braid groups and their homology groups have been studied by many
researchers. Abrams \cite{Abrams2000} introduced a combinatorial model
for their configuration spaces and showed their \(K(\pi,1)\)- and
\(CAT(0)\)-ness. Farley and Sabalka \cite{Farley2006, FS2005,FS2008,Sabalka2009}
applied Forman's discrete Morse theory to the combinatorial model to
provide a way to compute group presentations, and analyze the
cohomology ring of graph braid groups. Kim, Ko, La, and Park
\cite{KP2012,KKP2012,KLP2016} characterized various group-theoretic
properties of graph braid groups and their first homology groups.

Rather recently, in \cite{Ramos2018,ADCK2020,ADCK2022}, Ramos introduced a chain level edge action for trees, and Knudsen, Drummond-Cole,
and the first author generalized it to a continuous edge action for arbitrary graphs on the total configuration space \(B\Gamma=\coprod_{k\ge0}B_k\Gamma\), which
has two important consequences: for a fixed field \(\bbF\) and
each \(i\ge 0\),
\begin{enumerate}
\item[(i)] the homology group \(H_i(B\Gamma)\) over \(\bbF\)
is finitely generated as an \(\bbF[E]\)-module, and 
\item[(ii)] the \(i\)-th
Betti number \(\dim_\bbF H_i(B_k\Gamma)\) is equal to a polynomial
\(P_\Gamma^i(k)\) in \( k \) for sufficiently large \(k\).
\end{enumerate}
Moreover, they found a combinatorial description for the leading term
of \(P_\Gamma^i(k)\).

\subsection{Motivating question and main results}
The motivating question of this paper is as follows:
\begin{question}\label{que:1}
  Given a fixed \(i\ge 0\), at what value of \(k\) does the
  polynomiality begin? In other words, what is the minimal \(k_0\)
  such that \(\dim_\bbF H_i(B_k\Gamma) = P^i_\Gamma(k)\) for all
  \(k\ge k_0\)? Equivalently, in terms of homological algebra, what is
  the \emph{regularity} of \(H_i(B\Gamma)\) as an \(\bbF[E]\)-module?
\end{question}

There are some known results regarding this question. Let \(\Gamma\)
be a connected graph. Then it is easy to observe that the \(0\)-th
Betti number stabilizes to the constant \(1\) from the beginning,
namely, \(k_0=0\) if \(\Gamma\) is not a singleton, and it stabilizes
to the constant \(0\) from \(k_0=2\) if \( \Gamma \) is a singleton.

According to the result of Ko and Park \cite{KP2012}, the first Betti
number of \( B_k\Gamma \) over a field \(\bbF\) of characteristic
\(0\) is given as follows:
\[
\dim_\bbF(H_1(B_k\Gamma)) = 
\begin{cases}
0 & k=0;\\
\dim_\bbF(H_1(\Gamma)) & k=1;\\
N_1(\Gamma,k)+N_2(\Gamma)+N_3(\Gamma)+\dim_\bbF(H_1(\Gamma)) & k\ge2.
\end{cases}
\]
Here, the quantities \(N_1(\Gamma,k)\), \(N_2(\Gamma)\), and
\(N_3(\Gamma)\) count the contributions of vertex-\(1\)-cuts,
vertex-\(2\)-cuts, and triconnected components. One may further
observe that \(N_1(\Gamma,k)\) is a polynomial in \(k\) with
\(N_1(\Gamma,1)=0\). Hence, \(k_0=2\) if \(N_2(\Gamma)\neq0\) or
\(N_3(\Gamma)\neq 0\), namely, if \(\Gamma\) has a nontrivial
vertex-\(2\)-cut or a triconnected subgraph. A similar criterion can
also be made for \(k_0=0\) and \(k_0=1\).

While we have a complete answer \Cref{que:1} for \(i\le 1\), only a
few results are known for \(i\ge 2\). 
For trees, Ramos also provided a complete answer for all \(i\ge 2\) \cite[Theorem~A]{Ramos2018}.
If \(\Gamma\) has exactly two
essential vertices, then all homology groups \(H_*(B\Gamma)\) of
degree \(3\) or higher vanish. Here, an \emph{essential vertex} is a
vertex of degree at least \( 3 \). Therefore, by using Euler
characteristics and the above observations on lower degrees, one may
compute the polynomial \(P_\Gamma^2(k)\) explicitly and conclude that
the value of \(k_0\) for \(i=2\) is the same as that for \(i=1\). In
general, there is a computational result only for graphs with at most
\(9\) edges in \cite{GCDC2019}, which provides Betti numbers, stable
ranges, and stable polynomials with visualized graphs.

In this paper, we consider a family of graphs whose circumferences are
at most \(1\). Here, the \emph{circumference} of a graph \(\Gamma\) is
the maximal edge length of simple cycles in \(\Gamma\). We call such
graphs \emph{bunches of grapes} due to their shapes: any bunch of
grapes \(\Gamma\) can be obtained by attaching \(\ell(\sfv)\) loops,
called \emph{grapes}, to each vertex \(\sfv\) of a tree \(\sfT\)
called a \emph{stem}. See \Cref{figure:example of a bunch of grapes}.
In this case, we simply write \(\Gamma\cong(\sfT,\ell)\). Then one can
apply Ko--Park's result on the first homology groups and conclude that
the first Betti number stabilizes from the beginning.

A useful property of a bunch of grapes \(\Gamma\) is that any edge in
the stem decomposes \(\Gamma\) into two subgraphs, which are again
bunches of grapes. Therefore, any bunch of grapes can be built from
bunches of grapes \(\Gamma_\sfv\) having exactly one essential vertex
\(\sfv\). Such a graph \(\Gamma_\sfv\) is called an \emph{elementary
  bunch of grapes}, which is completely determined by the numbers of
loops and leaves, called the \emph{local data} of \(\sfv\).

The main results of this paper are as follows.
\begin{thm}[\Cref{theorem:decomposition formula}]
  Let \(\Gamma\cong(\sfT,\loops)\) be a bunch of grapes, and let
  \(P_{\Gamma}^i(k)\) and \(P_{\Gamma_\sfv}^j(k)\) be the Hilbert
  polynomials of \(H_i(B_k\Gamma)\) and \(H_j(B_k\Gamma_{\sfv})\),
  respectively. Then for each \(i\ge 1\) and \(k\ge 0\),
\[
\dim_\bbF H_i(B_k\Gamma)=P^i_\Gamma(k) = \sum_{\substack{W\subseteq V^{\ess}(\Gamma)\\|W|=i}} \left(\conv_{\sfv\in W} P_{\Gamma_\sfv}^1\right)(k),
\]
where 
\[
(p_1*p_2)(k)=\sum_{j=0}^{d_1+d_2}\sum_{\substack{j_1+j_2=j\\j_1,j_2\ge 0}}b_{1,j_1}b_{2,j_2}\binom k j,
\]
for polynomials
\(p_1(k)=\sum_{j=0}^{d_1} b_{1,j}\binom k j\) and \(p_2(k)=\sum_{j=0}^{d_2} b_{2,j}\binom k j\).
\end{thm}

Furthermore, for a bunch of grapes \(\Gamma\), the sequence of Hilbert
polynomials \((P_\Gamma^0(k), P_\Gamma^1(k),\dots)\) determines the
local data \(\cD(\Gamma)\) of \(\Gamma\) consisting of the local data
of all essential vertices of \(\Gamma\).

\begin{thm}[\Cref{thm:local data}]
For two bunches of grapes \(\Gamma_1\) and \(\Gamma_2\),  
\[
\cD(\Gamma_1)=\cD(\Gamma_2) \Longleftrightarrow 
P_{\Gamma_1}^i(k)\equiv P_{\Gamma_2}^i(k) \quad\mbox{ for all } i\ge 0.
\]
\end{thm}

While the leading term of \(P_\Gamma^i(k)\) for any graph \(\Gamma\)
has a combinatorial description, little is known about the other
coefficients, even for \(P_\Gamma^1(k)\), i.e., for the first homology
group. In this paper, we use an expansion of a polynomial with respect
to the binomial coefficients \(\left\{\binom k j : j\ge 0\right\}\)
instead of the usual basis \(\{k^j : j\ge 0\}\). Then one may ask what
is the combinatorial meaning of the coefficient of \(P_\Gamma^i(k)\).
That is, what does \(b_{\Gamma;j}^i\) count if
\(P_\Gamma^i(k)=\sum_{j\ge 0} b_\Gamma^i \binom k j\)? We provide a
complete answer to this question in \Cref{thm:1} when \( \Gamma \) is
a bunch of graphs. Moreover, we have the following explicit
correspondence between our combinatorial models \(\HE^i_\Gamma(k)\)
and certain homology cycles that form a basis of \(H_i(B_k\Gamma)\).
Here, \(\HE^i_\Gamma(k)\) denotes the set of configurations on
\(\Gamma\) consisting of points and half-edges satisfying certain
conditions; see \Cref{defn:SHE for grapes} for the precise definition
and \Cref{fig:image21} for an example.

\begin{thm}[\Cref{thm:combinatorial description for homology cycles}]
  Let \(\Gamma\cong(\sfT,\loops)\) be a bunch of grapes such that
  \(\sfT\) has at least one edge. Then for each \(i\ge 0\) and
  \(k\ge 0\), there is an injective map
\[
\Psi:\HE^i_\Gamma(k)\to H_i(B_k\Gamma)
\]
whose image forms an \(\bbF\)-basis for \(H_i(B_k\Gamma)\).
\end{thm}

The map \(\Psi\) is described in \Cref{section:the map alpha}. 
Note that for a tree \(\sfT\), Farley \cite{Farley2006} also provided a basis for \(H_i(B_k\sfT)\) for every \(i\) and \(k\). The difference between Farley's basis and ours will be discussed in \Cref{section:comparison}.

\subsection{Future directions}
We propose the following questions as future directions.
\begin{enumerate}
\item \emph{Cohomology rings.} Recall from \Cref{thm:local data} that
  for a bunch of grapes \( \Gamma \), the Hilbert polynomials
  \(P_{\Gamma}^i(k)\) determine the local data of \(\Gamma\), and vice
  versa. Since there are non-homeomorphic bunches of grapes with the
  same local data, the Hilbert polynomials do not determine
  \( \Gamma \).

  By the result of Sabalka \cite{Sabalka2009}, if \( \Gamma \) and
  \( \Gamma' \) are trees, the cohomology rings \(H^*(B\Gamma)\) and
  \(H^*(B\Gamma')\) are isomorphic if and only if \( \Gamma \) and
  \( \Gamma' \) are homeomorphic. Hence one may ask the following relevant
  question:
\begin{question}
  For a bunch of grapes \(\Gamma\), does the cohomology ring
  \(H^*(B\Gamma)\) determine \(\Gamma\)?
\end{question}
\item \emph{Higher circumferences.} A graph \( \Gamma \) with
  circumference \(2\) is obtained from a bunch of grapes by replacing
  each edge of its stem \( \sfT \) with a multiple edge. Hence
  \( \Gamma \) is determined by a triple
  \((\sfT,\ell,\operatorname{mult})\), where
  \(\operatorname{mult}:E(\sfT)\to\bbZ_{\ge1}\) indicates the
  multiplicity of each edge.

  One of the simplest cases is when the stem \(\sfT\) is an edge. Then
  \(\Gamma\) is homeomorphic to a complete bipartite graph
  \(\mathsf{K}_{2,M}\) for some \(M\ge 3\). In this case, it is known
  that the polynomiality begins at \(k=2\) for the first and second
  homology groups. Another simple case is when the multiplicity of
  each edge is at most \(2\). Then one can check that the
  polynomiality begins at \(k=1\) for the first homology group.

  Even though graphs with circumference \(2\) share some properties
  with bunches of grapes, the action of multiple edges does not yield
  an injection, so \Cref{proposition:free action of leaves} fails. It
  might suggest that we need a resolution of length at least \(1\)
 in order to compute the Betti number of modules coming from
  its bridge decomposition. However, since multiple edges share their
  endpoints, there is still at most one half edge occupied for each
  endpoint in a \'Swi\k atkowski chain complex, which is a finitely
  generated model over \(\bbF[E]\) described in \cite{ADCK2022}.
\begin{question}
For a graph \(\Gamma\) with circumference \(2\), does the polynomiality of the Betti number \(\dim_\bbF H_i(B_k\Gamma)\) begin at \(k=1\) or \(2\)?
What if the circumference of \(\Gamma\) is larger than \(2\)?
\end{question}

\item \emph{Regularities.} One may ask the following
  questions on the regularity of \(H_i(B\Gamma)\) as an
  \(\bbF[E]\)-module:
\begin{question}
  For a general graph \(\Gamma=(V,E)\) and \(i\ge 2\), is there an
  expression for a lower bound of the regularity of \(H_i(B\Gamma)\)
  in terms of graph-theoretic properties of \( \Gamma \) such as the
  numbers of vertices and edges, some connectivity invariants, and so
  on?
\end{question}
\begin{question}
  For a fixed \(i\ge 2\), is there a graph \(\Gamma\) with
  arbitrarily large regularity of \(H_i(B\Gamma)\)?
\end{question}
\end{enumerate}

\subsection{Outline of the paper}
In \Cref{prelim}, we briefly review the basics of generating
functions, graphs, and configuration spaces, and their
Hilbert--Poincar\'e series. \Cref{edge stabilization} introduces the
edge stabilization and considers one-bridge decompositions. In
\Cref{sec:hilbert polynomials}, we compute the Hilbert polynomial
\(P_\Gamma^i(k)\) of \(H_i(B_k\Gamma)\) for a bunch of grapes
\(\Gamma\) and prove \Cref{theorem:decomposition formula,thm:local
  data}. In \Cref{sec:comb-descr-hilb}, we provide a combinatorial
description of the coefficients of \(P_\Gamma^i(k)\), for a bunch of
grapes \(\Gamma\) whose stem has at least one edge, and we prove
\Cref{thm:1}. In \Cref{homology cycles}, an explicit correspondence
between combinatorial objects given in the previous section and
certain homology cycles that form an \(\bbF\)-basis of
\(H_i(B\Gamma)\) (\Cref{thm:combinatorial description for homology
  cycles}). Finally, in \Cref{appendix}, we consider bunches of grapes
whose stems are singletons.

\subsection*{Acknowledgments}
The authors would like to express their sincere gratitude to Eric Ramos. His insightful comments greatly contributed to the development and refinement of this paper.
The first author was supported by Samsung Science and Technology Foundation under Project Number SSTF-BA2202-03. 
The second author was supported by the National Research Foundation of Korea (NRF) grant funded by the Korea government RS-2025-00557835.

\section{Preliminaries}\label{prelim}

In this section, we introduce the necessary definitions and notations.

\subsection{Generating functions}
Let \((c_k)_{k\in\bbZ}\) be a sequence with \(c_{-1}=c_{-2}=\cdots=0\).
We say that \(c_k\) is \emph{eventually a polynomial} if there exists a (unique) polynomial \(p(x)\) such that \(c_k = p(k)\) for every sufficiently large \(k\).
Then the generating function \(f(y)=\sum_{k\ge 0} c_k y^k\) of such a sequence \((c_k)_{k\in \bbZ}\) can be expressed uniquely as
\[
f(y) =r(y) + \sum_{k\ge 0} p(k)y^k
\]
for some polynomials \(p(x)\) and \(r(x)\).

Let \(\binom k j\coloneqq\frac{k(k-1)\cdots(k-j+1)}{j!}\) be the
binomial coefficient, which is a polynomial in \(k\) of degree \(j\).
Since \( \{k^j:j\ge0\} \) and \( \{\binom{k}{j}:j\ge0\} \) are bases of the space of polynomials,
any polynomial \(p(k)\) of degree \( d \) can be written uniquely as
\[
p(k)=\sum_{j=0}^d a_j k^j = \sum_{j=0}^d b_j \binom k j,
\]
where \( a_d\ne0 \) and \( b_d\ne 0 \). We call \( a_dk^d \) and
\( b_d \binom{k}{d} \) the \emph{leading term} and
\emph{\( b \)-leading term} of \( p(k) \), respectively. We also call
\( a_d \) and \( b_d \) the \emph{leading coefficient} and
\emph{\( b \)-leading coefficient} of \( p(k) \), respectively.

For \(p(k)=\sum_{j=0}^d b_j\binom k j\), we define the \emph{shifted
  polynomial} \(p^+(k)\) by
\[
p^+(k)=\sum_{j=0}^d b_j \binom {k+1} {j+1}.
\]
One can easily check that for an integer \(k\ge 0\),
\[
  p^+(k) = p(0)+p(1)+\cdots+p(k).
\]

\begin{defn}\label{defn:conv}
For two polynomials 
\[
p_1(k)=\sum_{j=0}^{d_1} b_{1,j}\binom k j
\quad\text{ and }\quad 
p_2(k)=\sum_{j=0}^{d_2} b_{2,j}\binom k j,
\]
we define a polynomial \((p_1*p_2)(k)\) by
\begin{align}\label{equation:product}
(p_1*p_2)(k)=\sum_{j=0}^{d_1+d_2}\sum_{\substack{j_1+j_2=j\\j_1,j_2\ge 0}}b_{1,j_1}b_{2,j_2}\binom k j.
\end{align}
\end{defn}

The binary operation \(*\) is associative and commutative,
and it satisfies the distributive law. Moreover, we have
\((p_1*p_2)(k) = c p_2(k)\) for a constant function
\(p_1(k)\equiv c\). Hence, the constant function
\(\mathbf{1}(k)\equiv 1\) plays the role of the identity under the
operation \(*\).

For a finite set of polynomials \(\{p_i(k): i\in \mathcal{I}\}\)
indexed by \(\mathcal{I}\), the polynomial
\(\left(\conv_{i\in\mathcal{I}} p_i\right)(k)\) is well-defined. We
define the empty product to be the constant function \(\mathbf{1}\),
that is, 
\[
\conv_{i\in \varnothing}\equiv \mathbf{1}.
\]

\begin{lem}\label{lemma:product}
  For polynomials \(p_1(k)\) and \(p_2(k)\), let
\[
f_1(y)=\sum_{k\ge 0}p_1(k)y^k\quad\text{ and }\quad 
f_2(y)=\sum_{k\ge 0}p_2(k)y^k.
\]
Then 
\[
f_1(y)f_2(y) = \sum_{k\ge0} (p_1*p_2)^+(k) y^k.
\]
\end{lem}
\begin{proof}
  Let
  \[
p_1(k)=\sum_{j\ge 0} b_{1,j}\binom k j
\quad\text{ and }\quad 
p_2(k)=\sum_{j\ge 0} b_{2,j}\binom k j,
\]
Then
\begin{align*}
  f_1(y)f_2(y)
  &= \sum_{k_1,k_2\ge0} \sum_{j_1,j_2\ge0} b_{1,j_1} b_{2,j_1}
  \binom{k_1}{j_1} \binom{k_2}{j_2} y^{k_1+k_2},\\
    \sum_{k\ge0} (p_1*p_2)^+(k) y^k
  &= \sum_{k\ge0} \sum_{j_1,j_2\ge0} b_{1,j_1} b_{2,j_2} \binom{k+1}{j_1+j_2+1} y^k.
\end{align*}
Thus, it suffices to show that, for fixed nonnegative integers \( j_1 \), \( j_2 \), and \( k \),
\begin{equation}\label{eq:10}
 \sum_{\substack{k_1+k_2=k\\ k_1,k_2\ge0}} \binom{k_1}{j_1} \binom{k_2}{j_2}=
  \binom{k+1}{j_1+j_2+1}.
\end{equation}
The right-hand side \( \binom{k+1}{j_1+j_2+1} \) of \eqref{eq:10} is
the number of subsets \( A=\{a_1 < \cdots < a_{j_1+j_2+1}\} \) of
\( \{1,2,\dots,k+1\} \). The number of such \( A \)'s with
\( a_{j_1+1}=k_1+1 \), where \( k_1+k_2=k \), is
\( \binom{k_1}{j_1} \binom{k_2}{j_2} \). Summing over all such pairs
\( (k_1,k_2) \) gives the left-hand side of \eqref{eq:10}, completing
the proof.
\end{proof}

\subsection{Basics on graphs}
Throughout this paper, a \emph{graph} \(\Gamma\) means a finite and
regular CW complex of dimension one\footnote{Hence, graphs without
  edges are not allowed.}, whose \(0\)-cells and \(1\)-cells are
called \emph{vertices} and \emph{edges}, respectively. By denoting the
sets of vertices and edges by \(V\) and \(E\), respectively, we may
write \(\Gamma=(V,E)\) as usual. A graph may have loops and multiple
edges, where a \emph{loop} is an edge whose boundary is a single
vertex and \emph{multiple edges} are edges with the same boundaries.

Considering each edge \(\sfe\) as a map \(\sfe:[0,1]\to\Gamma\), the
two restrictions \(\sfe|_{[0,1/3)}\) and \(\sfe|_{(2/3,1]}\) of
\(\sfe\) are called the \emph{half edges} of \(\sfe\). We say that
an half edge \(\sfh\) of \(\sfe\) is \emph{attached to a vertex}
\(\sfv\) if \(\sfv\in \sfh\). For each vertex \(\sfv\in \Gamma\), we
denote the set of half edges attached to \(\sfv\) by
\(\mathcal{H}_\sfv\), and the \emph{degree} \(\deg(\sfv)\) of \(\sfv\)
is defined to be the cardinality of \(\mathcal{H}_\sfv\). An
\emph{essential} vertex is a vertex of degree at least three. We say
that a graph is \emph{nontrivial} if it contains at least one
essential vertex.

Let \(\Graph\) be the set of finite and connected graphs. Two graphs
are said to be \emph{topologically isomorphic} or \emph{homeomorphic}
if they are homeomorphic as topological spaces. Equivalently, two
graphs are homeomorphic if and only if they are isomorphic as graphs
up to \emph{subdivision} and \emph{smoothing}, which are the
operations that respectively add and remove a vertex of degree two in
the middle of an edge. See \Cref{figure:subdivision and smoothing}. We
write \( \Gamma=\Gamma' \) and \( \Gamma\cong\Gamma' \) if \(\Gamma\)
and \(\Gamma'\) are isomorphic and homeomorphic, respectively.

\begin{figure}
\[
\begin{tikzcd}[column sep=4pc]
\begin{tikzpicture}[baseline=-.5ex]
\draw[dashed] (0,0) circle (1);
\draw (-1,0) -- node[above] {\(\sfe\)} (1,0);
\end{tikzpicture}\ar[r,"\text{subdiv.}", yshift=.5ex]&
\begin{tikzpicture}[baseline=-.5ex]
\draw[dashed] (0,0) circle (1);
\draw (-1,0) -- (0,0) node[midway, above] {\(\sfe'\)} -- node[midway, above] {\(\sfe''\)} (1,0);
\draw[fill] (0,0) circle (2pt) node[above] {\(\sfv\)};
\end{tikzpicture}\ar[l,"\text{smoothing}", yshift=-.5ex]
\end{tikzcd}
\]
\caption{The subdivision along \(\sfe\) and the smoothing at \(\sfv\).}
\label{figure:subdivision and smoothing}
\end{figure}

Let \(E_0=\{\sfe_1,\dots,\sfe_m\}\subseteq E\) be a subset of edges of a connected graph \(\Gamma=(V,E)\).
We say that \(E_0\) is an \emph{edge-\(m\)-cut} of \(\Gamma\) if it satisfies the following conditions:
\begin{enumerate}
\item the graph \((V,E\setminus E_0)\) is disconnected, and
\item for each proper subset \(E_0'\subsetneq E_0\), the graph \((V,E\setminus E_0')\) is connected.
\end{enumerate}
In this case, we define the \emph{\(m\)-bridge decomposition} of
\( \Gamma \) along \(E_0\) to be the graph obtained from \( \Gamma \)
by cutting the edges in \( E_0 \) as depicted in
\Cref{figure:cutting}. The edges may be renamed, if necessary.

\begin{figure}
\[
\begin{tikzcd}
\Gamma=
\begin{tikzpicture}[baseline=-.5ex]
\draw (-1.5,-1) rectangle (-0.5,1) (0.5,-1) rectangle (1.5,1);
\draw (-0.5,0.8) -- node[above=-.75ex] {\(\sfe_1\)} (0.5,0.8) ;
\draw (-0.5,0.2) -- node[above=-.75ex] {\(\sfe_2\)} (0.5,0.2) ;
\draw (0,-0.5) node {\(\vdots\)};
\draw (-0.5,-0.8) -- node[below=-.5ex] {\(\sfe_m\)} (0.5,-0.8) ;
\end{tikzpicture}\ar[r]&
\begin{tikzpicture}[baseline=-.5ex]
\draw (-1.5,-1) rectangle (-0.5,1) (0.5,-1) rectangle (1.5,1);
\draw[fill] (-0.5,0.8) -- (0,0.8) circle(2pt) -- (0.5,0.8);
\draw[fill] (-0.5,0.2) -- (0,0.2) circle(2pt) -- (0.5,0.2);
\draw (0,-0.5) node {\(\vdots\)};
\draw[fill] (-0.5,-0.8) -- (0,-0.8) circle(2pt) -- (0.5,-0.8);
\end{tikzpicture}\ar[r]&
\Gamma_1=\begin{tikzpicture}[baseline=-.5ex]
\draw (-1.5,-1) rectangle (-0.5,1);
\draw[fill] (-0.5,0.8) -- node[midway,above=-.75ex] {\(\sfe_1\)} (0,0.8) circle(2pt) node[right] {\(\sfv_1\)};
\draw[fill] (-0.5,0.2) -- node[midway,above=-.75ex] {\(\sfe_2\)} (0,0.2) circle(2pt) node[right] {\(\sfv_2\)};
\draw (-0.25,-0.5) node {\(\vdots\)};
\draw[fill] (-0.5,-0.8) -- node[midway,below=-.5ex] {\(\sfe_m\)} (0,-0.8) circle(2pt) node[right] {\(\sfv_m\)};
\end{tikzpicture}
\begin{tikzpicture}[baseline=-.5ex]
\draw (0.5,-1) rectangle (1.5,1);
\draw[fill] (0,0.8) circle(2pt) node[left] {\(\sfv_1\)} -- node[midway,above=-.75ex] {\(\sfe_1\)} (0.5,0.8);
\draw[fill] (0,0.2) circle(2pt) node[left] {\(\sfv_2\)} -- node[midway,above=-.75ex] {\(\sfe_2\)} (0.5,0.2);
\draw (0.25,-0.5) node {\(\vdots\)};
\draw[fill] (0,-0.8) circle(2pt) node[left] {\(\sfv_m\)} -- node[midway,below=-.5ex] {\(\sfe_m\)} (0.5,-0.8);
\end{tikzpicture}=\Gamma_2
\end{tikzcd}
\]
\caption{An \(m\)-bridge decomposition of \(\Gamma\).}
\label{figure:cutting}
\end{figure}

Let \(\gamma=(\sfv_0,\dots,\sfv_n)\) be a simple cycle of \(\Gamma\),
that is, \(\sfv_i\) and \(\sfv_{i+1}\) are joined by an edge for each
\(0\le i<n\), \(\sfv_0=\sfv_n\), and \(\sfv_1,\dots, \sfv_n\) are
distinct.
We call the number \(n\) the \emph{length} of \(\gamma\).

The \emph{circumference} \(\sfC(\Gamma)\) of \(\Gamma\) is the maximum
number of edges in a simple cycle of \(\Gamma\). The \emph{topological
  circumference} \(\TC(\Gamma)\) of \(\Gamma\) is the minimum
circumference among all graphs homeomorphic to \(\Gamma\), that is,
\[
\TC(\Gamma)=\min\{\sfC(\Gamma'): \Gamma\cong\Gamma'\}.
\]
Note that the circumference is an invariant of the isomorphic classes
of graphs, while the topological circumference an invariant of the
homeomorphic classes of graphs. For any graph \(\Gamma\) containing at
least one essential vertex, the topological circumference is equal to
the maximum of numbers of essential vertices among simple cycles in
\(\Gamma\).

Let us denote the set of (homeomorphic classes of) graphs with topological circumference at most \(i\) by \(\Graph_{i}\). Then we have a strictly increasing filtration \(\Graph_\bullet\) as follows:
\[
\Graph_0 \subsetneq \Graph_1 \subsetneq \Graph_2 \subsetneq \cdots \subsetneq \Graph.
\]

One can observe the following: the set \(\Graph_0\) consists of trees,
and any graph in \(\Graph_1\) is obtained by attaching loops to a
tree. A \emph{bunch of grapes} is a graph with topological
circumference at most one. We will denote the set of bunches of grapes
by \(\Grape\) instead of \(\Graph_1\). See \Cref{figure:example of a
  bunch of grapes}.

\begin{figure}
\begin{align*}
\Gamma&\cong\begin{tikzpicture}[baseline=-.5ex]
\draw[thick,fill] (0,0) circle (2pt) node (A) {} -- ++(5:2) circle(2pt) node(B) {} -- +(0,1.5) circle (2pt) node (C) {} +(0,0) -- +(0,-1.5) circle (2pt) node(D) {} +(0,0) -- ++(-10:1.5) circle(2pt) node(E) {} -- ++(10:1.5) circle(2pt) node(F) {} -- +(120:1) circle(2pt) node(G) {} +(0,0) -- +(60:1) circle (2pt) node(H) {} +(0,0) -- ++(-5:1) circle (2pt) node(I) {} -- +(30:1) circle(2pt) node(J) {} +(0,0) -- +(-30:1) circle (2pt) node(K) {};
\draw[thick, fill] (E.center) -- ++(0,-1) circle (2pt) -- +(-60:1) circle (2pt) +(0,0) -- +(-120:1) circle (2pt);
\grape[180]{A};
\grape[0]{C}; \grape[90]{C}; \grape[180]{C};
\grape[-90]{D};
\grape[-90]{F};
\grape[120]{G};
\grape[60]{H};
\grape[-30]{K};
\end{tikzpicture}&
\Gamma_{\ell,m}&=\ell\left\{\begin{tikzpicture}[baseline=-.5ex]
\draw[thick,fill] (0,0) circle (2pt) node (X) {} -- (0:1) circle (2pt) (0,0) -- (30:1) circle (2pt) (0,0) -- (-30:1) circle (2pt);
\grape[90]{X};
\grape[180]{X};
\grape[270]{X};
\end{tikzpicture}
\right\}m
\end{align*}
\caption{A bunch of grapes (left) and an elementary bunch of grapes (right).}
\label{figure:example of a bunch of grapes}
\end{figure}

Let \(\Gamma\) be a graph containing only one essential vertex. Then,
up to homeomorphism, \(\Gamma\) is uniquely determined by the numbers
of loops and leaves, where a \emph{leaf} is a vertex of degree \( 1 \).

\begin{defn}[Graphs with only one essential vertex]\label{defn:elementary}
  A bunch of grapes is called an \emph{elementary bunch of grapes},
  denoted by \(\Gamma_{\ell,m}\), if it is a connected graph with at most one essential vertex having \(\ell\) loops and \(m\) leaves; see the right
  picture of \Cref{figure:example of a bunch of grapes}.
\end{defn}

It is important to note that for any bunch of grapes, each \(1\)-bridge decomposition is along an edge that is not a loop.

\begin{remark}
Note that any graph in \(\Graph_2\) can be obtained from a bunch of grapes by replacing some edges by multiple edges.
\end{remark}

\subsection{Configuration spaces and Hilbert--Poincar\'e series}
Let \(X\) be a topological space. For each integer \(k\ge 0\), called
a \emph{braid index}, the \emph{ordered configuration space} \(F_kX\)
is the set of ordered \(k\)-tuples of distinct points in \( X \), and
the \emph{(unordered) configuration space} \(B_kX\) is the set of
subsets of \( X \) consisting of \(k\) distinct points, viewed as
topological spaces:
\begin{align*}
F_kX&=\{(x_1,\dots, x_k)\in X^k: x_i\neq x_j\mbox{ for all } i\neq j\},\\
B_kX&=\{\{x_1,\dots, x_k\}\subseteq X: x_i\neq x_j\mbox{ for all } i\neq j\}.
\end{align*}
We also denote the unions of all \(F_kX\)'s and \(B_kX\)'s by \(F_\star X\) and \(B_\star X\), respectively:
\[
F_\star X=\coprod_{k\ge 0}F_kX \quad\text{ and }\quad
B_\star X=\coprod_{k\ge 0}B_kX.
\]

Note that \(B_kX\) can be regarded as the quotient space of \(F_kX\) by the action of the symmetric group \(\bbS_k\) that permutes the coordinates in a canonical way.
Hence \(F_kX\) and \(B_kX\) are equipped with the subspace topology of the product topology on \(X^k\) and the quotient topology via the canonical quotient map \(F_kX\to B_kX\cong F_kX/\bbS_k\).

The following lemma is immediate from the definitions of \( F_kX \) and \( B_kX \).
\begin{lem}\label{lemma:facts}
Let \(X\), \( X_1 \), and \(X_2\) be topological spaces. Then the following hold.
\begin{enumerate}
\item For \(k\le 1\), the topological spaces \( F_kX \) and \( B_kX \) are given by
\[
F_0X=B_0X\cong\{*\}\quad\text{ and }\quad
F_1X=B_1X=X.
\]
\item If there is an embedding \(f:X_1\to X_2\), then for each \(k\ge 0\), there are induced maps
\[
F_kf:F_kX_1\to F_kX_2\quad\text{ and }\quad
B_kf:B_kX_1\to B_kX_2.
\]
\item There are homeomorphisms 
\[
B_k(X_1\sqcup X_2) \cong
\coprod_{\substack{k_1+k_2=k\\k_1,k_2\ge 0}} B_{k_1}X_1\times B_{k_2}X_2 \quad\text{ and }\quad
B_\star(X_1\sqcup X_2) \cong B_\star X_1 \times B_\star X_2.
\]
\end{enumerate}
\end{lem}

\begin{remark}
One may regard \(B_\star(-)\) as a monoidal functor 
\(B_\star(-):(\Emb,\sqcup) \to (\Emb,\times)\)
between monoidal categories of topological spaces with embedding, whose monoidal structures are given by the disjoint union \(\sqcup\) and the Cartesian product \(\times\).
\end{remark}

Let us fix a field \(\bbF\) and consider a singular chain complex \(C_\bullet(B_\star X)\) over \(\bbF\).
Then \(C_\bullet(B_\star X)\) and its homology group \(H_\bullet(B_\star X)\) are bigraded by the homological degree and the braid index, denoted by \(i\) and \(k\), respectively:
\[
C_\bullet(B_\star X) \cong \bigoplus_{i,k\ge 0} C_i(B_kX)\quad\text{ and }\quad
H_\bullet(B_\star X)\cong \bigoplus_{i,k\ge 0} H_i(B_kX).
\]

We define the \emph{Hilbert--Poincar\'e series} \(\cP_X(x,y)\) as the formal sum
\[
\cP_X(x,y)=\sum_{i,k\ge 0} \dim_\bbF H_i(B_kX)x^i y^k.
\]
Then, by \Cref{lemma:facts}(3), we obtain the following corollary.
\begin{cor}\label{corollary:disjoint union}
Let \(X_1\) and \(X_2\) be topological spaces. Then 
\[
\cP_{X_1\sqcup X_2}(x,y) = \cP_{X_1}(x,y) \cdot \cP_{X_2}(x,y).
\]
\end{cor}

\section{Edge stabilization and one-bridge decomposition}\label{edge stabilization}
From now on, we focus on the case when \(X\) is a graph
\( \Gamma=(V,E) \).

\subsection{Edge stabilization and the Hilbert polynomial}
Let \({\sfe}\) be an open edge of \(\Gamma\), which is
parameterized by the unit open interval \((0,1)\). Hence there is a
homeomorphism \(\phi_{{\sfe}}:(0,1)\to {\sfe}\subseteq \Gamma\). We
define a map \((-)_{{\sfe}}:B_\star\Gamma\to B_{\star}(0,1)\) as follows: for
each \(t\in(0,1)\) and \(\bfx\in B_{\star}\Gamma\),
\[
t\in\bfx_{{\sfe}}\in B_{\star}(0,1) \Longleftrightarrow \phi_{{\sfe}}(t)\in \bfx.
\]

Let \(\bft=\{t_1,\dots,t_k\}\in B_{\star}(0,1)\) for \(k\ge 0\), where
\(0<t_1<\dots<t_k<1\), and let \( t_0=0 \) and \( t_{k+1} \). The
\emph{stabilization} \(s\cdot(-):B_\star(0,1)\to B_{\star}(0,1)\) on
\((0,1)\) is the map defined by \(s\cdot\bft = \bft'\), where
\(\bft'=\{t_1',\dots,t_{k+1}'\}\) is given by
\( t_i' = (t_{i-1}+t_i)/2 \) for \( 1\le i\le k+1 \). Note that
\(\bft'\) has exactly one more point than \(\bft\).

The \emph{edge stabilization} \(\sfe\cdot(-):B_\star\Gamma\to B_{\star}\Gamma\) on \(\sfe\) is then the map defined by
\[
\sfe\cdot\bfx = \bfx\setminus \phi_{{\sfe}}(\bfx_{{\sfe}}) \cup \phi_{{\sfe}}(s\cdot\bfx_{{\sfe}}).
\]
Then the edge stabilization on \(\sfe\) induces a continuous map
\(\sfe\cdot(-):B_k\Gamma\to B_{k+1}\Gamma\) for each \(k\ge 0\).
Therefore we may regard the singular chain complex
\(C_\bullet(B_\star\Gamma)\) as an \(\bbF[E]\)-module\footnote{\(\bbF[E]\) is the multivariate polynomial ring generated by the elements in \(E\).}, which is in fact chain homotopy equivalent to a finitely generated \(\bbF[E]\)-module.

\begin{thm}\cite[Theorems~2.10 and 1.1]{ADCK2020}\label{thm:ADCK2020}
  Let \(\Gamma=(V,E)\) be a graph. Then the singular chain complex
  \(C_\bullet(B_\star\Gamma)\) is an \(\bbF[E]\)-module, and there is
  a finitely generated free \(\bbF[E]\)-module \(S_\bullet\Gamma\),
  called the \emph{\'Swi\k{a}tkowski complex} of \(\Gamma\), that is
  chain equivalent to \(C_\bullet(B_\star\Gamma)\). Therefore the
  homology group \(H_\bullet(B_\star\Gamma)\) is a finitely generated
  \(\bbF[E]\)-module.
\end{thm}

The \'Swi\k{a}tkowski complex \(S_\bullet(\Gamma)\) is a free \(\bbF[E]\)-module defined by
\begin{align}\label{eq:Swiatkowski}
\left(S_\bullet(\Gamma),\partial\right) &= \bigotimes_{\sfv\in V} \left(S_\sfv,\partial_\sfv\right),&
S_\sfv &= \bbF[E]\langle \{\varnothing_\sfv, \sfv\}\cup \mathcal{H}_\sfv\rangle,&
\partial_\sfv(\sfg) &= 
\begin{cases}
0 & \text{if }\sfg=\varnothing_\sfv\text{ or }\sfv;\\
\sfe(\sfh)\varnothing_\sfv-\sfv & \text{if }\sfg=\sfh\in\mathcal{H}_\sfv,
\end{cases}
\end{align}
where \(\sfe(\sfh)\) is the edge containing \(\sfh\) for each half
edge \(\sfh\in\mathcal{H}_\sfv\) attached to \(\sfv\). Using a pair
consisting of the homological degree and the braid index, the
bigrading is given as follows: for each \(\sfv\in V\),
\(\sfh\in\mathcal{H}_\sfv\), and \(\sfe\in E\),
\begin{align*}
|\varnothing_\sfv|&=(0,0),&
|\sfv|&=(0,1),&
|\sfh|&=(1,1),&
&\text{and}&
|\sfe|&=(1,0).
\end{align*}

Since \(S_\bullet(\Gamma)\) is finitely generated over the Noetherian
ring \(\bbF[E]\), its Betti number is eventually given by a polynomial
whose degree comes from a combinatorial invariant \(\Delta_\Gamma^i\)
of \(\Gamma\), called \emph{Ramos's invariant} \cite{Ramos2018, ADCK2020}, defined as follows:
For each \(W\subseteq V\),
\begin{equation}\label{eq:ramos}
\Delta_\Gamma^W = |\pi_0(\Gamma\setminus W)|, \quad\text{ and }\quad
\Delta_\Gamma^i = \max_{W\in\binom{V^{\ess}}i} \Delta_\Gamma^W,
\end{equation}
where \(V^{\ess}\subseteq V\) is the set of essential vertices and
\(\binom{X}i\) is the set of subsets of \(X\) consisting of \(i\)
elements. In other words, \(\Delta_\Gamma^i\) is the maximum number of
connected components that can be obtained from \( \Gamma \) by
removing \( i \) essential vertices. For the sake of convenience, we
define \(\Delta_\Gamma^i=-\infty\) for each \(i>|V^{\ess}|\).

\begin{cor}[{\cite[Theorem~A]{Ramos2018} for trees; \cite[Theorem~1.2]{ADCK2020} for general graphs}]\label{corollary:degree}
  Let \(\Gamma=(V,E)\) be a graph. Then for each \(i\ge0\), the
  \(i\)-th Betti number \(\dim_\bbF H_i(B_k\Gamma)\) over a field
  \(\bbF\) is eventually a polynomial in \(k\) of degree
  \(\Delta_\Gamma^i-1\). In particular, \(\dim_\bbF H_i(B_k\Gamma)=0\)
  for each \(i>|V^{\ess}|\).
\end{cor}

\begin{defn}[Hilbert polynomials]
Let \(\Gamma=(V,E)\) be a graph. For each \(i\ge 0\), the \emph{\(i\)-th Hilbert polynomial} \(P_\Gamma^i(x)\) of the \(\bbF[E]\)-module \(H_i(B\Gamma)\) is the polynomial satisfying
\begin{equation}\label{eq:9}
P^i_\Gamma(k) = \dim_\bbF H_i(B_k\Gamma)
\end{equation}
for all large enough \(k\).
\end{defn}

\begin{remark}
  There is a computational result
  \cite{GCDC2019} on the smallest integer \( k_0 \) such that \eqref{eq:9}
  holds for all \(k\ge k_0\) for some small graphs. However, the Hilbert polynomial is unknown in general.
\end{remark}

On the other hand, the leading coefficient of \(P_\Gamma^i(k)\) can be explicitly computed from \(\Gamma\) as follows:
\begin{thm}\label{thm:2}\cite[Theorem~1.1]{ADCK2022}
Let \(\Gamma=(V,E)\) be a nontrivial graph.
For each \(i\ge 0\), if \(\Delta_\Gamma^i>0\), then the leading coefficient \( c_\Gamma^i \) of \(P_\Gamma^i(k)\) is given by 
\[
c_\Gamma^i=\frac1{(\Delta_\Gamma^i-1)!}\sum_{\substack{W\in \binom{V^{\ess}}i\\ \Delta_\Gamma^W=\Delta_\Gamma^i}} \prod_{\sfv\in W}(\deg(\sfv)-2),
\]
where \(W\) ranges over the sets of essential vertices of cardinality \(i\) such that \(|\pi_0(\Gamma\setminus W)|=\Delta_\Gamma^i\).
\end{thm}

Since the asymptotic dimension of \(H_i(B_k\Gamma)\), as \(k\)
increases, is given by
\(c_\Gamma^i k^{\Delta_\Gamma^i-1}\), \Cref{thm:2} is
equivalent to
\[
  \lim_{k\to\infty} \frac{\dim_\bbF H_i(B_k\Gamma)}{k^{\Delta_\Gamma^i-1}} =c_\Gamma^i.
\]
Note also that, by \Cref{thm:2}, the coefficient
\(c_\Gamma^i\) counts the number of ways to choose a
half edge at each essential vertex in \(W\), avoiding two fixed
half edges. However, a combinatorial description of any coefficient in
\(P_\Gamma^i(k)\) other than the leading one is not yet known.

We define a quantity that measures the difference between the sequence of polynomials \(P_\Gamma^i(k)\) and the Betti numbers \(\dim_\bbF H_i(B_k\Gamma)\).
\begin{defn}[Residual polynomial]
Let \(\Gamma\) be a graph. The \emph{residual polynomial} \(\cR_\Gamma(x,y)\) is defined by
\[
\cR_\Gamma(x,y) = \cP_\Gamma(x,y) - \sum_{i,k\ge 0} P_\Gamma^i(k) x^i y^k.
\]
\end{defn}
By \Cref{corollary:degree}, the residual polynomial
\(\cR_\Gamma(x,y)\) is indeed a polynomial in \( x \) and \( y \).

\begin{exam}[The \(0\)-th Betti numbers]
Let \(\Gamma\) be a connected graph. Then \(\dim_\bbF H_0(B_k\Gamma)\equiv P_\Gamma^0(k)\equiv 1\) for all \(k\ge 0\).
\end{exam}

\begin{prop}[Elementary bunches of grapes]\label{prop:elementary}
  For \(\ell,m\ge 0\) with \(\ell+m\ge 1\) and \(i\ge 0\), the Hilbert
  polynomial \(P_{\Gamma_{\ell,m}}^i(k)\) and the residual polynomial
  \(\cR_{\Gamma_{\ell,m}}(x,y)\) for the elementary bunch of grapes
  \(\Gamma_{\ell,m}\) are given by
\begin{align}\label{equation:residual for elementary}
P^i_{\Gamma_{\ell,m}}(k)&=\begin{cases}
1 & \text{if }i=0;\\
\ell & \text{if }i=1, \ell+m=1;\\
N_{\ell,m}(k) & \text{if }i=1, \ell+m\ge 2;\\
0 & \text{if }i\ge 2,
\end{cases}&
\cR_{\Gamma_{\ell,m}}(x,y)&=\begin{cases}
-\ell & \text{if }\ell+m=1;\\
0 & \text{if }\ell+m\ge 2,
\end{cases}
\end{align}
where 
\begin{equation}\label{eq:first homology}
N_{\ell,m}(k)\coloneqq (2\ell+m-2)\binom {k+\ell+m-2}{\ell+m-1} - \binom{k+\ell+m-2}{\ell+m-2}+1.
\end{equation}
\end{prop}

\begin{proof}
  The second identity follows from the first identity. For the first
  identity, we consider the cases \( i=0 \), \( i=1 \), and
  \( i\ge2 \) as follows.

  If \( i=0 \), then \(\dim_\bbF H_0(B_k\Gamma_{\ell,m}) \equiv 1 \).

Let \( i=1 \). If \(\ell+m=1\), then \(\Gamma_{\ell,m}\) is
homeomorphic to either an interval when \((\ell,m)=(0,1)\) or a circle
when \((\ell,m)=(1,0)\). By the direct computation, we have
\(\dim_\bbF H_1(B_0\Gamma_{\ell,m})=0\) and
\(\dim_\bbF H_1(B_k\Gamma_{\ell,m})=\ell\) for \(k\ge 1\). Therefore
\(P_{\Gamma_{\ell,m}}^1(k)\equiv \ell\). If \(\ell+m\ge 2\), then it
is known by \cite{KP2012} that for any \(k\ge 0\),
\[
  \dim_\bbF H_1(B_k\Gamma_{\ell,m})=N_{\ell,m}(k).
\]

Now suppose \( i\ge2 \). For a connected graph \(\Gamma\) and
\(k\ge 0\), the configuration space \(B_k \Gamma\) is path-connected
and the chain complex \(C_\bullet(B_k\Gamma)\) is chain homotopy
equivalent to the \'Swi\k{a}tkowski complex \(S_\bullet\Gamma\) whose
degree is \(\min\{k,|V^\ess(\Gamma)|\}\) by \cite{ADCK2019}. Since
\(|V^\ess(\Gamma_{\ell,m})|\le 1\) for any \(\ell,m\ge 0\), we have
\[
\dim_\bbF H_i(B_k\Gamma_{\ell,m}) \equiv  0.\qedhere
\]
\end{proof}

\subsection{One-bridge decompositions and Betti numbers}
Suppose that \(\Gamma\) has an edge-\(m\)-cut
\(E_0=\{\sfe_1,\dots, \sfe_m\}\). Let \(\Gamma_1\) and \(\Gamma_2\) be
the two components in the \(m\)-bridge decomposition of \( \Gamma \)
along \(E_0\) as depicted in \Cref{figure:cutting}. Then by
\Cref{thm:ADCK2020} and \eqref{eq:Swiatkowski}, there are chain homotopy equivalences between
\(\bbF\)-modules
\[
S_\bullet\Gamma \simeq 
S_\bullet\Gamma_1\otimes_{\bbF[E_0]} S_\bullet\Gamma_2.
\]
If \(H_\bullet(B_\star\Gamma_1)\) or \(H_\bullet(B_\star\Gamma_2)\) is
flat over \(\bbF[E_0]\), then this yields a decomposition
\[
H_\bullet(B_\star\Gamma) \simeq H_\bullet(B_\star\Gamma_1)\otimes_{\bbF[E_0]} H_\bullet(B_\star\Gamma_2).
\]

For example, if \(m=0\), namely, \(\Gamma=\Gamma_1\sqcup\Gamma_2\),
then \(\bbF\) acts freely on both \(H_\bullet(B_\star\Gamma_1)\) and
\(H_\bullet(B_\star\Gamma_2)\), and we have the decomposition as
above, which is exactly the K\"unneth formula by
\Cref{lemma:facts}(3).

\begin{prop}\label{proposition:disjoint union}
Let \(\Gamma_1\) and \(\Gamma_2\) be graphs.
Suppose that \(\cR_{\Gamma_1}(x,y)=\cR_{\Gamma_2}(x,y)\equiv 0\).
Then \(\cR_{\Gamma_1\sqcup \Gamma_2}(x,y)\equiv0\) and moreover,
\[
P_{\Gamma_1\sqcup \Gamma_2}^i(k)=\sum_{\substack{i_1+i_2=i\\i_1,i_2\ge 0}}
\left(P_{\Gamma_1}^{i_1}*P_{\Gamma_2}^{i_2}\right)^+(k).
\]
\end{prop}
\begin{proof}
We have
\begin{align*}
\sum_{k\ge 0}\dim_\bbF H_i(B_k(\Gamma_1\sqcup\Gamma_2))y^k &=
\sum_{\substack{i_1+i_2=i\\i_1,i_2\ge 0}}
\sum_{\substack{k_1+k_2=k\\k_1,k_2\ge 0}}
\dim_\bbF H_{i_1}(B_{k_1}\Gamma_1)\dim_\bbF H_{i_2}(B_{k_2}\Gamma_2) y^k
\tag{\Cref{corollary:disjoint union}}
\\
&=
\sum_{\substack{i_1+i_2=i\\i_1,i_2\ge 0}}
\sum_{\substack{k_1+k_2=k\\k_1,k_2\ge 0}} P_{\Gamma_1}^{i_1}(k_1)P_{\Gamma_2}^{i_2}(k_2) y^k
\tag{\(\cR_{\Gamma_1}(x,y)=\cR_{\Gamma_2}(x,y)\equiv 0\)}
\\
&=\sum_{\substack{i_1+i_2=i\\i_1,i_2\ge 0}}
\left(\sum_{k\ge 0}P_{\Gamma_1}^{i_1}(k) y^k\right)\cdot
\left(\sum_{k\ge 0}P_{\Gamma_2}^{i_2}(k) y^k\right)
\\
&=
\sum_{\substack{i_1+i_2=i\\i_1,i_2\ge 0}}
\sum_{k\ge 0} \left(P_{\Gamma_1}^{i_1}*P_{\Gamma_2}^{i_2}\right)^+(k) y^k
\tag{\Cref{lemma:product}}\\
&=
\sum_{k\ge 0}
\left(\sum_{\substack{i_1+i_2=i\\i_1,i_2\ge 0}}
\left(P_{\Gamma_1}^{i_1}*P_{\Gamma_2}^{i_2}\right)^+(k)\right) y^k.
\end{align*}
This implies both of the assertions.
\end{proof}

Let us consider a one-bridge decomposition. We first recall the following result about stabilization on a single edge.
\begin{prop}\cite[Proposition~5.21]{ADCK2019}\label{proposition:free action of leaves}
Let \(\Gamma=(V,E)\) be a graph. For any edge \(\sfe\in E\), the action \(\sfe\cdot(-):H_\bullet(B_\star\Gamma)\to H_\bullet(B_{\star}\Gamma)\) is injective.
\end{prop}

Therefore, for a one-bridge decomposition, the following holds.
\begin{cor}\label{corollary:1-bridge}
  Let \(\Gamma\) be a graph with an edge-one-cut \(\{\sfe\}\), and let
  \(\Gamma_1\) and \(\Gamma_2\) be the two components of the
  one-bridge decomposition along \(\sfe\). Then there is an isomorphism between \(\bbF\)-modules
\[
H_\bullet(B_\star\Gamma)\cong H_\bullet(B_\star\Gamma_1)\otimes_{\bbF[\sfe]} H_\bullet(B_\star\Gamma_2).
\]
In particular, for all \(i,k\ge0\), we have
\[
\dim_\bbF H_i(B_k\Gamma)
=\dim_\bbF H_i(B_k(\Gamma_1\sqcup \Gamma_2)) - \dim_\bbF H_i(B_{k-1}(\Gamma_1\sqcup \Gamma_2)),
\]
where \( \dim_\bbF H_i(B_{-1}(\Gamma_1\sqcup \Gamma_2))=0 \).
\end{cor}
\begin{proof}
  Since \(\bbF[\sfe]\) is a PID, we can express
  \(H_\bullet(B_\star\Gamma_1)\) as a direct sum of a free module and
  a torsion module over \(\bbF[\sfe]\). Since the action of
  \(\sfe\) is injective by \Cref{proposition:free action of leaves}
  and it increases the number \(k\) of points by \(1\), there are no
  \(\bbF[\sfe]\)-torsion elements in \(H_\bullet(B_\star\Gamma_1)\).
  Therefore \(H_\bullet(B_\star\Gamma_1)\) is a free
  \(\bbF[\sfe]\)-module, and so the first assertion holds.

Note that since the action of \(\bbF[\sfe]\) preserves the homological degree,
\begin{equation}\label{eq:decomp}
H_i(B_\star\Gamma)\cong \bigoplus_{\substack{i_1+i_2=i\\i_1,i_2\ge 0}} H_{i_1}(B_\star\Gamma_1) \otimes_{\bbF[\sfe]} H_{i_2}(B_\star\Gamma_2).
\end{equation}
On the other hand, the canonical inclusion
\(\Gamma_1\sqcup\Gamma_2\to\Gamma\) induces the top horizontal map in
the following commutative diagram:
\[
\begin{tikzcd}
H_i(B_\star(\Gamma_1\sqcup\Gamma_2))\ar[r] \ar[d,"\cong"']& H_i(B_\star\Gamma)\ar[d,"\cong"]\\
\displaystyle\bigoplus_{\substack{i_1+i_2=i\\i_1,i_2\ge 0}} H_{i_1}(B_\star\Gamma_1)\otimes_\bbF H_{i_2}(B_\star\Gamma_2) \ar[r] &
\displaystyle\bigoplus_{\substack{i_1+i_2=i\\i_1,i_2\ge 0}} H_{i_1}(B_\star\Gamma_1)\otimes_{\bbF[\sfe]} H_{i_2}(B_\star\Gamma_2).
\end{tikzcd}
\]
Here, the vertical maps are given by \Cref{lemma:facts}(3) and
\Cref{eq:decomp}, and the bottom map is the canonical projection on each summand.

By definition of a tensor product, the kernel of the bottom map on
each direct summand is generated by the elements of the form
\(\sfe\cdot\alpha\otimes\beta - \alpha\otimes\sfe\cdot\beta\) for some
\(k_1',k_2'\ge 0\), \(\alpha\in H_{i_1}(B_{k_1'}\Gamma_1)\), and
\( \beta\in H_{i_2}(B_{k_2'}\Gamma_2) \). Moreover, if its total braid
index is \(k\), then \(k_1'+k_2'=k-1\). Therefore,
\[
\ker\left(H_i(B_k(\Gamma_1\sqcup\Gamma_2))\to H_i(B_k\Gamma)\right)\cong \begin{cases}
0 & \text{if }k=0;\\
H_i(B_{k-1}(\Gamma_1\sqcup\Gamma_2)) & \text{if }k\ge 1,
\end{cases}
\]
and we are done.
\end{proof}

\begin{prop}\label{proposition:1-bridge}
Let \(\Gamma\), \( \Gamma_1 \), \( \Gamma_2 \), and \(\sfe\) be as in \Cref{corollary:1-bridge}.
Suppose that \(\cR_{\Gamma_1}(x,y)=\cR_{\Gamma_2}(x,y)\equiv 0\).
Then 
\[
\cR_\Gamma(x,y)\equiv 0\quad\text{ and }\quad
P_{\Gamma}^i(k)=\sum_{\substack{i_1+i_2=i\\i_1,i_2\ge 0}}
\left(P_{\Gamma_1}^{i_1}*P_{\Gamma_2}^{i_2}\right)(k).
\]
\end{prop}

\begin{proof}
By \Cref{corollary:1-bridge} and the assumption,
\begin{align*}
\sum_{k\ge 0}\dim_\bbF H_i(B_k\Gamma) y^k
&=\sum_{k\ge 0}\left(\dim_\bbF H_i(B_k(\Gamma_1\sqcup \Gamma_2)) - \dim_\bbF H_i(B_{k-1}(\Gamma_1\sqcup \Gamma_2))\right)y^k
\tag{\Cref{corollary:1-bridge}}
\\
&=
  \sum_{k\ge 0} P_{\Gamma_1\sqcup \Gamma_2}^i(k) y^k
  - \sum_{k\ge 1}  P_{\Gamma_1\sqcup \Gamma_2}^i(k-1) y^k
\tag{\Cref{proposition:disjoint union}}
\\
&=
\sum_{k\ge 0}
\left(\sum_{\substack{i_1+i_2=i\\i_1,i_2\ge 0}}
\left(P_{\Gamma_1}^{i_1}*P_{\Gamma_2}^{i_2}\right)^+(k)\right) y^k
- \sum_{k\ge 1}
\left(\sum_{\substack{i_1+i_2=i\\i_1,i_2\ge 0}}
\left(P_{\Gamma_1}^{i_1}*P_{\Gamma_2}^{i_2}\right)^+(k-1)\right) y^k
\tag{\Cref{proposition:disjoint union}}.
\end{align*}
Note that for any polynomial \( p(k) \), we have \( p^+(0) = p(0) \) and for \( k\ge1 \),
\[
  p^+(k) - p^+(k-1) = (p(0) + \cdots + p(k)) - (p(0) + \cdots + p(k-1)) = p(k).
\]
Hence, we obtain
\begin{align*}
\sum_{k\ge 0}\dim_\bbF H_i(B_k\Gamma) y^k
&=
\sum_{k\ge 0}
\left(
\sum_{\substack{i_1+i_2=i\\i_1,i_2\ge 0}}
\left(P_{\Gamma_1}^{i_1}*P_{\Gamma_2}^{i_2}\right)(k)\right)y^k,
\end{align*}
which implies both of the assertions.
\end{proof}

\section{The Hilbert polynomials for bunches of grapes}
\label{sec:hilbert polynomials}

Let \( \Gamma \) be a bunch of grapes. The \emph{stem} of \( \Gamma \)
is the subgraph \( \sfT \) obtained from \( \Gamma \) by removing all
cycles except their endpoints. Note that \(\Gamma\) can be recovered
up to homeomorphism\footnote{Indeed, up to subdivision on cycles.} by
attaching a certain number of loops at each vertex of \(\sfT\).
Therefore, \(\Gamma\) is completely characterized by the pair
\((\sfT, \loops)\), where \(\sfT\) is a tree called the \emph{stem} of
\(\Gamma\), and \(\loops:V(\sfT)\to\mathbb{Z}_{\ge 0}\) is a
nonnegative integer-valued function on the vertices of \(\sfT\), which
determines the graph obtained by attaching \(\loops(\sfv)\) loops,
called \emph{grapes}, at each vertex \(\sfv\) of \(\sfT\). We say that
\(\Gamma\) is \emph{of type} \((\sfT,\loops)\) and simply denote
\(\Gamma\cong(\sfT,\loops)\).

For example, the elementary bunch of grapes \(\Gamma_{\ell,m}\) in
\Cref{defn:elementary} is given by
\begin{equation}\label{eq:elementary}
\Gamma_{\ell,m}\cong (\Gamma_{0,m},f),\qquad
f(\sfv) = \begin{cases}
\ell & \text{if }\sfv\text{ is the essential vertex of }\Gamma_{0,m};\\
0 & \text{otherwise}.
\end{cases}
\end{equation}
The bunch of grapes \(\Gamma\) in the left of \Cref{figure:example of
  a bunch of grapes} is of type \((\sfT, \loops)\), as described in
\Cref{figure:stem and number of grapes}.
\begin{figure}
\[
\sfT=\begin{tikzpicture}[baseline=-.5ex]
\draw[thick,fill] (0,0) circle (2pt) node (A) {} node[above] {\(\mathsf{1}\)} -- ++(5:2) circle(2pt) node(B) {} node[above left] {\(\mathsf{0}\)} -- +(0,1.5) circle (2pt) node (C) {} node[left] {\(\mathsf{13}\)} +(0,0) -- +(0,-1.5) circle (2pt) node(D) {} node[left] {\(\mathsf{2}\)} +(0,0) -- ++(-10:1.5) circle(2pt) node(E) {} node[above] {\(\mathsf{3}\)} -- ++(10:1.5) circle(2pt) node(F) {} node[below] {\(\mathsf{7}\)} -- +(120:1) circle(2pt) node(G) {} node[above] {\(\mathsf{12}\)} +(0,0) -- +(60:1) circle (2pt) node(H) {} node[above] {\(\mathsf{11}\)} +(0,0) -- ++(-5:1) circle (2pt) node(I) {} node[below] {\(\mathsf{8}\)} -- +(30:1) circle(2pt) node(J) {} node[right] {\(\mathsf{10}\)} +(0,0) -- +(-30:1) circle (2pt) node(K) {} node[right] {\(\mathsf{9}\)};
\draw[thick, fill] (E.center) -- ++(0,-1) circle (2pt) node[left] {\(\mathsf{4}\)} -- +(-60:1) circle (2pt) node[right] {\(\mathsf{6}\)} +(0,0) -- +(-120:1) circle (2pt) node[left] {\(\mathsf{5}\)};
\end{tikzpicture}\qquad
\loops(\sfv)=\begin{cases}
0 & \mbox{if }\sfv\in \{\mathsf{0}, \mathsf{3}, \mathsf{4}, \mathsf{5}, \mathsf{6}, \mathsf{8}, \mathsf{10}\};\\
1 & \mbox{if }\sfv\in \{\mathsf{1}, \mathsf{2}, \mathsf{7}, \mathsf{9}, \mathsf{11}, \mathsf{12}\};\\
3 & \mbox{if }\sfv\in\{\mathsf{13}\}.
\end{cases}
\]
\caption{A stem \(\sfT\) with a root \(\sfv_0=\mathsf{0}\) and a function \(\loops\).}
\label{figure:stem and number of grapes}
\end{figure}

Let \(\Gamma\cong(\sfT,\loops)\in\Grape\) be a nontrivial bunch of
grapes. One important observation is that the set \(\Grape\) is closed
under \(1\)-bridge decomposition along any edge of a stem. That is,
for any \(\Gamma\in\Grape\) and an edge \(\sfe\) in the stem of
\(\Gamma\), the graphs \(\Gamma_1\) and \(\Gamma_2\) obtained by
cutting \(\Gamma\) along \(\sfe\) are contained in \(\Grape\) as well.

Note that both \(\Gamma_1\) and \(\Gamma_2\) are nontrivial if and
only if \(\sfe\) joins two essential vertices in \(\Gamma\).
Therefore, by taking the \(1\)-bridge decomposition of \(\Gamma\)
along every edge in its stem joining two essential vertices, we
obtain the \emph{decomposition of \(\Gamma\) along the stem} as seen
in \Cref{figure:decomposion along the stem}, where each component is
elementary and nontrivial.

\begin{figure}
\[
\begin{tikzpicture}[baseline=-.5ex]
\draw[thick,fill] (0,0) circle (2pt) node (A) {} node[above] {\(\mathsf{1}\)} -- ++(5:1) circle (2pt);
\grape[180]{A};
\begin{scope}[xshift=0.2cm]
	\draw[thick,fill] (0,0) +(5:1) circle (2pt) -- ++(5:2) circle(2pt) node(B) {} node[above left] {\(\mathsf{0}\)} -- +(0,0.75) circle (2pt) +(0,0) -- +(0,-0.75) circle (2pt) +(0,0) -- ++(-10:0.75) circle (2pt);

	\begin{scope}[yshift=0.2cm]
		\draw[thick,fill] (0,0) ++(5:2) +(0,0.75) circle (2pt) -- +(0,1.5) circle (2pt) node (C) {} node[below=2ex, left=-.5ex] {\(\mathsf{13}\)};
		\grape[0]{C}; \grape[90]{C}; \grape[180]{C};
	\end{scope}

	\begin{scope}[yshift=-0.2cm]
		\draw[thick,fill] (0,0) ++(5:2) +(0,-0.75) circle (2pt) -- +(0,-1.5) circle (2pt) node (D) {} node[left] {\(\mathsf{2}\)};
		\grape[-90]{D};
	\end{scope}

	\begin{scope}[xshift=0.2cm]
		\draw[thick,fill] (0,0) ++(5:2) ++(-10:0.75) circle (2pt) -- ++(-10:0.75) circle(2pt) node(E) {} node[above] {\(\mathsf{3}\)} -- +(10:0.75) circle (2pt) +(0,0) -- +(0,-0.5) circle (2pt);

		\begin{scope}[yshift=-0.2cm]
			\draw[thick,fill] (0,0) ++(5:2) ++(-10:1.5) ++(0,-0.5) circle (2pt) -- ++(0,-0.5) circle (2pt) node[left] {\(\mathsf{4}\)} -- +(-60:0.5) +(0,0) -- +(-120:0.5);

			\begin{scope}[yshift=0cm]
				\draw[thick,fill] (0,0) ++(5:2) ++(-10:1.5) ++(0,-1) +(-120:0.5) -- +(-120:1) circle (2pt) node[left] {\(\mathsf{5}\)} +(-60:0.5) -- +(-60:1) circle (2pt) node[right] {\(\mathsf{6}\)};
			\end{scope}
		\end{scope}

		\begin{scope}[xshift=0.2cm]
			\draw[thick,fill] (0,0) ++(5:2) ++(-10:1.5) ++(10:0.75) circle(2pt) -- ++(10:0.75) circle(2pt) node(F) {} node[below=0.5ex] {\(\mathsf{7}\)} -- +(120:0.5) circle (2pt) +(0,0) -- +(60:0.5) circle (2pt) +(0,0) -- +(-5:0.5) circle (2pt);
			\grape[-90]{F};

			\begin{scope}[xshift=-0.2cm, yshift=0.2cm]
				\draw[thick,fill] (0,0) ++(5:2) ++(-10:1.5) ++(10:1.5) +(120:0.5) circle (2pt) -- +(120:1) circle (2pt) node(G) {} node[below left] {\(\mathsf{12}\)};
				\grape[120]{G};
			\end{scope}
			\begin{scope}[xshift=0.2cm, yshift=0.2cm]
				\draw[thick,fill] (0,0) ++(5:2) ++(-10:1.5) ++(10:1.5) +(60:0.5) circle (2pt) -- +(60:1) circle (2pt) node(H) {} node[below right] {\(\mathsf{11}\)};
				\grape[60]{H};
			\end{scope}

			\begin{scope}[xshift=0.2cm]
				\draw[thick,fill] (0,0) ++(5:2) ++(-10:1.5) ++(10:1.5) ++(-5:0.5) circle (2pt) -- ++(-5:0.5) circle (2pt) node(I) {} node[below] {\(\mathsf{8}\)} -- +(30:0.5) +(0,0) -- +(-30:0.5) circle (2pt);

				\begin{scope}[xshift=0cm]
					\draw[thick,fill] (0,0) ++(5:2) ++(-10:1.5) ++(10:1.5) ++(-5:1) +(30:0.5) -- +(30:1) circle (2pt) node[right] {\(\mathsf{10}\)};
				\end{scope}
				\begin{scope}[xshift=0.2cm,yshift=-.2cm]
					\draw[thick,fill] (0,0) ++(5:2) ++(-10:1.5) ++(10:1.5) ++(-5:1) +(-30:0.5) circle (2pt) -- +(-30:1) circle (2pt) node(K) {} node[above] {\(\mathsf{9}\)};
					\grape[-30]{K};
				\end{scope}
			\end{scope}
		\end{scope}
	\end{scope}
\end{scope}

\end{tikzpicture}
\]
\caption{The decomposition of \(\Gamma\) along the stem.}
\label{figure:decomposion along the stem}
\end{figure}

Hence, each component of the decomposition of \(\Gamma\) contains a
unique essential vertex \(\sfv\) of the stem \(\sfT\), and we denote
it by \(\Gamma_{\sfv}\), which is called the \emph{local graph near}
\(\sfv\). Then for each essential vertex \(\sfv\) in the stem
\(\sfT\), the homeomorphic type of the local graph \(\Gamma_\sfv\) is
given as \(\Gamma_{\loops(\sfv),m(\sfv)}\), where \(m(\sfv)\) is the
degree of \(\sfv\) in \(\sfT\).
\begin{remark}\label{rmk:degree}
Since \(\sfv\) is essential, we have \(2\ell(\sfv)+m(\sfv)\ge 3\), which implies that \(\ell(\sfv)+m(\sfv)\ge 2\) as well.
\end{remark}
For the example in \Cref{figure:decomposion
  along the stem}, the local graphs are as follows:
\begin{align*}
\Gamma_{\mathsf{0}} &= \Gamma_{0,4}, &
\Gamma_{\mathsf{1}} &= \Gamma_{\mathsf{2}}=\Gamma_{\mathsf{9}}=\Gamma_{\mathsf{11}}=\Gamma_{\mathsf{12}}=\Gamma_{1,1},&
\Gamma_{\mathsf{3}} &= \Gamma_{\mathsf{4}}=\Gamma_{\mathsf{8}}=\Gamma_{0,3},& &\text{and}&
\Gamma_{\mathsf{7}} &= \Gamma_{1,4}, &
\Gamma_{\mathsf{13}} &= \Gamma_{3,1}.
\end{align*}

\begin{remark}
  For any graph \(\Gamma\), the local graph \(\Gamma_\sfv\) near a
  vertex \(\sfv\) can be defined as follows. Considering \( \Gamma \)
  as a a topological space, \(\Gamma_\sfv\) is defined to be the
  mapping cone of
  \(f_\sfv:\mathcal{H}_\sfv\to \pi_0(\Gamma\setminus\{\sfv\})\), where
  \(\mathcal{H}_\sfv\) is the set of half edges adjacent to \(\sfv\)
  and \(f_\sfv(\sfh) = [\sfh\setminus\{\sfv\}]\) is the connected
  component of \(\sfh\setminus\{\sfv\}\) in
  \(\Gamma\setminus\{\sfv\}\).
\end{remark}

\begin{thm}\label{theorem:decomposition formula}
For each nontrivial bunch of grapes \(\Gamma\) and \(i\ge 0\),
we have \(\cR_\Gamma(x,y)\equiv0\) and 
\[
P^i_\Gamma(k) = 
\sum_{W\in \binom{V^{\ess}(\Gamma)}i} \left(\conv_{\sfv\in W} P_{\Gamma_\sfv}^1\right)(k).
\]
\end{thm}
\begin{proof}
We may assume that \(\Gamma\) has no vertices of degree 2 by smoothing them if necessary.
We will use induction on the number \(N\) of essential vertices in \(\Gamma\), noting that the case when \(N=1\) is done by \Cref{prop:elementary}.

Suppose that the assertion holds for \(N\), and assume that \(\Gamma\)
has \(N+1\) essential vertices. Let \(\sfe\) be a non-loop edge in \(\sfT\) joining two essential vertices of \(\Gamma\).
Let \(\Gamma_1\) and \(\Gamma_2\) be the components in the
\(1\)-bridge decomposition of \( \Gamma \) along \(\sfe\).
Then both \(\Gamma_1\) and \(\Gamma_2\) are nontrivial bunches of grapes, and by the induction hypothesis, we have \(\cR_{\Gamma_1}(x,y)\equiv\cR_{\Gamma_2}(x,y)\equiv 0\).
Hence by \Cref{proposition:1-bridge},
for each \(i\ge 0\),
\[
\cR_\Gamma(x,y)\equiv 0\quad\text{ and }\quad
P_{\Gamma}^i(k)=\sum_{\substack{i_1+i_2=i\\i_1,i_2\ge 0}}
\left(P_{\Gamma_1}^{i_1}*P_{\Gamma_2}^{i_2}\right)(k).
\]
Therefore, 
\begin{align*}
P_{\Gamma}^i(k)
&=\sum_{\substack{i_1+i_2=i\\i_1,i_2\ge 0}}
\left(
\sum_{W_1\in\binom{V^{\ess}(\Gamma_1)}{i_1}} \left(\conv_{\sfv\in W_1} P_{\Gamma_\sfv}^1\right) * 
\sum_{W_2\in\binom{V^{\ess}(\Gamma_2)}{i_2}} \left(\conv_{\sfv\in W_2} P_{\Gamma_\sfv}^1\right)
\right)(k)\tag{Induction hypothesis}\\
&=\sum_{\substack{i_1+i_2=i\\i_1,i_2\ge 0}}
\sum_{W_1\in\binom{V^{\ess}(\Gamma_1)}{i_1}}
\sum_{W_2\in\binom{V^{\ess}(\Gamma_2)}{i_2}}
\left(
\left(\conv_{\sfv\in W_1} P_{\Gamma_\sfv}^1\right) * 
\left(\conv_{\sfv\in W_2} P_{\Gamma_\sfv}^1\right)
\right)(k)\tag{\(*\) is distributive}\\
&=\sum_{\substack{i_1+i_2=i\\i_1,i_2\ge 0}}
\sum_{W_1\in\binom{V^{\ess}(\Gamma_1)}{i_1}}
\sum_{W_2\in\binom{V^{\ess}(\Gamma_2)}{i_2}}
\left(\conv_{\sfv\in W_1\sqcup W_2} P_{\Gamma_\sfv}^1\right)(k)\tag{\(*\) is commutative and associative}\\
&=\sum_{W\in\binom{V^{\ess}(\Gamma)}{i}}
\left(\conv_{\sfv\in W} P_{\Gamma_\sfv}^1\right)(k).
\tag{\(V^\ess(\Gamma)=V^\ess(\Gamma_1)\sqcup V^\ess(\Gamma_2)\)}
\end{align*}
Hence, the statement also holds for \( N+1 \), completing the proof.
\end{proof}

Let us denote the coefficient of \(\binom k j\) in \(P_{\Gamma_\sfv}^1(k)\) by \(b_{\sfv;j}^1 = b_{\ell(\sfv),m(\sfv);j}^1\),
that is,
\[
  P_{\Gamma_\sfv}^1(k) = \sum_{j\ge0}  b_{\sfv;j}^1 \binom{k}{j},
\]
where 
\begin{equation}\label{eq:b}
b_{\sfv;j}^1=\begin{cases}
(2\ell(\sfv)+m(\sfv)-2)\binom{\ell(\sfv)+m(\sfv)-2}{\ell(\sfv)+m(\sfv)-1-j}-\binom{\ell(\sfv)+m(\sfv)-2}{\ell(\sfv)+m(\sfv)-2-j} & \text{if }j\ge 1;\\
0 & \text{if }j=0.
\end{cases}
\end{equation}

One can express the coefficients of \(P_\Gamma^i(k)\) for each \(i\ge 0\) in terms of \(b_{\sfv,j}^i\)'s as follows:

\begin{cor}\label{cor:1}
For each nontrivial bunch of grapes \(\Gamma\) and \(i\ge 0\),
\[
P_{\Gamma}^i(k) = \sum_{j\ge 0}
\left(
\sum_{W\in\binom{V^{\ess}(\Gamma)}{i}}
\sum_{\substack{\vec j \in \bbZ_{\ge0}^W\\ \norm{\vec j}= j}} \prod_{\sfv\in W} b^1_{\sfv;j_\sfv}\right) \binom k j.
\]
\end{cor}
\begin{proof}
This is a direct consequence of \Cref{theorem:decomposition formula}, \eqref{equation:product}, and \Cref{prop:elementary} since \(\ell(\sfv)+m(\sfv)\ge 2\) as observed in \Cref{rmk:degree}.
\end{proof}

\begin{remark}
  When \(i=0\), then \(W=\varnothing\) and there is a unique
  \(\vec j \in \bbZ_{\ge0}^W\) so that \(\norm{\vec j}=0\). Since the
  empty product \(\prod_{\sfv\in W}b_{\sfv;j_\sfv}^1\) is \( 1 \), we
  have \(P^i_\Gamma(k) = 1\) for any \(k\ge0\).
\end{remark}

In \Cref{sec:comb-descr-hilb}, we will give a combinatorial
description of all coefficients of \(P_{\Gamma}^i(k)\) using
\Cref{cor:1}. As a consequence of \Cref{cor:1}, we recover
\Cref{corollary:degree} and \Cref{thm:2} for the case when
\( \Gamma \) is a bunch of grapes.

\begin{cor}\cite{ADCK2020,ADCK2022}
  For each nontrivial bunch of grapes \(\Gamma\) and \(i\ge 0\), the
  leading term of \(P_\Gamma^i(k)\) is given by
  \(c_\Gamma^i k^{\Delta_\Gamma^i-1}\), where \(\Delta_\Gamma^i\) and
  \(c_\Gamma^i\) are given in \Cref{eq:ramos,thm:2}.
\end{cor}

\begin{proof}
  We will prove the following equivalent statement: the
  \( b \)-leading term of \(P_\Gamma^i(k)\) is equal to
\begin{equation}\label{eq:11}
(\Delta_\Gamma^i-1)! c_\Gamma^i \binom k {\Delta_\Gamma^i-1}=
\left(\sum_{\substack{W\in \binom{V^{\ess}}i\\ \Delta_\Gamma^W=\Delta_\Gamma^i}} \prod_{\sfv\in W}(\deg(\sfv)-2)\right)\binom k {\Delta_\Gamma^i-1},
\end{equation}

For any vertex \(\sfv\in V^\ess(\Gamma)\), by \eqref{eq:b} and the
fact that \(\deg \sfv = 2\ell(\sfv)+m(\sfv)\), we have
\(b_{\sfv,j}^1=\deg \sfv - 2 \) if \(j=\ell(\sfv)+m(\sfv)-1\) and
\(b_{\sfv,j}^1=0\) if \(j\ge \ell(\sfv)+m(\sfv)\). Therefore, by
\Cref{cor:1}, each \(W\in \binom{V^\ess(\Gamma)}{i}\) contributes to
\( P_\Gamma^i(k) \) a polynomial of degree \( d_W \) whose
\( b \)-leading coefficient is \(\prod_{\sfv\in W}(\deg \sfv -2)>0\),
where
\[
  d_W := \sum_{\sfv\in W}(\ell(\sfv)+m(\sfv)-1).
\]
Since every essential vertex \(\sfv\) separates \(\Gamma\) into
\(\ell(\sfv)+m(\sfv)\) components, we have
\[
 d_W= \sum_{\sfv\in W}(\ell(\sfv)+m(\sfv)-1) = |\pi_0(\Gamma\setminus W)| -1 = \Delta_\Gamma^W-1.
\]

Hence,
\[
\deg P_\Gamma^i(k) = \max_{W\in\binom{V^\ess(\Gamma)}{i}}(\Delta_\Gamma^W-1) = \Delta_\Gamma^i-1,
\]
and the \( b \)-leading coefficient of \(P_\Gamma^i(k)\) is equal to
\[
\sum_{\substack{W\in \binom{V^{\ess}}i\\ \Delta_\Gamma^W=\Delta_\Gamma^i}} \prod_{\sfv\in W}(\deg(\sfv)-2).
\]
Therefore the \( b \)-leading term of \(P_\Gamma^i(k)\)
is equal to the right-hand side of \eqref{eq:11}, as desired.
\end{proof}

Observe that \Cref{cor:1} implies that the Hilbert polynomials
\(P_{\Gamma}^i(k)\) of \( \Gamma \) are independent of how the
elementary pieces of \(\Gamma\) are connected. Therefore, one may
decompose \( \Gamma \) along all edge-\(1\)-cuts and reconstruct it in
an arbitrary way without changing the sequence
\(\left(P^i_\Gamma(k)\right)_{i\ge0}\). Hence, the Hilbert polynomials
capture only the local data of \(\Gamma\), which can be defined
rigorously as follows: for \(\Gamma\cong(\sfT,\loops)\), the
\emph{local data} \(\cD(\Gamma)\) of \(\Gamma\) is the multiset of
pairs given by
\[
\cD(\Gamma)=\{(\loops(\sfv),m(\sfv)): \sfv\in V^{\ess}(\Gamma)\}.
\]
Then, for two bunches of grapes \(\Gamma_1\) and \(\Gamma_2\)
with \(\cD(\Gamma_1)=\cD(\Gamma_2)\), we have
\(P_{\Gamma_1}^i(k)\equiv P_{\Gamma_2}^i(k)\) for each \(i\ge 0\).
The following theorem states that the converse is also true.

\begin{thm}\label{thm:local data}
For \(\Gamma_1, \Gamma_2\in\Grape\), we have
\[
  \cD(\Gamma_1)=\cD(\Gamma_2)
  \quad \mbox{if and only if} \quad 
P_{\Gamma_1}^i(k)\equiv P_{\Gamma_2}^i(k) \quad\mbox{ for all } i\ge 0.
\]
\end{thm}
\begin{proof}
  The ``only if'' part has already been established. For the ``if''
  part, let \(\Gamma\in\Grape\). We need to show that
  \( \cD(\Gamma) \) is uniquely determined by the polynomials
  \( P_{\Gamma}^i(k)\), \( i\ge0 \). Recall from
  \Cref{equation:product}, the \( b \)-leading coefficient of
  \( (p_1 * p_2)(k) \) is the product of the \( b \)-leading
  coefficients of \( p_1(k) \) and \( p_2(k) \).

  Let \( P_1(k) = P_{\Gamma}^1(k) \) and \(d_1 = \deg P_1(k)\). Then,
  by \Cref{theorem:decomposition formula},
\begin{equation}\label{eq:6}
P_1(k) = P_\Gamma^1(k)= \sum_{\sfv\in V^\ess(\Gamma)} P_{\Gamma_\sfv}^1(k).
\end{equation}
By \eqref{eq:first homology}, for any essential vertex \( \sfv \), we
have \(\deg P_{\Gamma_\sfv}^1(k) = \loops(\sfv)+m(\sfv)-1\) and the
\( b \)-leading coefficient of \(P_{\Gamma_\sfv}^1(k)\) is
\( 2\loops(\sfv)+m(\sfv)-2 \). Therefore,
\[
d_1 = \max\{\deg P_{\Gamma_\sfv}^1(k):\sfv\in V^\ess(\Gamma)\}=\max\{\loops(\sfv)+m(\sfv)-1: \sfv\in V^\ess(\Gamma)\}.
\]

Define \(n_1=\max\{i: \deg P_{\Gamma}^i(k)=i\cdot d_1\}\). Then, by \Cref{theorem:decomposition formula} and
 the definitions of \(d_1\) and
\(n_1\), there are exactly \(n_1\) essential vertices, say
\(\sfv_1,\dots, \sfv_{n_1}\), such that
\[
  \deg P_{\Gamma_{\sfv_j}}^1(k) = d_1 = \loops(\sfv_j)+m(\sfv_j)-1.
\]
Since the \( b \)-leading coefficient of \( P_{\Gamma_{\sfv_j}}^1(k) \) is
\( 2\loops(\sfv_j)+m(\sfv_j)-2=\loops(\sfv_j)-1+d_1 \), by
  \Cref{theorem:decomposition formula} again, for each
  \(1\le i\le n_1\), the \( b \)-leading coefficient \( s_i \) of
  \(P^i_{\Gamma}(k)\) is given by
\[
  s_i= \sum_{1\le j_1<\dots<j_i\le n_1} (\loops(\sfv_{j_1})-1+d_1)\cdots(\loops(\sfv_{j_i})-1+d_1).
\]
Thus \( \loops(\sfv_{1})-1+d_1,\dots,\loops(\sfv_{n_1})-1+d_1 \) are the roots of
\[
  x^{n_1}- s_1 x^{n_1-1}+  s_2 x^{n_1-2} -\cdots+(-1)^{n_1}  s_{n_1}=0.
\]
This implies that the local data at \(\sfv_1,\dots, \sfv_{n_1}\) are uniquely determined by \( P_{\Gamma}^1(k) \).

Now let \( P_2(k) = P_1(k) - \sum_{i=1}^{n_1} P_{\Gamma_{\sfv_i}}^1(k) \) and
\(d_2=\deg P_2(k)\). By \eqref{eq:6},
\[
  P_2(k) = \sum_{\sfv\in V^\ess(\Gamma)\setminus\{\sfv_1,\dots,\sfv_{n_1}\}} P_{\Gamma_\sfv}^1(k).
\]
Let \(n_2=\max\{i: \deg P_{\Gamma}^{n_1+i}(k)=n_1d_1+id_2\}\). Then
there are \(n_2\) essential vertices \( \sfv'_1,\dots,\sfv'_{n_2} \)
in \( V^\ess(\Gamma)\setminus\{\sfv_1,\dots,\sfv_{n_1}\} \)
such that
\[
  \deg P_{\Gamma_{\sfv'_j}}^1(k) = d_2 = \loops(\sfv'_j)+m(\sfv'_j)-1.
\]
By the similar argument as above, the local data for these vertices
are uniquely determined by \( P_{\Gamma}^1(k) \).

By continuing this process, one can recover the whole local data \(\cD(\Gamma)\) as claimed.
\end{proof}

\section{A combinatorial description for the Hilbert polynomial}
\label{sec:comb-descr-hilb}

In this section, we give combinatorial descriptions for the Hilbert
polynomial \(P_{\Gamma}^i(k)\) and its coefficients for a bunch of
grapes \(\Gamma\cong(\sfT,\ell)\) when \(\sfT\) has at least one edge.
If \(\sfT\) has no edges, then it is a singleton, and therefore
\(\Gamma\) is homeomorphic to a bouquet of circles, which will be
considered separately in \Cref{appendix}. We first consider the case
that \( \Gamma \) has a unique essential vertex.

\subsection{HE-configurations and SHE-configurations}\label{sec:HE and SHE}
Throughout this subsection, we consider the elementary bunch of grapes
\( \Gamma_{\ell,m} \), which is a graph with a unique essential vertex
\( \sfv \), \( \ell \) loops, and \( m \) non-loops. Here, a
\emph{non-loop} is an edge that is not a loop. We assume that
\(m\ge 1\). Since \(\sfv\) is an essential vertex, we have
\(2\ell+m\ge 3\) and \(\ell+m\ge 2\).

By regarding \(\Gamma\) as a CW complex, each edge \(\sfe\) is defined
as a map \(\sfe:[0,1]\to\Gamma\) whose restriction to \( (0,1) \) is
an embedding of \( (0,1) \) onto the interior of \( \sfe \). Then the two
restrictions \(\sfh^1(\sfe)\) and \(\sfh^2(\sfe)\) of \(\sfe\) to
\([0,1/3]\) and \([2/3,1]\) will be called the \emph{first} and
\emph{second half edges} of \(\sfe\), respectively. We require that the first half
edge \(\sfh^1(\sfe)\) must be attached to the central vertex \(\sfv\).
From now on, we will only consider the half edges attached to
\( \sfv \), that is, for a non-loop \(\sfe\), we will not use the
second half edge \(\sfh^2(\sfe)\).

Note that \( \Gamma_{\ell,m} \) has \(\ell+m \) edges
\(\sfe_1,\dots,\sfe_{\ell+m}\) and \( 2\ell+m \) half edges
\(\sfh_1,\dots,\sfh_{2\ell+m}\) attached to \(\sfv\). 

From now on, we assume the following:
\begin{enumerate}
\item The edge \(\sfe_1\) is a non-loop edge. (This is always possible since \( m\ge1 \).)
\item The two half edges \(\sfh^1(\sfe_i)\) and \(\sfh^2(\sfe_i)\) for
  each loop \(\sfe_i\) are consecutive. Namely, there exists
  \(1<j<2\ell+m\) such that \(\sfh_j=\sfh^1(\sfe_i)\) and
  \(\sfh_{j+1}=\sfh^2(\sfe_i)\).
\end{enumerate}
We call \(\sfe_1\) the \emph{pivot edge} of \(\Gamma_{\ell,m}\). We
denote by \(*\) the vertex of the pivot edge other than \(\sfv\).

Now we fix an embedding of \(\Gamma_{\ell,m}\) into \(\bbR^2\) so that the half edges \( \sfh_1,\dots,\sfh_{2\ell+m} \) are aligned counterclockwise.
See \Cref{fig:image16}. 
We may omit the labels for all edges and half edges since there is no ambiguity to recover them from the picture by specifying the pivot edge.

\begin{figure}
  \centering
\begin{tikzpicture}[baseline=-.5ex]
\draw[thick, fill, blue] 
(0,0) -- node[near end, above] {\(\sfe_1\)} (-36:3) circle(2pt) node[below right] {\(*\)};
\draw[thick, fill] 
(0,0) -- node[near end, above] {\(\sfe_2\)} (-12:3) circle(2pt) 
(0,0) -- node[near end, above] {\(\sfe_3\)} (12:3) circle(2pt) 
(0,0) -- node[near end, above] {\(\sfe_4\)} (36:3) circle (2pt);
\draw[line width=3, gray] 
(0,0) -- (-36:1.5) node[above] {\(\sfh_1\)} 
(0,0) -- (-12:1.5) node[above] {\(\sfh_2\)} 
(0,0) -- (12:1.5) node[above] {\(\sfh_3\)} 
(0,0) -- (36:1.5) node[above] {\(\sfh_4\)};
\draw[thick] 
(0,0) -- (60:1.5) arc (-30:210:{1.5/sqrt(3)}) node[midway, above] {\(\sfe_5\)} -- (0,0)
(0,0) -- (150:1.5) arc (60:300:{1.5/sqrt(3)}) node[midway, left] {\(\sfe_6\)} -- (0,0)
(0,0) -- (240:1.5) arc (150:390:{1.5/sqrt(3)}) node[midway, below] {\(\sfe_7\)} -- (0,0);
\draw[line width=3, gray]
(0,0) -- (60:1.5) node[left] {\(\sfh_5\)}
(0,0) -- (120:1.5) node[right] {\(\sfh_6\)}
(0,0) -- (150:1.5) node[below] {\(\sfh_7\)}
(0,0) -- (210:1.5) node[above] {\(\sfh_8\)}
(0,0) -- (240:1.5) node[right] {\(\sfh_9\)}
(0,0) -- (300:1.5) node[left] {\(\sfh_{10}\)}
;
\draw[thick, fill] (0,0) node[above] {\(\sfv\)} circle (2pt);
\end{tikzpicture}
  \caption{The graph \( \Gamma_{\ell,m} \) with \( \ell=3 \) and \( m=4 \).
    There are \( \ell+m=7 \) edges and \( 2\ell+m=10 \) half edges attached to \( \sfv \).}
  \label{fig:image16}
\end{figure}

For a vector \(\vec a = (a_1,\dots, a_n)\), we define \(\norm{\vec a}\) to be the sum of its entries.

\begin{defn}\label{def:HE}
  Let \( k\ge0 \). An \emph{HE-configuration of size \( k \)} is a pair
  \( (\sfh,\vec a) \) of a half edge \( \sfh \) adjacent to \(\sfv\) and a weight vector
  \( \vec a = (a_1,\dots,a_{\ell+m})\in \bbZ_{\ge0}^{\ell+m} \) with
  \( \norm{\vec a} = k\) satisfying one of the following
  conditions:
  \begin{enumerate}
\item \(\sfh=\sfh^1(\sfe_r)\) and \(a_r\ge 1\) for some \(1<r<\ell+m\), and there exists \(s\) such that \(r<s\le \ell+m\) and \(a_s\ge 1\),
\item \(\sfh=\sfh^2(\sfe_r)\) and \(a_r\ge 1\) for some \(1<r\le \ell+m\).
  \end{enumerate}
  We say that the HE-configuration is of \emph{type 1} if
  \( \sfh=\sfh^1(\sfe_r) \), and of \emph{type 2} if
  \( \sfh=\sfh^2(\sfe_r) \).

  For each \(\epsilon=1,2\), let \( \HE_{\ell,m}^\epsilon(k) \) denote
  the set of HE-configurations of type \(\epsilon\) of size \( k \). We
  also denote by
  \[\HE_{\ell,m}(k) = \HE_{\ell,m}^1(k)\sqcup \HE_{\ell,m}^2(k)\] the
  set of HE-configurations of size \( k \). Note that, by definition, 
  \( \HE_{\ell,m}(0) = \HE_{\ell,m}^1(0)= \HE_{\ell,m}^2(0) = \varnothing \).
\end{defn}

We may depict an HE-configuration \((\sfh,\vec a)\) on
\(\Gamma_{\ell,m}\) as follows: a thick red half edge \(\sfh\),
\((a_r-1)\) red dots on the edge \(\sfe_r\), and \(a_i\) red dots on
the edge \(\sfe_i\) for \(i\neq r\) as seen in
\Cref{fig:image13,fig:image14}.

\begin{figure}
\subcaptionbox{\((\sfh^1(\sfe_2), (2,2,3,0,2,1,0))\in\HE^1_{3,4}(10)\)}[.45\textwidth]{
\begin{tikzpicture}[baseline=-.5ex]
\draw[thick, fill, blue] 
(0,0) -- (-36:2) circle(2pt) node[below right] {\(*\)};
\draw[thick, fill] 
(0,0) -- (-12:2) circle(2pt) 
(0,0) -- (12:2) circle(2pt) 
(0,0) -- (36:2) circle (2pt);
\draw[thick] 
(0,0) -- (60:1) arc (-30:210:{1/sqrt(3)}) -- (0,0)
(0,0) -- (150:1) arc (60:300:{1/sqrt(3)}) -- (0,0)
(0,0) -- (240:1) arc (150:390:{1/sqrt(3)}) -- (0,0);
\draw[line width=3, red] 
(0,0) -- (-12:1);
\draw[thick, fill] (0,0) circle (2pt);
\draw[thick, red, fill] 
(-36:1.33) circle (2pt) (-36:1.67) circle (2pt)
(-12:1.5) circle (2pt)
(12:1.25) circle (2pt) (12:1.5) circle (2pt) (12:1.75) circle (2pt)
(90:{2/sqrt(3)}) + (60:{1/sqrt(3)}) circle (2pt) + (120:{1/sqrt(3)}) circle (2pt)
(180:{3/sqrt(3)}) circle (2pt)
;
\end{tikzpicture}
}
\subcaptionbox{\((\sfh^1(\sfe_5), (2,0,2,0,3,1,0))\in\HE^1_{3,4}(8)\)}[.45\textwidth]{
\begin{tikzpicture}[baseline=-.5ex]
\draw[thick, fill, blue] 
(0,0) -- (-36:2) circle(2pt) node[below right] {\(*\)};
\draw[thick, fill] 
(0,0) -- (-12:2) circle(2pt) 
(0,0) -- (12:2) circle(2pt) 
(0,0) -- (36:2) circle (2pt);
\draw[thick] 
(0,0) -- (60:1) arc (-30:210:{1/sqrt(3)}) -- (0,0)
(0,0) -- (150:1) arc (60:300:{1/sqrt(3)}) -- (0,0)
(0,0) -- (240:1) arc (150:390:{1/sqrt(3)}) -- (0,0);
\draw[line width=3, red] 
(0,0) -- (60:1);
\draw[thick, fill] (0,0) circle (2pt);
\draw[thick, red, fill] 
(-36:1.33) circle (2pt) (-36:1.67) circle (2pt)
(12:1.33) circle (2pt) (12:1.67) circle (2pt)
(90:{2/sqrt(3)}) + (60:{1/sqrt(3)}) circle (2pt) + (120:{1/sqrt(3)}) circle (2pt)
(180:{3/sqrt(3)}) circle (2pt)
;
\end{tikzpicture}
}
\caption{HE-configurations of type \(1\).}
\label{fig:image13}
\end{figure}

\begin{figure}
  \centering
\begin{tikzpicture}[baseline=-.5ex]
\draw[thick, fill, blue] 
(0,0) -- (-36:2) circle(2pt) node[below right] {\(*\)};
\draw[thick, fill] 
(0,0) -- (-12:2) circle(2pt) 
(0,0) -- (12:2) circle(2pt) 
(0,0) -- (36:2) circle (2pt);
\draw[thick] 
(0,0) -- (60:1) arc (-30:210:{1/sqrt(3)}) -- (0,0)
(0,0) -- (150:1) arc (60:300:{1/sqrt(3)}) -- (0,0)
(0,0) -- (240:1) arc (150:390:{1/sqrt(3)}) -- (0,0);
\draw[line width=3, red] 
(0,0) -- (120:1);
\draw[thick, fill] (0,0) circle (2pt);
\draw[thick, red, fill] 
(-36:1.5) circle (2pt)
(12:1.33) circle (2pt) (12:1.67) circle (2pt)
(180:{2/sqrt(3)}) + (90:{1/sqrt(3)}) circle (2pt) + (270:{1/sqrt(3)}) circle (2pt) + (180:{1/sqrt(3)}) circle (2pt)
(270:{3/sqrt(3)}) circle (2pt)
;
\end{tikzpicture}

  \caption{An HE-configuration \( (\sfh,\vec a)\in\HE^2_{3,4}(8) \) of type 2, where
    \( \sfh=\sfh^2(\sfe_5) \) and \( \vec a = (1,0,2,0,1,3,1) \).}
\label{fig:image14}
\end{figure}

\begin{defn}\label{defn:SHE}
  An HE-configuration \( (\sfh,\vec b=(b_1,\dots,b_{\ell+m})) \) is \emph{standard} if 
\(b_1=0\) and \( b_i\in \{0,1\} \) for all \( i\ge2 \).
We also call a standard HE-configuration an \emph{SHE-configuration}.

For each \(\epsilon=1,2\), let \( \SHE_{\ell,m}^\epsilon(j) \) denote the set of standard
HE-configurations of type \(\epsilon\) of size \( j \). We also
denote by \[ \SHE_{\ell,m}(j) = \SHE_{\ell,m}^1(j)\sqcup \SHE_{\ell,m}^2(j)\] the set of standard
HE-configurations of size \( j \).
\end{defn}

Note that if \( (\sfh,\vec b)\in \SHE_{\ell,m}(j) \), then
\( \sum_{i}b_i = j \) and there are exactly \( j \) nonzero entries in
\( \vec b \).

\begin{prop}\label{pro:HE=sum SHE}
  For \( k\ge0 \), we have
  \begin{align}
    \label{eq:3}
    |\HE_{\ell,m}(k)| &= \sum_{j\ge1} |\SHE_{\ell,m}(j)|\binom{k}{j},\\
    \label{eq:4}
    |\HE_{\ell,m}^1(k)| &= \sum_{j\ge1} |\SHE_{\ell,m}^1(j)|\binom{k}{j},\\
    \label{eq:5}
    |\HE_{\ell,m}^2(k)| &= \sum_{j\ge1} |\SHE_{\ell,m}^2(j)|\binom{k}{j}.
  \end{align}
\end{prop}

\begin{proof}
Since \Cref{eq:3} is obtained
  by adding \Cref{eq:4,eq:5}, we only need to prove the
  last two identities.

Let \(\epsilon\) be either \(1\) or \(2\).
For each \((\sfh,\vec a=(a_1,\dots,a_{\ell+m}))\in\HE^\epsilon_{\ell,m}(k)\),
we define the vector \(\vec b=(b_1,\dots,b_{\ell+m})\) by 
\[
b_1=0
\qquad\text{ and }\qquad
b_i=\begin{cases}
0 & \text{if }a_i=0;\\
1 & \text{if }a_i\ge 1,
\end{cases}
\quad\text{ for }i\ge 2.
\]
Then we have \((\sfh, \vec b)\in\SHE^\epsilon_{\ell,m}(j)\) for some \(j\).
Hence the set \(\HE^\epsilon_{\ell,m}(k)\) can be decomposed into the sets \(\SHE^\epsilon_{\ell,m}(j)\) for \(j\ge 1\).

  Conversely, suppose \( 1\le j\le k \) and let
  \( (\sfh,\vec b)\in \SHE_{\ell,m}^\epsilon(j) \) be a standard \(\HE\)-configuration of type
  \(\epsilon\) with \( \sfh=\sfh^\epsilon(e_r) \). Let
  \( 1\le t_1<t_2<\dots<t_j\le \ell+m \) be the indices \( i \) such
  that \( b_i\ne 0 \). We will show that there are \( \binom{k}{j} \)
  ways to construct an \(\HE\)-configuration \( (\sfh,\vec a')\in\HE_{\ell,m}^\epsilon(k) \)
  using \( (\sfh,\vec a) \). Note that \( \binom{k}{j} \) is the number
  of sequences \( \vec c=(c_0,\dots,c_j)\in \bbZ_{\ge 1}^{j+1} \) of positive
  integers such that \(\norm{\vec c}=k+1\).
  Given such a sequence \( \vec c \), we define the sequence \( \vec a = (a_1,\dots,a_{\ell+m}) \) as
  follows:
  \begin{equation}\label{eq:2}
    a_{p} =
    \begin{cases}
    c_0-1 & \text{if }p=1,\\
    c_i  & \text{if }p=t_i\text{ for some }1\le i\le j,\\
    0  & \text{otherwise.}
    \end{cases}
  \end{equation}
  Then \( (\sfh,\vec a)\in \HE_{\ell,m}^\epsilon(k) \). Since any element in
  \( (\sfh,\vec a)\in \HE_{\ell,m}^\epsilon(k) \) can be constructed in this
  way, we obtain \Cref{eq:4,eq:5}.
\end{proof}

\begin{prop}\label{prop:SHE}
For \( j\ge1 \),  we have
  \begin{align}
    \label{eq:SHE}
    |\SHE_{\ell,m}(j)| &= (2\ell+m-2)\binom{\ell+m-2}{j-1} - \binom{\ell+m-2}{j},\\
    \label{eq:SHE1}
    |\SHE_{\ell,m}^1(j)| &= (\ell+m-2)\binom{\ell+m-2}{j-1} - \binom{\ell+m-2}{j},\\
    \label{eq:SHE2}
    |\SHE_{\ell,m}^2(j)| &= \ell\binom{\ell+m-2}{j-1}.
  \end{align}
\end{prop}

\begin{proof}
  Note that \Cref{eq:SHE} is obtained by adding \Cref{eq:SHE1,eq:SHE2}.
  \Cref{eq:SHE2} follows easily from
  the observation that each \( (\sfh,\vec a)\in \SHE_{\ell,m}^2(j) \) is
  obtained by choosing a loop \( \sfe_r \) for \( \sfh=\sfh^2(\sfe_r) \) in
  \( \ell \) ways and choosing a weight vector \( \vec a \) satisfying
  \( a_i\le 1 \) for all \( i \), \( a_1=0 \), \( a_r=1 \), and
  \( \norm{\vec a} = j \) in \( \binom{\ell+m-2}{j-1} \)
  ways.

  To prove \Cref{eq:SHE1}, observe that if
  \( (\sfh,\vec a)\in \SHE_{\ell,m}^1(k) \), then \( \sfh=\sfh^1(\sfe_r) \) is
  determined by the index \( r \), which satisfies \( 1<r<\ell+m \).
  Hence, letting
  \( W \) denote
  the set of weight vectors
  \( \vec a = (a_1,\dots,a_{\ell+m}) \)
  such that \( a_i\le 1 \) for all \( i \)
  and \( \norm{\vec a} = k \), we have
\( |\SHE_{\ell,m}^1(k)| = |A| \), where
  \[
    A = \left\{(r,\vec a): 1<r<\ell+m, \vec a\in W, a_1=0, a_r=1,
      \mbox{ and there is \( r<s\le \ell+m \) such that \( a_s=1 \)} \right\}.
  \]
  Then \( |A|=|A_1|-|A_2| \), where
  \begin{align*}
    A_1 &= \left\{(r,\vec a): 1<r<\ell+m, \vec a\in W, a_1=0, a_r=1 \right\},\\
    A_2 &= \left\{(r,\vec a): 1<r<\ell+m, \vec a\in W, a_1=0, a_r=1 ,
     \mbox{ and there is no \( r<s\le \ell+m \) such that \( a_s=1 \)} \right\}.
  \end{align*}
  It suffices to show that
  \begin{align}
    \label{eq:7}
    |A_1|&= (\ell+m-2) \binom{\ell+m-2}{k-1},\\
    \label{eq:8}
    |A_2|&= \binom{\ell+m-2}{k}.
  \end{align}

  \Cref{eq:7} follows from the observation that there
  are \( \ell+m-2 \) ways to choose \( r \) such that
  \( 1<r<\ell+m \), and then \( \binom{\ell+m-2}{k-1} \) ways to
  choose \( \vec a\in W \) such that \( a_1=0 \) and \( a_r= 1 \). To
  see \Cref{eq:8}, note that if \( (r,\vec a)\in A_2 \), then
  \( r \) is determined by \( \vec a \), namely,
  \( r = \max\{i:a_i=1\} \). Hence, \( |A_2| \) is the number of
  weight vectors \( \vec a \in W \) such that \( a_1=0 \) and
  \( 1<\max\{i:a_i=1\}<\ell+m \). Equivalently, \( |A_2| \) is the
  number of weight vectors \( \vec a \in W \) such that
  \( a_1=a_{\ell+m}=0 \), which implies \Cref{eq:8}. This completes
  the proof.
\end{proof}

\begin{prop}\label{prop:HE and H1}
  For \( k\ge0 \), we have
  \begin{align}
    \label{eq:HE}
    |\HE_{\ell,m}(k)| &= (2\ell+m-2)\binom {k+\ell+m-2}{\ell+m-1} - \binom{k+\ell+m-2}{\ell+m-2}+1,\\
    \label{eq:HE1}
    |\HE_{\ell,m}^1(k)| &= (\ell+m-2)\binom {k+\ell+m-2}{\ell+m-1} - \binom{k+\ell+m-2}{\ell+m-2}+1,\\
    \label{eq:HE2}
    |\HE_{\ell,m}^2(k)| &= \ell \binom{k+\ell+m-2}{\ell+m-1}.
  \end{align}
\end{prop}

\begin{proof}
  Note that \Cref{eq:HE} is obtained by adding \Cref{eq:HE1} and
  \Cref{eq:HE2}. By \Cref{eq:4,eq:SHE1}, and the identity
  \( \sum_{j\ge0} \binom{a}{j} \binom{b}{c-j} = \binom{a+b}{c} \), we
  have
  \begin{align*}
    |\HE_{\ell,m}^1(k)|
    &= (\ell+m-2)\sum_{j\ge1} \binom{k}{j} \binom{\ell+m-2}{j-1}
      - \sum_{j\ge1} \binom{k}{j} \binom{\ell+m-2}{j}\\
    &= (\ell+m-2)\sum_{j\ge0} \binom{k}{j} \binom{\ell+m-2}{\ell+m-1-j}
      - \sum_{j\ge0} \binom{k}{j} \binom{\ell+m-2}{\ell+m-2-j} + 1\\
     &= (\ell+m-2)\binom{k+\ell+m-2}{\ell+m-1} - \binom{k+\ell+m-2}{\ell+m-2} + 1,
  \end{align*}
  which gives \Cref{eq:HE1}. Similarly, using \Cref{eq:5} and
  \Cref{eq:SHE2}, we obtain \Cref{eq:HE2}.
\end{proof}

By \eqref{equation:residual for elementary}, \eqref{eq:HE}, and \eqref{eq:3}, we obtain the following
combinatorial descriptions for \( P^1_{\Gamma_{\ell,m}}(k) \) and its
coefficients. 
\begin{prop}\label{pro:1}
  Let \( \ell\ge0 \) and \( m\ge1 \) be integers
  such that \(2\ell+m\ge 3\).
  Then we have
  \begin{equation}\label{eq:P1_comb}
    P^1_{\Gamma_{\ell,m}}(k) = 
        |\HE_{\ell,m}(k)| = \sum_{j\ge1} |\SHE_{\ell,m}(j)|\binom{k}{j}.
  \end{equation}
In other words,
 \begin{enumerate}
 \item for a nonnegative integer \( k \),
   \( P^1_{\Gamma_{\ell,m}}(k) \) is the number of ways to mark a half
   edge \( \sfh \) and place \( (k-1) \) dots in \( \Gamma_{\ell,m} \)
   in such a way that
\begin{enumerate}
\item \( \sfh \) is not on the pivot edge, and
\item if \( \sfh = \sfh^1(\sfe) \), then there is a dot in an edge
  with a label smaller than that of \( \sfe \);
\end{enumerate}
\item considering \( P^1_{\Gamma_{\ell,m}}(k) \) as a polynomial in
  \( k \), its coefficient of \( \binom{k}{j} \) is the number of ways
  to mark a half edge \( \sfh \) and place \( (j-1) \) dots in
  \( \Gamma_{\ell,m} \) in such a way that
\begin{enumerate}
\item \( \sfh \) is not on the pivot edge,
\item every edge has at most one dot,
  and the first edge does not have a dot,
\item if \( \sfh = \sfh^1(\sfe) \), then there is a dot in an edge
  with a label smaller than that of \( \sfe \).
\end{enumerate}
 \end{enumerate}
\end{prop}

\subsection{Hilbert polynomials}

Let \(\Gamma\cong(\sfT,\loops)\in\Grape\).
As mentioned earlier at the beginning of this section, we assume that the stem \(\sfT\) has at least one edge.

An \emph{oriented root} of \(\Gamma\) is a pair \((\sfv_0,\sfe_0)\) of
a vertex \( \sfv_0 \) and an edge \(\sfe_0\) of \( \sfT \)
adjacent to $\sfv_0$. We call \(\sfv_0\) and \(\sfe_0\) the \emph{root
  vertex} and the \emph{root edge}, respectively.

Let us fix an oriented root of \(\Gamma\). We label the edges of
\(\Gamma\) and its decomposition as seen in \Cref{figure:decomposion
  along the stem}. More precisely, at each vertex \(\sfv\) of the stem
\(\sfT\) different from \(\sfv_0\), let \(\sfe^\sfv\) be the unique
edge adjacent to \(\sfv\) that is the first edge of the shortest path
from \(\sfv\) to \(\sfv_0\). Then the edge \(\sfe^\sfv\) defines the
pivot edge \(\sfe_1^{\sfv}\) of the local graph \(\Gamma_\sfv\) in a
natural way.

Now we fix an embedding of \(\Gamma\) in \( \bbR^2 \) such that the
restriction to each local graph \(\Gamma_\sfv\) as a subspace of
\(\Gamma\) satisfies the conditions on the embedding of
\( \Gamma_{\ell,m} \) described in \Cref{sec:HE and SHE}. Since the
pivot edge \( \sfe^{\sfv}_1 \) of
\( \Gamma_{\sfv} = \Gamma_{\ell(\sfv),m(\sfv)} \) has been defined for
each essential vertex \( \sfv \) of \( \Gamma \), we can label all
edges and half edges of \( \Gamma_\sfv \) with
\( \sfe^{\sfv}_1,\dots,\sfe^{\sfv}_{\ell(\sfv)+m(\sfv)} \) and
\( \sfh^{\sfv}_1,\dots,\sfh^{\sfv}_{2\ell(\sfv)+m(\sfv)} \) in
counterclockwise order by using the given embedding of \(\Gamma\) as
before. Then every half edge adjacent to \( \sfv \) in \(\Gamma_\sfv\)
can be identified with a unique half edge in \(\Gamma\) adjacent to \(\sfv\).

Finally, the label of each loop in \(\Gamma\) is given by the label of
the corresponding loop in the decomposition. See
\Cref{figure:labelling1,figure:labelling2}.

\begin{figure}
  \subcaptionbox{Labels on edges of \(\Gamma\).
      The oriented root of \(\Gamma\) is \((\mathsf{0},\sfe^{\mathsf{13}})\), which is denoted by a red arrow.
    \label{figure:labelling1}}[.8\textwidth]{
\(
\begin{tikzpicture}[baseline=-.5ex]
\draw[thick,fill] (0,0) circle (2pt) node (A) {} node[above] {\(\mathsf{1}\)} -- ++(5:2) node[midway, above] {\(\sfe^{\mathsf{1}}\)} circle(2pt) node(B) {} node[above left] {\(\mathsf{0}\)};
\draw[thick,fill] (B.center) +(0,1.5) circle (2pt) node (C) {} node[left=0.5ex] {\(\mathsf{13}\)};
\draw[very thick,fill,<-, red] (B.center)+(0,2pt) -- (C.center) node[midway, right] {\(\sfe^{\mathsf{13}}\)};
\draw[thick,fill] (B.center) -- +(0,-1.5) node[midway, right] {\(\sfe^{\mathsf{2}}\)} circle (2pt) node(D) {} node[left] {\(\mathsf{2}\)} +(0,0) -- ++(-10:1.5) node[midway, above] {\(\sfe^{\mathsf{3}}\)} circle(2pt) node(E) {} node[above] {\(\mathsf{3}\)} -- ++(10:1.5) node[midway, below] {\(\sfe^{\mathsf{7}}\)} circle(2pt) node(F) {} node[below=0.5ex] {\(\mathsf{7}\)} -- +(120:1) node[midway, left] {\(\sfe^{\mathsf{12}}\)} circle(2pt) node(G) {} node[xshift=-1.2ex,yshift=2ex] {\(\mathsf{12}\)} +(0,0) -- +(60:1) node[midway, right] {\(\sfe^{\mathsf{11}}\)} circle (2pt) node(H) {} node[xshift=1.2ex,yshift=2ex] {\(\mathsf{11}\)} +(0,0) -- ++(-5:1) node[midway, below] {\(\sfe^{\mathsf{8}}\)} circle (2pt) node(I) {} node[below] {\(\mathsf{8}\)} -- +(30:1) node[midway, above] {\(\sfe^{\mathsf{10}}\)} circle(2pt) node(J) {} node[right] {\(\mathsf{10}\)} +(0,0) -- +(-30:1) node[midway, below] {\(\sfe^{\mathsf{9}}\)} circle (2pt) node(K) {} node[above] {\(\mathsf{9}\)};
\draw[thick, fill] (E.center) -- ++(0,-1) node[midway, right] {\(\sfe^{\mathsf{4}}\)} circle (2pt) node[left] {\(\mathsf{4}\)} -- +(-60:1) node[midway, right] {\(\sfe^{\mathsf{6}}\)} circle (2pt) node[right] {\(\mathsf{6}\)} +(0,0) -- +(-120:1) node[midway, left] {\(\sfe^{\mathsf{5}}\)} circle (2pt) node[left] {\(\mathsf{5}\)};
\grape[180]{A};
\draw (A.center)+(180:{1.5/sqrt(3)}) node[left] {\(\sfe_2^{\mathsf{1}}\)};
\grape[0]{C}; \grape[90]{C}; \grape[180]{C};
\draw (C.center)+(0:{1.5/sqrt(3)}) node[right] {\(\sfe_2^{\mathsf{13}}\)};
\draw (C.center)+(90:{1.5/sqrt(3)}) node[above] {\(\sfe_3^{\mathsf{13}}\)};
\draw (C.center)+(180:{1.5/sqrt(3)}) node[left] {\(\sfe_4^{\mathsf{13}}\)};
\grape[-90]{D};
\draw (D.center)+(-90:{1.5/sqrt(3)}) node[below] {\(\sfe_2^{\mathsf{2}}\)};
\grape[-90]{F};
\draw (F.center)+(-90:{1.5/sqrt(3)}) node[below] {\(\sfe_2^{\mathsf{7}}\)};
\grape[120]{G};
\draw (G.center)+(120:{1.5/sqrt(3)}) node[above] {\(\sfe_2^{\mathsf{12}}\)};
\grape[60]{H};
\draw (H.center)+(60:{1.5/sqrt(3)}) node[above] {\(\sfe_2^{\mathsf{11}}\)};
\grape[-30]{K};
\draw (K.center)+(-30:{1.5/sqrt(3)}) node[right] {\(\sfe_2^{\mathsf{9}}\)};
\end{tikzpicture}
\)
}
\subcaptionbox{Labels on edges of the decomposition of \(\Gamma\).
  The oriented root of \(\Gamma_{\mathsf{0}}\) is \((\mathsf{0},\sfe_1^{\mathsf{0}})\), which is denoted by a red arrow.
  \label{figure:labelling2}}[.8\textwidth]{
\(
\begin{tikzpicture}[baseline=-.5ex]
\draw[thick,fill] (0,0) circle (2pt) node (A) {} node[above] {\(\mathsf{1}\)} [blue] -- ++(5:1) circle (2pt) node[midway,above] {\(\sfe_1^{\mathsf{1}}\)};
\draw (-180:{1.5/sqrt(3)}) node[left] {\(\sfe_2^{\mathsf{1}}\)};
\grape[180]{A};
\begin{scope}[xshift=0.2cm]
	\draw[thick,fill] (0,0) +(5:1) circle (2pt) -- ++(5:2) circle(2pt) node(B) {} node[above left] {\(\mathsf{0}\)} +(0,0) -- +(0,-0.75) circle (2pt) +(0,0) -- +(-10:0.75) circle (2pt);
	\draw[very thick, fill, red,<-] (B.center)+(0,2pt) -- +(0,0.75) node[midway, right] {\(\sfe_1^{\mathsf{0}}\)};
	\draw[thick, fill, blue] (B.center) +(0,0.75) circle (2pt);

	\begin{scope}[yshift=0.2cm]
		\draw[thick,fill, blue] (0,0) ++(5:2) +(0,0.75) circle (2pt) -- +(0,1.5) node[near start,right] {\(\sfe_1^{\mathsf{13}}\)};
		\draw[thick,fill] (0,0) ++(5:2) +(0,0.75) +(0,1.5) circle (2pt) node (C) {}  node[left=0.5ex] {\(\mathsf{13}\)};
		\grape[0]{C}; \grape[90]{C}; \grape[180]{C};
		\draw (C.center) +(0:{1.5/sqrt(3)}) node[right] {\(\sfe_2^{\mathsf{13}}\)};
		\draw (C.center) +(90:{1.5/sqrt(3)}) node[above] {\(\sfe_3^{\mathsf{13}}\)};
		\draw (C.center) +(180:{1.5/sqrt(3)}) node[left] {\(\sfe_4^{\mathsf{13}}\)};
	\end{scope}

	\begin{scope}[yshift=-0.2cm]
		\draw[thick,fill, blue] (0,0) ++(5:2) +(0,-0.75) circle (2pt) -- +(0,-1.5) node[midway,right] {\(\sfe_1^{\mathsf{2}}\)};
		\draw[thick,fill] (0,0) ++(5:2) +(0,-0.75) +(0,-1.5) circle (2pt) node (D) {} node[left] {\(\mathsf{2}\)};
		\grape[-90]{D};
		\draw (D.center) +(-90:{1.5/sqrt(3)}) node[below] {\(\sfe_2^{\mathsf{2}}\)};
	\end{scope}

	\begin{scope}[xshift=0.2cm]
		\draw[thick,fill, blue] (0,0) ++(5:2) ++(-10:0.75) circle (2pt) -- ++(-10:0.75) node[midway,above] {\(\sfe_1^{\mathsf{3}}\)};
		\draw[thick,fill] (0,0) ++(5:2) ++(-10:0.75) ++(-10:0.75) circle(2pt) node(E) {} node[above] {\(\mathsf{3}\)} -- +(10:0.75) circle (2pt) +(0,0) -- +(0,-0.5) circle (2pt);

		\begin{scope}[yshift=-0.2cm]
			\draw[thick,fill, blue] (0,0) ++(5:2) ++(-10:1.5) ++(0,-0.5) circle (2pt) -- ++(0,-0.5) node[midway,right] {\(\sfe_1^{\mathsf{4}}\)};
			\draw[thick,fill] (0,0) ++(5:2) ++(-10:1.5) ++(0,-0.5) ++(0,-0.5) circle (2pt) node[left] {\(\mathsf{4}\)} -- +(-60:0.5) +(0,0) -- +(-120:0.5);

			\begin{scope}[yshift=0cm]
				\draw[thick,fill] (0,0) ++(5:2) ++(-10:1.5) ++(0,-1) +(-120:0.5) -- +(-120:1) circle (2pt) node[left] {\(\mathsf{5}\)} +(-60:0.5) -- +(-60:1) circle (2pt) node[right] {\(\mathsf{6}\)};
			\end{scope}
		\end{scope}

		\begin{scope}[xshift=0.2cm]
			\draw[thick,fill,blue] (0,0) ++(5:2) ++(-10:1.5) ++(10:0.75) circle(2pt) -- ++(10:0.75) node[near start,below] {\(\sfe_1^{\mathsf{7}}\)};
			\draw[thick,fill] (0,0) ++(5:2) ++(-10:1.5) ++(10:0.75) ++(10:0.75) circle(2pt) node(F) {} node[below=0.5ex] {\(\mathsf{7}\)} -- +(120:0.5) circle (2pt) +(0,0) -- +(60:0.5) circle (2pt) +(0,0) -- +(-5:0.5) circle (2pt);
			\grape[-90]{F};
			\draw (F.center) +(-90:{1.5/sqrt(3)}) node[below] {\(\sfe_2^{\mathsf{7}}\)};

			\begin{scope}[xshift=-.2cm, yshift=0.2cm]
				\draw[thick,fill,blue] (0,0) ++(5:2) ++(-10:1.5) ++(10:1.5) +(120:0.5) circle (2pt) -- +(120:1) node[near start,left] {\(\sfe_1^{\mathsf{12}}\)};
				\draw[thick,fill] (0,0) ++(5:2) ++(-10:1.5) ++(10:1.5) +(120:1) circle (2pt) node(G) {} node[xshift=-1.2ex,yshift=2ex] {\(\mathsf{12}\)};
				\grape[120]{G};
				\draw (G.center) +(120:{1.5/sqrt(3)}) node[above] {\(\sfe_2^{\mathsf{12}}\)};
			\end{scope}
			\begin{scope}[xshift=.2cm, yshift=0.2cm]
				\draw[thick,fill,blue] (0,0) ++(5:2) ++(-10:1.5) ++(10:1.5) +(60:0.5) circle (2pt) -- +(60:1) node[near start,right] {\(\sfe_1^{\mathsf{11}}\)};
				\draw[thick,fill] (0,0) ++(5:2) ++(-10:1.5) ++(10:1.5) +(60:1) node(H) {} node[xshift=1.2ex,yshift=2ex] {\(\mathsf{11}\)};
				\grape[60]{H};
				\draw (H.center) +(60:{1.5/sqrt(3)}) node[above] {\(\sfe_2^{\mathsf{11}}\)};
			\end{scope}

			\begin{scope}[xshift=0.2cm]
				\draw[thick,fill,blue] (0,0) ++(5:2) ++(-10:1.5) ++(10:1.5) ++(-5:0.5) circle (2pt) -- ++(-5:0.5) node[midway, above] {\(\sfe_1^{\mathsf{8}}\)};
				\draw[thick,fill] (0,0) ++(5:2) ++(-10:1.5) ++(10:1.5) ++(-5:0.5) ++(-5:0.5) circle (2pt) node(I) {} node[below] {\(\mathsf{8}\)} -- +(30:0.5) +(0,0) -- +(-30:0.5) circle (2pt);

				\begin{scope}[xshift=0cm]
					\draw[thick,fill] (0,0) ++(5:2) ++(-10:1.5) ++(10:1.5) ++(-5:1) +(30:0.5) -- +(30:1) circle (2pt) node[right] {\(\mathsf{10}\)};
				\end{scope}
				\begin{scope}[xshift=0.2cm,yshift=-.2cm]
					\draw[thick,fill,blue] (0,0) ++(5:2) ++(-10:1.5) ++(10:1.5) ++(-5:1) +(-30:0.5) circle (2pt) -- +(-30:1) node[midway,below left] {\(\sfe_1^{\mathsf{9}}\)};
					\draw[thick,fill] (0,0) ++(5:2) ++(-10:1.5) ++(10:1.5) ++(-5:1) +(-30:0.5) +(-30:1) circle (2pt) node(K) {} node[above] {\(\mathsf{9}\)};
					\grape[-30]{K};
					\draw (K.center) +(-30:{1.5/sqrt(3)}) node[right] {\(\sfe_2^{\mathsf{9}}\)};
				\end{scope}
			\end{scope}
		\end{scope}
	\end{scope}
\end{scope}
\end{tikzpicture}
\)
}
\caption{Labeling convention of edges in \(\Gamma\) and its decomposition.
For brevity, some labels on edges are suppressed.}
\label{figure:labelling}
\end{figure}

For the sake of convenience, we fix a linear order on the edges.
\begin{remark}
  By using a fixed a planar embedding and the oriented root
  \((\sfv_0,\sfe_0)\), one may assign linear orders to the vertices
  and edges. One typical way is to assign linear orders by traveling
  counterclockwise along the boundary of a thickening of \(\Gamma\),
  starting from a point on \(\sfe_0\) toward \(\sfv_0\).

The labels of the vertices in \Cref{figure:stem and number of grapes,figure:labelling1} follow this convention.
Then the linear order on the edges are given as follows: let \(\sfe_1^\sfv=\sfe^\sfv\) and
\[
\sfe^{\sfv}_r < \sfe^{\sfw}_s \Longleftrightarrow \sfv<\sfw \text{ or } (\sfv=\sfw \text{ and } r<s).
\]
\end{remark}

\begin{defn}[SHE- and HE-configurations on bunches of grapes]\label{defn:SHE for grapes}
Let \(\Gamma\cong(\sfT,\loops)\).
For each \(W\subseteq V^\ess(\Gamma)\), \(\vec j=(j_\sfv)_{\sfv\in W}\in \bbZ_{\ge 1}^W\), and \(i,j,k\ge 0\), we define 
\begin{align*}
\SHE_\Gamma(W,\vec j)&=\prod_{\sfv\in W} \SHE_{\loops(\sfv),m(\sfv)}(j_\sfv),&
\HE_\Gamma(W,\vec j,k)&=\SHE_\Gamma(W,\vec j)\times\{\vec c\in \bbZ_{\ge 1}^{\norm{\vec j}}:\norm{\vec c}\le k \},\\
\SHE_\Gamma(W,j)&=\coprod_{\substack{\vec j\in \bbZ_{\ge 1}^W\\ \norm{\vec j} = j}}  \SHE_\Gamma(W,\vec j),&
\HE_\Gamma(W,j,k)&=\SHE_\Gamma(W,j)\times\{\vec c\in \bbZ_{\ge 1}^{j}:\norm{\vec c}\le k \},\\
\SHE_\Gamma^i(j) &= \coprod_{\substack{W\subseteq V^\ess(\Gamma)\\ |W|=i}} \SHE_\Gamma(W,j),&
\HE_\Gamma^i(j,k)&=\SHE_\Gamma^i(j)\times\{\vec c\in \bbZ_{\ge 1}^{j}:\norm{\vec c}\le k \},
\end{align*}
and 
\[
\HE^i_\Gamma(k)=\coprod_{j\ge0} \HE^i_\Gamma(j,k).
\]
\end{defn}

By definition, one can easily see that
\begin{align*}
|\HE_\Gamma(W,\vec j, k)|&=|\SHE_\Gamma(W,\vec j)|\binom k {\norm{\vec j}},&
|\HE_\Gamma(W,j, k)|&=|\SHE_\Gamma(W,j)|\binom k j,
\end{align*}
and
\[
|\HE_\Gamma^i(j,k)|=|\SHE_\Gamma^i(j)|\binom k j.
\]
Thus, by \Cref{theorem:decomposition formula}, \eqref{eq:P1_comb},
\Cref{equation:product}, and \Cref{defn:SHE for grapes}, we have
\begin{equation}\label{eq:1}
P_{\Gamma}^i(k) =\sum_{j\ge 0}
|\SHE_\Gamma^i(j)|\binom k j
=\sum_{j\ge 0}
|\HE_\Gamma^i(j,k)|
=|\HE_\Gamma^i(k)|.
\end{equation}

Let \(\mathscr{X}=(W,\vec j, ((\sfh^\sfv,\vec a^\sfv)), \vec c)\in \HE^i_\Gamma(j,k)\) be an HE-configuration consisting of the following data:
\begin{enumerate}
\item a subset \(W\subseteq V^\ess(\Gamma)\) with \(|W|=i\),
\item a sequence \(\vec j=(j_\sfv)\in \mathbb{Z}_{\ge 1}^W\) of positive integers with \(\norm{\vec j}=j\),
\item a sequence \(((\sfh^\sfv, \sfa^\sfv))_{\sfv\in W}\) of SHE-configurations in  \(\SHE_{\ell(\sfv),m(\sfv)}(j_\sfv)\), and
\item a sequence \(\vec c\in \mathbb{Z}_{\ge 1}^j\) of positive integer with \(\norm{\vec c}\le k\).
\end{enumerate}
In what follows, we will visualize \( \mathscr{X} \) as a
configuration on the whole graph \( \Gamma \).

First, observe that even though each SHE-configuration
\((\sfh^\sfv, \sfa^\sfv)\) is for the local graph \(\Gamma_\sfv\), one
can regard the sequence \((\sfh^\sfv, \sfa^\sfv)_{\sfv\in W}\) as a
single configuration on \(\Gamma\) consisting of dots and half edges.
See \Cref{fig:example_conf,fig:example_conf2} for examples. More
precisely, for each dot on an edge \(\sfe\) in a dot-and-half edge
configuration on \(\Gamma\), it must belong to the local graph
\(\Gamma_\sfv\) for the vertex \(\sfv\) that is the endpoint of the
edge \(\sfe\) closer to the root \(\sfv_0\).

\begin{figure}
\subcaptionbox{A SHE-configuration on the decomposition of \(\Gamma\).\label{fig:example_conf}}[.8\textwidth]{
\(
\begin{tikzpicture}[baseline=-.5ex]
\draw[thick,fill] (0,0) circle (2pt) node (A) {} node[above] {\(\mathsf{1}\)} [blue] -- ++(5:1) circle (2pt);
\grape[180]{A};
\begin{scope}[xshift=0.2cm]
	\draw[thick,fill] (0,0) +(5:1) circle (2pt) -- ++(5:2) node(B) {} node[above left] {\(\mathsf{0}\)} +(0,0) -- +(0,-0.75) circle (2pt) +(0,0) -- +(-10:0.75) circle (2pt);
	\draw[thick,fill,red] (B.center) ++(-10:0.5) circle(2pt);
	\draw[line width=3,fill,red] (0,0) +(5:1.5) -- (5:2);
	\draw[thick,fill] (B.center) circle (2pt);
	\draw[thick, fill, blue] (B.center) +(0,0.75) circle (2pt);
	\draw[thick, fill, blue, <-] (B.center)+(0,2pt) -- +(0,0.75);

	\begin{scope}[yshift=0.2cm]
		\draw[thick,fill, blue] (0,0) ++(5:2) +(0,0.75) circle (2pt) -- +(0,1.5);
		\draw[thick,fill] (0,0) ++(5:2) +(0,0.75) +(0,1.5) node (C) {} node[left=0.5ex] {\(\mathsf{13}\)};
		\grape[0]{C}; \grape[90]{C}; \grape[180]{C};
		\draw[line width=3,fill,red] (0,0) ++(5:2) ++(0,1.5) -- +(30:0.5);
		\draw[thick, fill, red] (0,0) ++(5:2) ++(0,1.5) +(90:{1.5/sqrt(3)}) circle (2pt)
			+(180:{1.5/sqrt(3)}) circle (2pt);
		\draw[thick, fill] (C.center) circle (2pt);
	\end{scope}

	\begin{scope}[yshift=-0.2cm]
		\draw[thick,fill, blue] (0,0) ++(5:2) +(0,-0.75) circle (2pt) -- +(0,-1.5);
		\draw[thick,fill] (0,0) ++(5:2) +(0,-0.75) +(0,-1.5) circle (2pt) node (D) {} node[left] {\(\mathsf{2}\)};
		\grape[-90]{D};
	\end{scope}

	\begin{scope}[xshift=0.2cm]
		\draw[thick,fill, blue] (0,0) ++(5:2) ++(-10:0.75) circle (2pt) -- ++(-10:0.75);
		\draw[thick,fill] (0,0) ++(5:2) ++(-10:0.75) ++(-10:0.75) node(E) {} node[above] {\(\mathsf{3}\)} -- +(10:0.75) circle (2pt) +(0,0) -- +(0,-0.5) circle (2pt);
		\draw[thick, fill, red] (E.center) +(10:0.5) circle (2pt);
		\draw[line width=3, fill, red] (E.center) -- +(0,-0.3);
		\draw[thick, fill] (E.center) circle (2pt);

		\begin{scope}[yshift=-0.2cm]
			\draw[thick,fill, blue] (0,0) ++(5:2) ++(-10:1.5) ++(0,-0.5) circle (2pt) -- ++(0,-0.5);
			\draw[thick,fill] (0,0) ++(5:2) ++(-10:1.5) ++(0,-0.5) ++(0,-0.5) -- +(-60:1) +(0,0) -- +(-120:0.5);
			\draw[line width=3,fill,red] (0,0) ++(5:2) ++(-10:1.5) ++(0,-0.5) ++(0,-0.5) -- +(-120:0.5);
			\draw[thick,fill] (0,0) ++(5:2) ++(-10:1.5) ++(0,-0.5) ++(0,-0.5) circle (2pt) node[left] {\(\mathsf{4}\)};

			\begin{scope}[yshift=0cm]
				\draw[thick,fill] (0,0) ++(5:2) ++(-10:1.5) ++(0,-1) +(-120:0.5) -- +(-120:1) circle (2pt) node[left] {\(\mathsf{5}\)} +(-60:0.5) -- +(-60:1) circle (2pt) node[right] {\(\mathsf{6}\)};
				\draw[thick,fill,red] (0,0) ++(5:2) ++(-10:1.5) ++(0,-0.5) ++(0,-0.5) ++(-60:0.5) circle (2pt);
			\end{scope}
		\end{scope}

		\begin{scope}[xshift=0.2cm]
			\draw[thick,fill,blue] (0,0) ++(5:2) ++(-10:1.5) ++(10:0.75) circle(2pt) -- ++(10:0.75);
			\draw[thick,fill] (0,0) ++(5:2) ++(-10:1.5) ++(10:0.75) ++(10:0.75) node(F) {} node[below=0.5ex] {\(\mathsf{7}\)} -- +(120:0.5) circle (2pt) +(0,0) -- +(60:0.5) circle (2pt) +(0,0) -- +(-5:0.5) circle (2pt);
			\grape[-90]{F};
			\draw[line width=3, fill, red] (F.center) -- +(60:0.3);
			\draw[thick, fill, red] (F.center) ++(120:0.3) circle (2pt);
			\draw[thick, fill] (F.center) circle (2pt);

			\begin{scope}[xshift=-0.2cm,yshift=0.2cm]
				\draw[thick,fill,blue] (0,0) ++(5:2) ++(-10:1.5) ++(10:1.5) +(120:0.5) circle (2pt) -- +(120:1);
				\draw[thick,fill] (0,0) ++(5:2) ++(-10:1.5) ++(10:1.5) +(120:0.5) +(120:1) circle (2pt) node(G) {} node[xshift=-1.2ex,yshift=2ex] {\(\mathsf{12}\)};
				\grape[120]{G};
			\end{scope}
			\begin{scope}[xshift=0.2cm,yshift=0.2cm]
				\draw[thick,fill,blue] (0,0) ++(5:2) ++(-10:1.5) ++(10:1.5) +(60:0.5) circle (2pt) -- +(60:1);
				\draw[thick,fill] (0,0) ++(5:2) ++(-10:1.5) ++(10:1.5) +(60:1) node(H) {} node[xshift=1.2ex,yshift=2ex] {\(\mathsf{11}\)};
				\grape[60]{H};
				\draw[line width=3, fill, red] (H.center) -- ++(90:0.5);
				\draw[thick,fill] (H.center) circle (2pt);
			\end{scope}

			\begin{scope}[xshift=0.2cm]
				\draw[thick,fill,blue] (0,0) ++(5:2) ++(-10:1.5) ++(10:1.5) ++(-5:0.5) circle (2pt) -- ++(-5:0.5);
				\draw[thick,fill] (0,0) ++(5:2) ++(-10:1.5) ++(10:1.5) ++(-5:0.5) ++(-5:0.5) node(I) {} node[below] {\(\mathsf{8}\)} -- +(30:0.5) +(0,0) -- +(-30:0.5) circle (2pt);
				\draw[line width=3, fill, red] (I.center) -- ++(-30:0.3);
				\draw[thick, fill] (I.center) circle (2pt);

				\begin{scope}[xshift=0cm]
					\draw[thick,fill] (0,0) ++(5:2) ++(-10:1.5) ++(10:1.5) ++(-5:1) +(30:0.5) -- +(30:1) circle (2pt) node[right] {\(\mathsf{10}\)};
					\draw[thick, fill, red] (I.center) +(30:0.75) circle (2pt);
				\end{scope}
				\begin{scope}[xshift=0.2cm,yshift=-0.2cm]
					\draw[thick,fill,blue] (0,0) ++(5:2) ++(-10:1.5) ++(10:1.5) ++(-5:1) +(-30:0.5) circle (2pt) -- +(-30:1);
					\draw[thick,fill] (0,0) ++(5:2) ++(-10:1.5) ++(10:1.5) ++(-5:1) +(-30:0.5) +(-30:1) node(K) {} node[above] {\(\mathsf{9}\)};
					\grape[-30]{K};
					\draw[line width=3, fill, red] (K.center) -- +(0:0.5);
					\draw[thick, fill] (K.center) circle (2pt);
				\end{scope}
			\end{scope}
		\end{scope}
	\end{scope}
\end{scope}
\end{tikzpicture}
\)
}
\subcaptionbox{The corresponding SHE-configuration on \(\Gamma\).\label{fig:example_conf2}}[.8\textwidth]{
\(
\begin{tikzpicture}[baseline=-.5ex]
\draw[thick,fill] (0,0) circle (2pt) node (A) {} -- ++(5:2) node[midway,above] {\(\sfe^\mathsf{1}\)} circle(2pt) node(B) {};
\draw[thick,fill] (B.center) +(0,1.5) circle (2pt) node (C) {};
\draw[thick,fill,<-] (B.center)+(0,2pt) -- (C.center);
\draw[thick,fill] (B.center) -- +(0,-1.5) circle (2pt) node(D) {} +(0,0) -- ++(-10:1.5) node[midway,above] {\(\sfe^\mathsf{3}\)} circle(2pt) node(E) {} -- ++(10:1.5) node[midway,below] {\(\sfe^\mathsf{7}\)} circle(2pt) node(F) {} -- +(120:1) node[midway,left] {\(\sfe^{\mathsf{12}}\)} circle(2pt) node(G) {} +(0,0) -- +(60:1) node[midway,right] {\(\sfe^{\mathsf{11}}\)} circle (2pt) node(H) {} +(0,0) -- ++(-5:1) circle (2pt) node(I) {} -- +(30:1) node[midway,above] {\(\sfe^{\mathsf{10}}\)} circle(2pt) node(J) {} +(0,0) -- +(-30:1) node[midway,below] {\(\sfe^{\mathsf{9}}\)} circle (2pt) node(K) {};
\draw[thick, fill] (E.center) -- ++(0,-1) node[midway,right] {\(\sfe^4\)} circle (2pt) -- +(-60:1) node[midway,right] {\(\sfe^6\)} circle (2pt) +(0,0) -- +(-120:1) node[midway,left] {\(\sfe^5\)} circle (2pt);
\grape[180]{A};
\grape[0]{C}; \grape[90]{C}; \grape[180]{C};
\grape[-90]{D};
\grape[-90]{F};
\grape[120]{G};
\grape[60]{H};
\grape[-30]{K};

% dots
\draw[thick,fill,red] 
(B.center) ++(-10:0.5) circle(2pt)
(C.center) +(90:{1.5/sqrt(3)}) circle (2pt) +(180:{1.5/sqrt(3)}) circle (2pt)
(E.center) +(10:0.5) circle (2pt)
(E.center) ++(0,-1) ++(-60:0.5) circle (2pt)
(F.center) ++(120:0.3) circle (2pt)
(I.center) +(30:0.75) circle (2pt)
;

%half edges
\draw[line width=3,fill,red] 
(B.center) -- +(5:-0.5)
(C.center) -- +(30:0.5)
(E.center) -- +(0,-0.3)
(E.center) ++(0,-1) -- +(-120:0.5)
(F.center) -- +(60:0.3)
(H.center) -- ++(90:0.5)
(I.center) -- ++(-30:0.5)
(K.center) -- +(0:0.5);

\draw[thick, fill] 
(B.center) circle (2pt)
(C.center) circle (2pt)
(E.center) circle (2pt)
(E.center) ++(0,-1) circle (2pt)
(F.center) circle (2pt)
(H.center) circle (2pt)
(I.center) circle (2pt)
(K.center) circle (2pt);

\draw (C.center)+(0:{1.5/sqrt(3)}) node[right] {\(\sfe_2^{\mathsf{13}}\)};
\draw (C.center)+(90:{1.5/sqrt(3)}) node[above] {\(\sfe_3^{\mathsf{13}}\)};
\draw (C.center)+(180:{1.5/sqrt(3)}) node[left] {\(\sfe_4^{\mathsf{13}}\)};
\draw (H.center)+(60:{1.5/sqrt(3)}) node[above] {\(\sfe_2^{\mathsf{11}}\)};
\draw (K.center)+(-30:{1.5/sqrt(3)}) node[right] {\(\sfe_2^{\mathsf{9}}\)};
\end{tikzpicture}
\)
}
  \caption{
An example for
\(i=8, j=15,  W=\{\mathsf{0},\mathsf{3},\mathsf{4},\mathsf{7},\mathsf{8},\mathsf{9},\mathsf{11},\mathsf{13}\}\subseteq V^{\ess}(\Gamma)\), \(j_{\mathsf{0}}=j_{\mathsf{3}}=j_{\mathsf{4}}=j_{\mathsf{7}}=j_{\mathsf{8}}=2\), \( j_{\mathsf{9}}=j_{\mathsf{11}}=1\), and \(j_{\mathsf{13}}=3 \).}
  \label{fig:image20}
\end{figure}

Let \(E_0\) be the set of edges of \(\Gamma\) that contain a dot or a
half edge. For the SHE-configuration \( \Gamma \) in \Cref{fig:image20}, we have
\[
E_0=\{\sfe^{\mathsf{1}}, \sfe^{\mathsf{3}}, \sfe^{\mathsf{4}}, \sfe^{\mathsf{5}}, \sfe^{\mathsf{6}}, \sfe^{\mathsf{7}}, \sfe^{\mathsf{9}}, \sfe_2^{\mathsf{9}}, \sfe^{\mathsf{10}}, \sfe^{\mathsf{11}}, \sfe_2^{\mathsf{11}}, \sfe^{\mathsf{12}}, \sfe_2^{\mathsf{13}}, \sfe_3^{\mathsf{13}}, \sfe_4^{\mathsf{13}}\}.
\]
For a given vector \(\vec c\in \bbZ_{\ge 1}^j\) with
\(\norm{\vec c}\le k\), we assume that the entries of \(\vec c\) are
indexed by the edges in \(E_0\). We put \(c_\sfe-1\) more dots on the edge
\(\sfe\) for each \(\sfe\in E_0\), and put \(k-\norm{\vec c}\) more
dots on the root edge \(\sfe_0\).

The resulting configuration of dots and half edges is the
visualization of the given HE-configuration \(\mathscr{X}\). For the
example shown in \Cref{fig:image21}, we have \(k=31\) and the sum of
all entries in \(\vec c\) is \(27\), and therefore four more dots are
added on the root edge.

\begin{figure}
\[
\begin{tikzpicture}[baseline=-.5ex]
\draw[thick,fill] (0,0) circle (2pt) node (A) {} -- ++(5:2) circle(2pt) node(B) {};
\draw[thick,fill] (B.center) +(0,1.5) circle (2pt) node (C) {};
\draw[thick,fill,<-] (B.center)+(0,2pt) -- (C.center);
\draw[thick,fill] (B.center) -- +(0,-1.5) circle (2pt) node(D) {} +(0,0) -- ++(-10:1.5) circle(2pt) node(E) {} -- ++(10:1.5) circle(2pt) node(F) {} -- +(120:1) circle(2pt) node(G) {} +(0,0) -- +(60:1) circle (2pt) node(H) {} +(0,0) -- ++(-5:1) circle (2pt) node(I) {} -- +(30:1) circle(2pt) node(J) {} +(0,0) -- +(-30:1) circle (2pt) node(K) {};
\draw[thick, fill] (E.center) -- ++(0,-1) circle (2pt) -- +(-60:1) circle (2pt) +(0,0) -- +(-120:1) circle (2pt);
\grape[180]{A};
\grape[0]{C}; \grape[90]{C}; \grape[180]{C};
\grape[-90]{D};
\grape[-90]{F};
\grape[120]{G};
\grape[60]{H};
\grape[-30]{K};

% dots
\draw[thick,fill,red] 
(A.center) +(5:1) circle(2pt) +(5:0.75) circle (2pt) +(5:0.5) circle (2pt)
(B.center) ++(-10:0.75) circle(2pt)
(C.center) +(90:{1.5/sqrt(3)}) circle (2pt) ++(180:{1/sqrt(3)}) +(150:{0.5/sqrt(3)}) circle (2pt) +(210:{0.5/sqrt(3)}) circle (2pt)
(E.center) +(10:0.75) circle (2pt) +(0,-0.5) circle (2pt) +(0,-0.75) circle (2pt)
(E.center) ++(0,-1) ++(-60:0.75) circle (2pt)
(F.center) ++(120:0.5) circle (2pt)
(H.center) ++(60:{1/sqrt(3)}) +(0:{0.5/sqrt(3)}) circle (2pt) +(60:{0.5/sqrt(3)}) circle (2pt) +(120:{0.5/sqrt(3)}) circle (2pt)
(I.center) +(30:0.75) circle (2pt) +(30:0.5) circle (2pt)
(K.center) ++(-30:{1/sqrt(3)}) +(-60:{0.5/sqrt(3)}) circle (2pt) +(0:{0.5/sqrt(3)}) circle (2pt);
% additional dots
\draw[thick,fill,red] 
(B.center) +(90:0.4) circle(2pt) +(90:0.6) circle(2pt) +(90:0.8) circle(2pt) +(90:1) circle(2pt)
;

%half edges
\draw[line width=3,fill,red] 
(B.center) -- +(5:-0.5)
(C.center) -- +(30:0.5)
(E.center) -- +(0,-0.3)
(E.center) ++(0,-1) -- +(-120:0.5)
(F.center) -- +(60:0.3)
(H.center) -- ++(90:0.5)
(I.center) -- ++(-30:0.5)
(K.center) -- +(0:0.5);

\draw[thick, fill] 
(B.center) circle (2pt)
(C.center) circle (2pt)
(E.center) circle (2pt)
(E.center) ++(0,-1) circle (2pt)
(F.center) circle (2pt)
(H.center) circle (2pt)
(I.center) circle (2pt)
(K.center) circle (2pt);
\end{tikzpicture}
\]
  \caption{An example of an HE-configuration for \(k=31\) and \(
\vec c =(4,1,3,1,1,1,1,3,2,1,4,1,1,1,2)
\). }
  \label{fig:image21}
\end{figure}

\begin{remark}
In particular, if \(i=0\), then
\(\HE^0_\Gamma(j,k)=\begin{cases}\{*\} & j=0;\\\varnothing & j\ge
  1,\end{cases}\) and so \(\HE^0_\Gamma(k)=\HE^0_\Gamma(0,k)=\{*\}\),
which corresponds to the unique HE-configuration consisting of \(k\)
dots on the root edge.
Similarly, if \(k=0\), then for any SHE-configuration with \(j>0\),
there are no corresponding HE-configurations, that is,
\(\HE^i_\Gamma(0)=\HE^i_\Gamma(0,0)\). Since
\(\SHE_{\ell(\sfv),m(\sfv)}(0)=\varnothing\) for every \(\sfv\), we
have \(\HE^i_\Gamma(0,0)\neq\varnothing\) if and only if \(i=0\), in
which case \( \HE^0_\Gamma(0,0) \) has the unique element \(*\) as well.
\end{remark}

In summary, we obtain the following combinatorial interpretations for
\( P^i_{\Gamma}(k) \) and its coefficients.

\begin{thm}\label{thm:1}
  Let \( \Gamma\cong(\sfT,\loops)\in \Grape \). Suppose that \(\sfT\)
  has at least one edge and let \((\sfv_0,\sfe_0)\) be the oriented
  root. For a nonnegative integer \( i \), we have
\[
P_{\Gamma}^i(k) =|\HE_\Gamma^i(k)|=\sum_{j\ge 0} |\SHE_\Gamma^i(j)|\binom k j .
\]
In other words, the following hold.
  \begin{enumerate}
  \item For a nonnegative integer \( k \), \( P^i_{\Gamma}(k) \) is
    the number of ways to place \( (k-i) \) dots and mark \(i\) half
    edges in \( \Gamma \) satisfying the following
    conditions:
    \begin{enumerate}
    \item For every essential vertex \( \sfv \), if there is a dot in an
      edge other than the pivot edge \( \sfe^{\sfv}_1 \) in \( \Gamma_\sfv \),
      then there is exactly one marked half edge in \( \Gamma_\sfv \) and
      the edge containing it must have a dot; otherwise,
      \( \Gamma_\sfv \) has no marked half edge. 
    \item If the marked half edge is the second half edge
      \( \sfh^2(\sfe^\sfv_{a}) \) of an edge \( \sfe^\sfv_{a} \), then
      there is a dot in an edge \( \sfe^\sfv_{b} \) for some
      \( b>a \).
    \end{enumerate}
  \item Considering \( P^i_{\Gamma}(k) \) as a polynomial in \( k \),
    its coefficient of \( \binom{k}{j} \) is the number of ways to
    place \( (j-i) \) dots and mark \(i\) half edges in \( \Gamma \)
    satisfying the following conditions:
	\begin{enumerate}
	\item Every edge has at most one dot, and the root edge \(\sfe_0\) has no dots.
	\item For each edge \(\sfe\) with a dot, the vertex adjacent to \(\sfe\) closer to the root vertex \(\sfv_0\) has a marked half edge attached to it.
    \item If a marked half edge \(\sfh\) attached to \(\sfv\) is the first half edge \( \sfh^1(\sfe^\sfv_{r}) \) of an edge
      \( \sfe^\sfv_{r} \) in \(\Gamma_\sfv\), then there is a dot in an edge \( \sfe^\sfv_{s} \) for some \( s>r \).
	\end{enumerate}
  \end{enumerate}
\end{thm}

\section{Homology cycles and HE-configurations}\label{homology cycles}
In this section, we will describe how each HE-configuration in \(\HE^i_\Gamma(k)\) can be identified with a homology cycle in \(H_i(B_k\Gamma)\) for a bunch of grapes \(\Gamma\).

\subsection{Star and loop classes}
For a general graph \(\Gamma\), according to
\cite[Section~5]{ADCK2019}, there are two types of \(1\)-cycles,
called star and loop classes, generating \(H_1(B\Gamma)\) as a
\(\mathbb{Z}[E]\)-module. To describe these classes, we need some
definitions.

We first consider generators of
\(H_1(B_2\Gamma_{0,3})\cong\mathbb{Z}\) and
\(H_1(B_1\Gamma_{1,0})\cong\mathbb{Z}\). Since \(\Gamma_{1,0}\)
is homeomorphic to a circle, it has a generator \(\beta_1\),
which is uniquely determined up to orientation. We represent
\( \beta_1 \) by a thick red loop with a small oriented gray circle at
the vertex as follows:
\begin{align*}
\beta_1 = \left[
\begin{tikzpicture}[baseline=-.5ex]
\begin{scope}[yshift=-.5cm]
\draw[thick, fill] (0,0) circle (2pt) node (A) {};
\grape[90]{A};
\draw[thick, fill, red] (120:0.5) circle (2pt);
\draw[thick, red, rounded corners] (0,{1/sqrt(3)}) ++ (-150:{0.8/sqrt(3)}) -- ++(-60:0.8) -- ++(60:0.8);
\draw[thick, red, ->] (0,{1/sqrt(3)}) ++ (-30:{0.8/sqrt(3)}) arc (-30:210:{0.8/sqrt(3)});
\end{scope}
\end{tikzpicture}
\right]
=
\begin{tikzpicture}[baseline=-.5ex]
\begin{scope}[yshift=-.5cm]
\draw[line width=3, red] (0,0) -- (120:0.5) (0,0) -- (60:0.5) (120:0.5) arc (210:-30:{0.5/sqrt(3)});
\draw[thick, fill] (0,0) circle (2pt);
\draw[thick, ->, gray] (0.2,0) arc (0:360:0.2);
\end{scope}
\end{tikzpicture}
\in H_1(B_1\Gamma_{1,0}).
\end{align*}

On the other hand, \(B_2\Gamma_{0,3}\) has a generator
\(\alpha_{123}\) that can be represented by a loop \(\gamma_{123}\)
as follows:
\begin{align*}
\gamma_{123} &= 
\left(
\begin{tikzpicture}[baseline=-.5ex]
\draw[thick, fill] (0,0) -- (-90:1) circle (2pt) node[below] {\(\sfv_1\)};
\draw[thick, fill] (0,0) circle (2pt) -- (30:1) circle (2pt) node[above right] {\(\sfv_2\)} (0,0) -- (150:1) circle (2pt) node[above left] {\(\sfv_3\)};
\draw[thick, fill, red] (150:0.5) circle (2pt) (30:0.5) circle (2pt);
\draw[thick, densely dotted, red] (-90:0.5) circle (2pt);
\draw[thick, red, ->] ({-0.5*sqrt(3)}, -0.5) ++ (60:0.6) arc (60:0:0.6);
\end{tikzpicture}
\right)
\cdot
\left(
\begin{tikzpicture}[baseline=-.5ex]
\draw[thick, fill] (0,0) -- (-90:1) circle (2pt) node[below] {\(\sfv_1\)};
\draw[thick, fill] (0,0) circle (2pt) -- (30:1) circle (2pt) node[above right] {\(\sfv_2\)} (0,0) -- (150:1) circle (2pt) node[above left] {\(\sfv_3\)};
\draw[thick, fill, red] (30:0.5) circle (2pt) (-90:0.5) circle (2pt);
\draw[thick, densely dotted, red] (150:0.5) circle (2pt);
\draw[thick, red, ->] (0,1) ++ (-60:0.6) arc (-60:-120:0.6);
\end{tikzpicture}
\right)
\cdot
\left(
\begin{tikzpicture}[baseline=-.5ex]
\draw[thick, fill] (0,0) -- (-90:1) circle (2pt) node[below] {\(\sfv_1\)};
\draw[thick, fill] (0,0) circle (2pt) -- (30:1) circle (2pt) node[above right] {\(\sfv_2\)} (0,0) -- (150:1) circle (2pt) node[above left] {\(\sfv_3\)};
\draw[thick, fill, red] (-90:0.5) circle (2pt) (150:0.5) circle (2pt);
\draw[thick, densely dotted, red] (30:0.5) circle (2pt);
\draw[thick, red, ->] ({0.5*sqrt(3)}, -0.5) ++ (-0.6,0) arc (180:120:0.6);
\end{tikzpicture}
\right).
\end{align*}
By permuting the vertices, we can consider \(\gamma_{ijk}\) and its
homology class \(\alpha_{ijk}=[\gamma_{ijk}]\) for every triple of distinct
integers \(i,j,k\) with \(1\le i,j,k\le 3\). Then it is easy to see
that
\begin{equation}\label{eq:star class}
\alpha_{123} = \alpha_{231} = \alpha_{312} = -\alpha_{132} = -\alpha_{321} = -\alpha_{213}.
\end{equation}
In particular, each \(\mathbb{Z}\langle\alpha_{ijk}\rangle\) is
isomorphic to \(H_1(B_2\Gamma_{0,3})\). Hence we may denote the
generator \(\alpha_{123}\) simply by three half edges with a
small oriented circle at the vertex as follows:
\[
\alpha_{123}=\begin{tikzpicture}[baseline=-.5ex]
\draw[thick, fill] (0,0) -- (-90:1) circle (2pt);
\draw[thick, fill] (0,0) circle (2pt) -- (30:1) circle (2pt) (0,0) -- (150:1) circle (2pt);
\draw[line width=3, red] (0,0) -- (30:0.5);
\draw[line width=3, red] (0,0) -- (150:0.5);
\draw[line width=3, red] (0,0) -- (-90:0.5);
\draw[thick, fill] (0,0) circle (2pt);
\draw[thick, ->, gray] (0.2,0) arc (0:360:0.2);
\end{tikzpicture}.
\]

Now let \(\iota:\Gamma_{0,3}\to\Gamma\) and
\(j:\Gamma_{1,0}\to\Gamma\) be some topological embeddings into a graph
\(\Gamma=(V,E)\). Let \(\alpha,\beta\in H_1(B_k\Gamma)\) be suitably
stabilized images of \(\alpha_{123}\) and \(\beta_1\) under the
induced maps \(\iota_*\) and \(j_*\), respectively. Here,
``stabilized'' means that a monomial \( p \in \bbZ[E] \) is applied.
More precisely, \(\alpha=p \cdot \iota_*(\alpha_{123})\) and
\(\beta=q \cdot j_*(\beta_1)\) for some monomials \(p, q\in\bbZ[E]\)
of degree \(k-2\) and \(k-1\). Such \( \alpha \) and \( \beta \) are
called a \emph{star} class and a \emph{loop} class in
\(H_1(B_k\Gamma)\), respectively. Examples of star and loop classes
in an elementary bunch of grapes are depicted in \Cref{figure:star and
  loop classes}. It is important to note that the edges intersecting with the
image of \(\iota(\Gamma_{0,3})\) or \(j(\Gamma_{1,0})\) can also be
stabilized.

\begin{figure}
\subcaptionbox{Star classes with stabilization \(p=\sfe_1^2\sfe_3^2\sfe_5^2\sfe_6\sfe_7\).}{
\begin{tikzpicture}[baseline=-.5ex]
\draw[thick, fill, blue] 
(0,0) -- (-36:2) circle(2pt) node[below right] {\(*\)};
\draw[thick, fill] 
(0,0) -- (-12:2) circle(2pt) 
(0,0) -- (12:2) circle(2pt) 
(0,0) -- (36:2) circle (2pt);
\draw[thick] 
(0,0) -- (60:1) arc (-30:210:{1/sqrt(3)}) -- (0,0)
(0,0) -- (150:1) arc (60:300:{1/sqrt(3)}) -- (0,0)
(0,0) -- (240:1) arc (150:390:{1/sqrt(3)}) -- (0,0);
\draw[line width=3, red] 
(0,0) -- (12:1)
(0,0) -- (-12:1)
(0,0) -- (-36:1)
;
\draw[thick, fill] (0,0) circle (2pt);
\draw[thick, green, fill] 
(-36:1.33) circle (2pt) (-36:1.67) circle (2pt)
(-90:{3/sqrt(3)}) circle (2pt)
(12:1.5) circle (2pt) (12:1.75) circle (2pt)
(90:{2/sqrt(3)}) + (60:{1/sqrt(3)}) circle (2pt) + (120:{1/sqrt(3)}) circle (2pt)
(180:{3/sqrt(3)}) circle (2pt)
;
\draw[thick, ->, gray] (-0.3,0) arc (180:540:0.3);
\end{tikzpicture}
\qquad
\begin{tikzpicture}[baseline=-.5ex]
\draw[thick, fill, blue] 
(0,0) -- (-36:2) circle(2pt) node[below right] {\(*\)};
\draw[thick, fill] 
(0,0) -- (-12:2) circle(2pt) 
(0,0) -- (12:2) circle(2pt) 
(0,0) -- (36:2) circle (2pt);
\draw[thick] 
(0,0) -- (60:1) arc (-30:210:{1/sqrt(3)}) -- (0,0)
(0,0) -- (150:1) arc (60:300:{1/sqrt(3)}) -- (0,0)
(0,0) -- (240:1) arc (150:390:{1/sqrt(3)}) -- (0,0);
\draw[line width=3, red] 
(0,0) -- (36:1)
(0,0) -- (210:1)
(0,0) -- (120:1)
;
\draw[thick, fill] (0,0) circle (2pt);
\draw[thick, green, fill] 
(-36:1.33) circle (2pt) (-36:1.67) circle (2pt)
(12:1.33) circle (2pt) (12:1.67) circle (2pt)
(90:{2/sqrt(3)}) + (60:{1/sqrt(3)}) circle (2pt) + (120:{1/sqrt(3)}) circle (2pt)
(180:{3/sqrt(3)}) circle (2pt)
(270:{3/sqrt(3)}) circle (2pt)
;
\draw[thick, <-, gray] (-0.3,0) arc (180:540:0.3);
\end{tikzpicture}
}

\subcaptionbox{Loop classes with stabilization \(q=\sfe_1\sfe_3^2\sfe_4^2\sfe_6^3\sfe_7\).}{
\begin{tikzpicture}[baseline=-.5ex]
\draw[thick, fill, blue] 
(0,0) -- (-36:2) circle(2pt) node[below right] {\(*\)};
\draw[thick, fill] 
(0,0) -- (-12:2) circle(2pt) 
(0,0) -- (12:2) circle(2pt) 
(0,0) -- (36:2) circle (2pt);
\draw[thick] 
(0,0) -- (60:1) arc (-30:210:{1/sqrt(3)}) -- (0,0)
(0,0) -- (150:1) arc (60:300:{1/sqrt(3)}) -- (0,0)
(0,0) -- (240:1) arc (150:390:{1/sqrt(3)}) -- (0,0);
\draw[line width=3, red] 
(0,0) -- (120:1) arc (210:-30:{1/sqrt(3)}) -- cycle;
\draw[thick, fill] (0,0) circle (2pt);
\draw[thick, green, fill] 
(-36:1.5) circle (2pt)
(12:1.33) circle (2pt) (12:1.67) circle (2pt)
(36:1.33) circle (2pt) (36:1.67) circle (2pt)
(180:{2/sqrt(3)}) + (90:{1/sqrt(3)}) circle (2pt) + (270:{1/sqrt(3)}) circle (2pt) + (180:{1/sqrt(3)}) circle (2pt)
(270:{3/sqrt(3)}) circle (2pt)
;
\draw[thick, ->, gray] (-0.3,0) arc (180:540:0.3);
\end{tikzpicture}
\qquad
\begin{tikzpicture}[baseline=-.5ex]
\draw[thick, fill, blue] 
(0,0) -- (-36:2) circle(2pt) node[below right] {\(*\)};
\draw[thick, fill] 
(0,0) -- (-12:2) circle(2pt) 
(0,0) -- (12:2) circle(2pt) 
(0,0) -- (36:2) circle (2pt);
\draw[thick] 
(0,0) -- (60:1) arc (-30:210:{1/sqrt(3)}) -- (0,0)
(0,0) -- (150:1) arc (60:300:{1/sqrt(3)}) -- (0,0)
(0,0) -- (240:1) arc (150:390:{1/sqrt(3)}) -- (0,0);
\draw[line width=3, red] 
(0,0) -- (300:1) arc (390:150:{1/sqrt(3)}) -- cycle;
\draw[thick, fill] (0,0) circle (2pt);
\draw[thick, green, fill] 
(-36:1.5) circle (2pt)
(12:1.33) circle (2pt) (12:1.67) circle (2pt)
(36:1.33) circle (2pt) (36:1.67) circle (2pt)
(180:{2/sqrt(3)}) + (90:{1/sqrt(3)}) circle (2pt) + (270:{1/sqrt(3)}) circle (2pt) + (180:{1/sqrt(3)}) circle (2pt)
(270:{3/sqrt(3)}) circle (2pt)
;
\draw[thick, ->, gray] (-0.3,0) arc (180:540:0.3);
\end{tikzpicture}
}
\caption{Star and loop classes in \(H_1(B_{10}\Gamma_{3,4})\). Here, stabilizations are denoted by green dots.}
\label{figure:star and loop classes}
\end{figure}

\begin{thm}[\cite{ADCK2019}]\label{thm:SL-gen}
  The star classes and loop classes generate \(H_1(B\Gamma)\) as an
  \(\mathbb{F}[E]\)-module.
\end{thm}

The relations among star classes and loop classes are also
given in \cite{ADCK2019}. There are three useful relations coming from
\(H_1(B_1\Gamma_{0,4})\) and \(H_1(B_1\Gamma_{1,1})\), called the
\emph{\(X\)-relations} (\( X_1 \)), the \emph{stabilized
  \(X\)-relations} (\( X_2 \)), and the \emph{\(Q\)-relation}
\( (Q) \) as follows:
\begin{align}
\begin{tikzpicture}[baseline=-.5ex]
\draw[thick, fill] (-1,0) circle (2pt) -- (1,0) circle (2pt) (0,-1) circle (2pt) -- (0,1) circle (2pt);
\draw[line width=3, red] (0,0) -- (0.5,0) (0,0) -- (0,0.5) (0,0) -- (-0.5, 0);
\draw[thick, ->, gray] (-0.3,0) arc (180:540:0.3);
\draw[thick, fill] (0,0) circle (2pt);
\end{tikzpicture}
+
\begin{tikzpicture}[baseline=-.5ex]
\draw[thick, fill] (-1,0) circle (2pt) -- (1,0) circle (2pt) (0,-1) circle (2pt) -- (0,1) circle (2pt);
\draw[line width=3, red] (0,0) -- (0.5,0) (0,0) -- (0,0.5) (0,0) -- (0,-0.5);
\draw[thick, <-, gray] (-0.3,0) arc (180:540:0.3);
\draw[thick, fill] (0,0) circle (2pt);
\end{tikzpicture}
+
\begin{tikzpicture}[baseline=-.5ex]
\draw[thick, fill] (-1,0) circle (2pt) -- (1,0) circle (2pt) (0,-1) circle (2pt) -- (0,1) circle (2pt);
\draw[line width=3, red] (0,0) -- (0.5,0) (0,0) -- (0,-0.5) (0,0) -- (-0.5, 0);
\draw[thick, ->, gray] (-0.3,0) arc (180:540:0.3);
\draw[thick, fill] (0,0) circle (2pt);
\end{tikzpicture}
+
\begin{tikzpicture}[baseline=-.5ex]
\draw[thick, fill] (-1,0) circle (2pt) -- (1,0) circle (2pt) (0,-1) circle (2pt) -- (0,1) circle (2pt);
\draw[line width=3, red] (0,0) -- (0,-0.5) (0,0) -- (0,0.5) (0,0) -- (-0.5, 0);
\draw[thick, <-, gray] (-0.3,0) arc (180:540:0.3);
\draw[thick, fill] (0,0) circle (2pt);
\end{tikzpicture}
&=0 \in H_1(B_2\Gamma_{0,4}),
\tag{\(X_1\)}\label{eq:X_1}
\\
\begin{tikzpicture}[baseline=-.5ex]
\draw[thick, fill] (-1,0) circle (2pt) -- (1,0) circle (2pt) (0,-1) circle (2pt) -- (0,1) circle (2pt);
\draw[line width=3, red] (0,0) -- (0.5,0) (0,0) -- (0,0.5) (0,0) -- (-0.5, 0);
\draw[thick, ->, gray] (-0.3,0) arc (180:540:0.3);
\draw[thick, fill, green] (0,-0.67) circle (2pt);
\draw[thick, fill] (0,0) circle (2pt);
\end{tikzpicture}
+
\begin{tikzpicture}[baseline=-.5ex]
\draw[thick, fill] (-1,0) circle (2pt) -- (1,0) circle (2pt) (0,-1) circle (2pt) -- (0,1) circle (2pt);
\draw[line width=3, red] (0,0) -- (0.5,0) (0,0) -- (0,0.5) (0,0) -- (0,-0.5);
\draw[thick, <-, gray] (-0.3,0) arc (180:540:0.3);
\draw[thick, fill, green] (-0.67,0) circle (2pt);
\draw[thick, fill] (0,0) circle (2pt);
\end{tikzpicture}
+
\begin{tikzpicture}[baseline=-.5ex]
\draw[thick, fill] (-1,0) circle (2pt) -- (1,0) circle (2pt) (0,-1) circle (2pt) -- (0,1) circle (2pt);
\draw[line width=3, red] (0,0) -- (0.5,0) (0,0) -- (0,-0.5) (0,0) -- (-0.5, 0);
\draw[thick, ->, gray] (-0.3,0) arc (180:540:0.3);
\draw[thick, fill, green] (0,0.67) circle (2pt);
\draw[thick, fill] (0,0) circle (2pt);
\end{tikzpicture}
+
\begin{tikzpicture}[baseline=-.5ex]
\draw[thick, fill] (-1,0) circle (2pt) -- (1,0) circle (2pt) (0,-1) circle (2pt) -- (0,1) circle (2pt);
\draw[line width=3, red] (0,0) -- (0,-0.5) (0,0) -- (0,0.5) (0,0) -- (-0.5, 0);
\draw[thick, <-, gray] (-0.3,0) arc (180:540:0.3);
\draw[thick, fill, green] (0.67,0) circle (2pt);
\draw[thick, fill] (0,0) circle (2pt);
\end{tikzpicture}
&=0 \in H_1(B_3\Gamma_{0,4}),
\tag{\(X_2\)}\label{eq:X_2}
\\
\begin{tikzpicture}[baseline=-.5ex]
\draw[thick, fill] (0,0) node (A) {} -- (1,0) circle (2pt);
\grape[180]{A};
\draw[line width=3, red] (0,0) -- (210:0.5) (0,0) -- (150:0.5) (210:0.5) arc (300:60:{0.5/sqrt(3)});
\draw[thick, ->, gray] (-0.3,0) arc (180:540:0.3);
\draw[thick, fill, green] (0.67,0) circle (2pt);
\draw[thick, fill] (0,0) circle (2pt);
\end{tikzpicture}
+
\begin{tikzpicture}[baseline=-.5ex]
\draw[thick, fill] (0,0) node (A) {} -- (1,0) circle (2pt);
\grape[180]{A};
\draw[line width=3, red] (0,0) -- (210:0.5) (0,0) -- (150:0.5) (210:0.5) arc (300:60:{0.5/sqrt(3)});
\draw[thick, <-, gray] (-0.3,0) arc (180:540:0.3);
\draw[thick, fill, green] ({-1.5/sqrt(3)},0) circle (2pt);
\draw[thick, fill] (0,0) circle (2pt);
\end{tikzpicture}
+
\begin{tikzpicture}[baseline=-.5ex]
\draw[thick, fill] (0,0) node (A) {} -- (1,0) circle (2pt);
\grape[180]{A};
\draw[line width=3, red] (0,0) -- (210:0.5) (0,0) -- (150:0.5) (0,0) -- (0.5,0);
\draw[thick, ->, gray] (-0.3,0) arc (180:540:0.3);
\draw[thick, fill] (0,0) circle (2pt);
\end{tikzpicture}
&=0 \in H_1(B_2\Gamma_{1,1}).
\tag{\(Q\)}\label{eq:Q}
\end{align}

Recall that given an embedding of a bunch of graphs \( \Gamma \) and
its essential vertex \( \sfv \), the edges of \( \Gamma_\sfv \) are
labeled by \( \sfe^{\sfv}_1,\dots,\sfe^{\sfv}_{\ell(\sfv)+m(\sfv)} \)
in counterclockwise order, and \( \sfe^{\sfv}_1 \) is the pivot edge
of \( \Gamma_\sfv \).

\begin{defn}[standard star classes]
  Let \(\Gamma\) be a bunch of grapes. We fix a planar embedding of
  \(\Gamma\) along with its oriented root. A star class
  \(\alpha\in H_1(B_k\Gamma_{\ell,m})\) is \emph{standard} if there
  exist an embedding \(\iota:\Gamma_{0,3}\to\Gamma\) and a monic
  monomial \(p\in \mathbb{Z}[E]\) such that
  \(\alpha=p\cdot\iota_*(\alpha_{123})\) and
\begin{enumerate}
\item \(\iota(\Gamma_{0,3})\) is the union of the first half edges of
  the edges \(\sfe_1^\sfv \), \( \sfe_r^\sfv \), and \( \sfe_s^\sfv \)
  for some essential vertex \(\sfv\) and \(1<r<s\),
\item the monomial \( p \) contains no variable \(\sfe_j^\sfv\) for each \(r<j<s\).
\end{enumerate}
\end{defn}

The notion of standard star classes can be defined consistently for
any graph \(\Gamma\), but we omit the detail.

\subsection{The first Homology classes for elementary bunches of grapes}
Let us consider an elementary bunch of grapes \(\Gamma_{\ell,m}\) with \(m\ge 1\).
We denote the sets of standard star classes and loop classes in \(H_1(B_k\Gamma_{\ell,m})\) by \(S^1_{\ell,m}(k)\) and \(L^1_{\ell,m}(k)\), respectively:
\begin{align*}
S^1_{\ell,m}(k)&\coloneqq \left\{\alpha\in H_1(B_k\Gamma_{\ell,m}): \alpha\text{ is a standard star class}\right\},&
S^1_{\ell,m}&=\coprod_{k\ge 0}S^1_{\ell,m}(k),\\
L^1_{\ell,m}(k)&\coloneqq\left\{\beta\in H_1(B_k\Gamma_{\ell,m}): \beta\text{ is a loop class}\right\},&
L^1_{\ell,m}&=\coprod_{k\ge 0}L^1_{\ell,m}(k).
\end{align*}

\begin{lem}\label{lem:standard star}
The first homology group \(H_1(B\Gamma_{\ell,m})\) is generated by \(S^1_{\ell,m}\sqcup L^1_{\ell,m}\) as an \(\bbF[E]\)-module.
\end{lem}
\begin{proof}
  By \Cref{thm:SL-gen}, it suffices to show that every star class can
  be expressed as a sum of stabilized elements in
  \(S^1_{\ell,m}\sqcup L^1_{\ell,m}\). Recall that the degree of the
  central vertex in \(\Gamma_{\ell,m}\) is \(2\ell+m\).

  If \(2\ell+m\le 2\), then there are neither star classes nor loop
  classes and we have \(H_1(B_k\Gamma_{\ell,m})=0\). Hence, the
  statement holds.

  If \(2\ell+m=3\), then we have \((\ell,m)=(0,3)\) or
  \((\ell,m)=(1,1)\). If \((\ell,m)=(0,3)\), then every star class is
  standard, so we are done. If \((\ell,m)=(1,1)\), then by
  \eqref{eq:Q}, every star class can be written as the difference of
  two stabilized loop classes. Therefore the claim holds.

  Suppose now that \(2\ell+m\ge 4\) and consider a star class
  \(\alpha\in H_1(B_k\Gamma_{\ell,m})\). Using \eqref{eq:X_1}, with
  the bottom edge as the pivot edge \( \sfe_1 \), if \( \alpha \) does
  not contain \( \sfe_1 \), then we can express it in terms of star
  classes containing \( \sfe_1 \). Hence, we may assume that
  \(\alpha\) is represented by three half edges \( \sfh_1 \),
  \( \sfh_r \), and \( \sfh_s \), oriented counterclockwise, where
  \( \sfh_1 \) is the first half edge of the pivot edge \( \sfe_1 \).

  On the other hand, stabilizing the bottom edge in \eqref{eq:X_1} and
  subtracting \eqref{eq:X_2} yields the following relation:
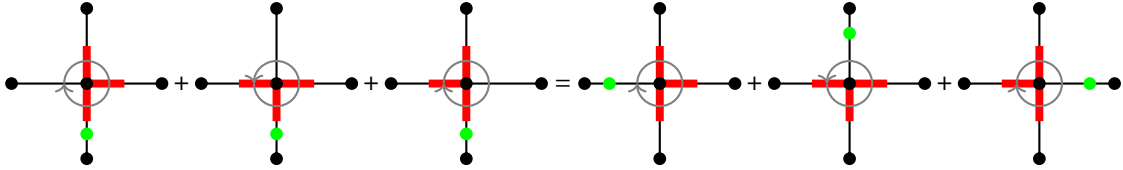
\[
\begin{tikzpicture}[baseline=-.5ex]
\draw[thick, fill] (-1,0) circle (2pt) -- (1,0) circle (2pt) (0,-1) circle (2pt) -- (0,1) circle (2pt);
\draw[line width=3, red] (0,0) -- (0.5,0) (0,0) -- (0,0.5) (0,0) -- (0,-0.5);
\draw[thick, <-, gray] (-0.3,0) arc (180:540:0.3);
\draw[thick, fill] (0,0) circle (2pt);
\draw[thick, fill, green] (0,-0.67) circle (2pt);
\end{tikzpicture}
+
\begin{tikzpicture}[baseline=-.5ex]
\draw[thick, fill] (-1,0) circle (2pt) -- (1,0) circle (2pt) (0,-1) circle (2pt) -- (0,1) circle (2pt);
\draw[line width=3, red] (0,0) -- (0.5,0) (0,0) -- (0,-0.5) (0,0) -- (-0.5, 0);
\draw[thick, ->, gray] (-0.3,0) arc (180:540:0.3);
\draw[thick, fill] (0,0) circle (2pt);
\draw[thick, fill, green] (0,-0.67) circle (2pt);
\end{tikzpicture}
+
\begin{tikzpicture}[baseline=-.5ex]
\draw[thick, fill] (-1,0) circle (2pt) -- (1,0) circle (2pt) (0,-1) circle (2pt) -- (0,1) circle (2pt);
\draw[line width=3, red] (0,0) -- (0,-0.5) (0,0) -- (0,0.5) (0,0) -- (-0.5, 0);
\draw[thick, <-, gray] (-0.3,0) arc (180:540:0.3);
\draw[thick, fill] (0,0) circle (2pt);
\draw[thick, fill, green] (0,-0.67) circle (2pt);
\end{tikzpicture}
  =
  \begin{tikzpicture}[baseline=-.5ex]
\draw[thick, fill] (-1,0) circle (2pt) -- (1,0) circle (2pt) (0,-1) circle (2pt) -- (0,1) circle (2pt);
\draw[line width=3, red] (0,0) -- (0.5,0) (0,0) -- (0,0.5) (0,0) -- (0,-0.5);
\draw[thick, <-, gray] (-0.3,0) arc (180:540:0.3);
\draw[thick, fill, green] (-0.67,0) circle (2pt);
\draw[thick, fill] (0,0) circle (2pt);
\end{tikzpicture}
+
\begin{tikzpicture}[baseline=-.5ex]
\draw[thick, fill] (-1,0) circle (2pt) -- (1,0) circle (2pt) (0,-1) circle (2pt) -- (0,1) circle (2pt);
\draw[line width=3, red] (0,0) -- (0.5,0) (0,0) -- (0,-0.5) (0,0) -- (-0.5, 0);
\draw[thick, ->, gray] (-0.3,0) arc (180:540:0.3);
\draw[thick, fill, green] (0,0.67) circle (2pt);
\draw[thick, fill] (0,0) circle (2pt);
\end{tikzpicture}
+
\begin{tikzpicture}[baseline=-.5ex]
\draw[thick, fill] (-1,0) circle (2pt) -- (1,0) circle (2pt) (0,-1) circle (2pt) -- (0,1) circle (2pt);
\draw[line width=3, red] (0,0) -- (0,-0.5) (0,0) -- (0,0.5) (0,0) -- (-0.5, 0);
\draw[thick, <-, gray] (-0.3,0) arc (180:540:0.3);
\draw[thick, fill, green] (0.67,0) circle (2pt);
\draw[thick, fill] (0,0) circle (2pt);
\end{tikzpicture}
\]
Using this relation, again with the bottom edge as the pivot edge, we
can express the middle configuration in the right-hand side in terms
of the other five configurations. Thus, we may also assume that
\( \alpha \) has no stabilizations on the edges lying between the two
half edges \( \sfh_r \) and \( \sfh_s \).

It remains show that we may also assume that \( \alpha \) does not
contain any second half edge. Suppose that \( \sfh_s \) is the second
half edge of some edge \( \sfe_t \). (The case that \( \sfh_r \) is a
second half edge can be proved similarly.) Then \( \alpha \) is
represented as follows:
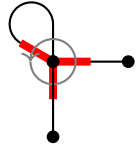
\[
\begin{tikzpicture}[baseline=-.5ex]
\draw[thick, fill] (0,0) node (A) {} -- (1,0) circle (2pt);
\grape[120]{A};
\draw[line width=3, red] (0,0) -- (-90:0.5) (0,0) -- (150:0.5) (0,0) -- (0.5,0);
\draw[thick, ->, gray] (-0.3,0) arc (180:540:0.3);
\draw[thick, fill] (0,0) circle (2pt) -- (0, -1) circle (2pt);
\end{tikzpicture}
\]
Applying \eqref{eq:X_1} and \eqref{eq:Q} to this shape, we obtain the following relations:
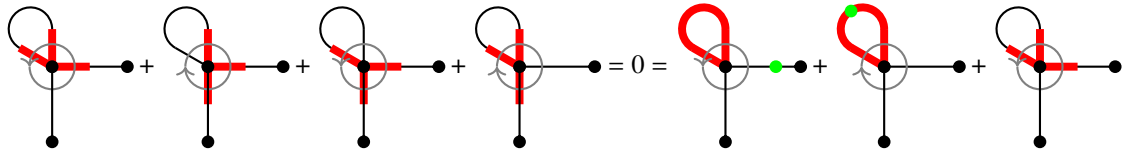
\[
\begin{tikzpicture}[baseline=-.5ex]
\draw[thick, fill] (0,0) node (A) {} -- (1,0) circle (2pt);
\grape[120]{A};
\draw[line width=3, red] (0,0) -- (90:0.5) (0,0) -- (150:0.5) (0,0) -- (0.5,0);
\draw[thick, ->, gray] (-0.3,0) arc (180:540:0.3);
\draw[thick, fill] (0,0) circle (2pt) -- (0, -1) circle (2pt);
\end{tikzpicture}
+
\begin{tikzpicture}[baseline=-.5ex]
\draw[thick, fill] (0,0) node (A) {} -- (1,0) circle (2pt);
\grape[120]{A};
\draw[line width=3, red] (0,0) -- (90:0.5) (0,0) -- (-90:0.5) (0,0) -- (0.5,0);
\draw[thick, <-, gray] (-0.3,0) arc (180:540:0.3);
\draw[thick, fill] (0,0) circle (2pt) -- (0, -1) circle (2pt);
\end{tikzpicture}
+
\begin{tikzpicture}[baseline=-.5ex]
\draw[thick, fill] (0,0) node (A) {} -- (1,0) circle (2pt);
\grape[120]{A};
\draw[line width=3, red] (0,0) -- (-90:0.5) (0,0) -- (150:0.5) (0,0) -- (0.5,0);
\draw[thick, ->, gray] (-0.3,0) arc (180:540:0.3);
\draw[thick, fill] (0,0) circle (2pt) -- (0, -1) circle (2pt);
\end{tikzpicture}
+
\begin{tikzpicture}[baseline=-.5ex]
\draw[thick, fill] (0,0) node (A) {} -- (1,0) circle (2pt);
\grape[120]{A};
\draw[line width=3, red] (0,0) -- (90:0.5) (0,0) -- (150:0.5) (0,0) -- (-90:0.5);
\draw[thick, <-, gray] (-0.3,0) arc (180:540:0.3);
\draw[thick, fill] (0,0) circle (2pt) -- (0, -1) circle (2pt);
\end{tikzpicture}
=0=
  \begin{tikzpicture}[baseline=-.5ex]
\draw[thick, fill] (0,0) node (A) {} -- (1,0) circle (2pt);
\grape[120]{A};
\draw[line width=3, red] (0,0) -- (150:0.5) (0,0) -- (90:0.5) (150:0.5) arc (240:0:{0.5/sqrt(3)});
\draw[thick, ->, gray] (-0.3,0) arc (180:540:0.3);
\draw[thick, fill, green] (0.67,0) circle (2pt);
\draw[thick, fill] (0,0) circle (2pt) -- (0, -1) circle (2pt);
\end{tikzpicture}
+
\begin{tikzpicture}[baseline=-.5ex]
\draw[thick, fill] (0,0) node (A) {} -- (1,0) circle (2pt);
\grape[120]{A};
\draw[line width=3, red] (0,0) -- (150:0.5) (0,0) -- (90:0.5) (150:0.5) arc (240:0:{0.5/sqrt(3)});
\draw[thick, <-, gray] (-0.3,0) arc (180:540:0.3);
\draw[thick, fill, green] (-.44,.74) circle (2pt);
\draw[thick, fill] (0,0) circle (2pt) -- (0, -1) circle (2pt);
\end{tikzpicture}
+
\begin{tikzpicture}[baseline=-.5ex]
\draw[thick, fill] (0,0) node (A) {} -- (1,0) circle (2pt);
\grape[120]{A};
\draw[line width=3, red] (0,0) -- (90:0.5) (0,0) -- (150:0.5) (0,0) -- (0.5,0);
\draw[thick, ->, gray] (-0.3,0) arc (180:540:0.3);
\draw[thick, fill] (0,0) circle (2pt) -- (0, -1) circle (2pt);
\end{tikzpicture}
\]
Therefore, we can express \( \alpha \) as follows:
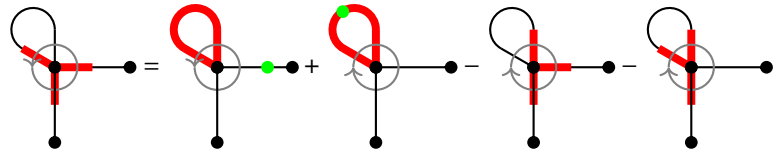
\[
\begin{tikzpicture}[baseline=-.5ex]
\draw[thick, fill] (0,0) node (A) {} -- (1,0) circle (2pt);
\grape[120]{A};
\draw[line width=3, red] (0,0) -- (-90:0.5) (0,0) -- (150:0.5) (0,0) -- (0.5,0);
\draw[thick, ->, gray] (-0.3,0) arc (180:540:0.3);
\draw[thick, fill] (0,0) circle (2pt) -- (0, -1) circle (2pt);
\end{tikzpicture}
=
  \begin{tikzpicture}[baseline=-.5ex]
\draw[thick, fill] (0,0) node (A) {} -- (1,0) circle (2pt);
\grape[120]{A};
\draw[line width=3, red] (0,0) -- (150:0.5) (0,0) -- (90:0.5) (150:0.5) arc (240:0:{0.5/sqrt(3)});
\draw[thick, ->, gray] (-0.3,0) arc (180:540:0.3);
\draw[thick, fill, green] (0.67,0) circle (2pt);
\draw[thick, fill] (0,0) circle (2pt) -- (0, -1) circle (2pt);
\end{tikzpicture}
+
\begin{tikzpicture}[baseline=-.5ex]
\draw[thick, fill] (0,0) node (A) {} -- (1,0) circle (2pt);
\grape[120]{A};
\draw[line width=3, red] (0,0) -- (150:0.5) (0,0) -- (90:0.5) (150:0.5) arc (240:0:{0.5/sqrt(3)});
\draw[thick, <-, gray] (-0.3,0) arc (180:540:0.3);
\draw[thick, fill, green] (-.44,.74) circle (2pt);
\draw[thick, fill] (0,0) circle (2pt) -- (0, -1) circle (2pt);
\end{tikzpicture}
-
\begin{tikzpicture}[baseline=-.5ex]
\draw[thick, fill] (0,0) node (A) {} -- (1,0) circle (2pt);
\grape[120]{A};
\draw[line width=3, red] (0,0) -- (90:0.5) (0,0) -- (-90:0.5) (0,0) -- (0.5,0);
\draw[thick, <-, gray] (-0.3,0) arc (180:540:0.3);
\draw[thick, fill] (0,0) circle (2pt) -- (0, -1) circle (2pt);
\end{tikzpicture}
-
\begin{tikzpicture}[baseline=-.5ex]
\draw[thick, fill] (0,0) node (A) {} -- (1,0) circle (2pt);
\grape[120]{A};
\draw[line width=3, red] (0,0) -- (90:0.5) (0,0) -- (150:0.5) (0,0) -- (-90:0.5);
\draw[thick, <-, gray] (-0.3,0) arc (180:540:0.3);
\draw[thick, fill] (0,0) circle (2pt) -- (0, -1) circle (2pt);
\end{tikzpicture}
\]
Using \eqref{eq:Q} again, we can express the last term in terms of
loop classes. Thus, \( \alpha \) can be expressed in terms of loop
classes and the third term above, say \( \alpha' \).

Note that \( \alpha' \) has three half edges
\( \sfh_1, \sfh_r,\sfh_{s-1} \), where \( \sfh_{s-1} \) is the first
half edge of \( \sfe_t \). By the assumption on \( \alpha \), there
are no stabilizations on the edges lying between \( \sfh_r \) and
\( \sfh_{s-1} \). Therefore, we can always reduce the number of second
half edges in \( \alpha \), while maintaining the assumptions that it
contains the first half edge \( \sfh_1 \), and that there are no
stabilizations on the edges lying between the other two half edges.
Hence, \( \alpha \) becomes a standard star class, which completes the
proof.
\end{proof}

\begin{remark}
  The arguments in the proof of \Cref{lem:standard star} have been
  known in the community, although they are not explicitly documented
  in the literature. In this paper, we take the opportunity to write
  down some details about them. See \cite{KP2012} and \cite{ADCK2019}
  for related topics.
\end{remark}

For a polynomial \( p\in \bbF[E] \) and an edge \( \sfe\in E \), we
denote by \( \deg_\sfe p \) the degree of \( \sfe \) in \( p \) as a
polynomial in the single variable \( \sfe \).

\begin{lem}\label{lem:one-to-one}
There are one-to-one correspondences
\[
S^1_{\ell,m}(k)\cong\HE^1_{\ell,m}(k)
\quad
\text{and}
\quad
L^1_{\ell,m}(k)\cong\HE^2_{\ell,m}(k).
\]
\end{lem}
\begin{proof}
  Let \(\alpha\in S^1_{\ell,m}(k)\) be a standard star class
  characterized by three edges \(\sfe_1, \sfe_r, \sfe_s\) and a monic
  monomial \(p\in\mathbb{Z}[E]\). Then we can assign an
  HE-configuration \((\sfh^{1}(\sfe_r),\vec a)\) of type \(1\), where
  \(\vec a = (\deg_{\sfe_t} (p\sfe_s))_{t=1}^{\ell+m}\). For a loop
  class \(\beta=q j_*(\beta_1)\in L^1_{\ell,m}(k)\), we assign an
  HE-configuration \((\sfh^2(\sfe_r),\vec a)\) of type \(2\), where
  \(\sfe_r\) is the edge that the class \(\beta\) involves and
  \(\vec a= (\deg_{\sfe_t} q\sfe_r)_{t=1}^{\ell+m}\).

  Conversely, for an HE-configuration \((\sfh^1(\sfe_r),\vec a)\) of
  type \(1\), let \(s\) be the smallest index with \(a_s\ge 1\) and
  \(r<s\). Then we have a homology class \(\alpha\) determined by the
  three half edges \(\sfh^1(\sfe_1),\sfh^1(\sfe_r),\sfh^1(\sfe_s)\)
  and the monic monomial \(p\in\mathbb{Z}[E]\) such that
  \(p \sfe_s = \prod_{t=1}^{\ell+m}\sfe_t^{a_t}\). For an
  HE-configuration \((\sfh^2(\sfe_r),\vec a)\) of type \(2\), the
  homology class \(j_*(\beta_1)\) is the homology class determined by
  the loop \(\sfe_r\) and the stabilization \(q\) is determined by
  \(q\sfe_r=\prod_{t=1}^{\ell+m} \sfe_t^{a_t}\).
\end{proof}

See \Cref{figure:star and HE1,figure:star and HE2} for examples of the correspondences
in \Cref{lem:one-to-one}.

\begin{figure}
\subcaptionbox{A star class and the corresponding HE-configurations of type \(1\).\label{figure:star and HE1}}{
\(
\begin{tikzcd}[column sep=4pc, ampersand replacement=\&]
\begin{tikzpicture}[baseline=-.5ex]
\draw[thick, fill, blue] 
(0,0) -- (-36:2) circle(2pt) node[below right] {\(*\)};
\draw[thick, fill] 
(0,0) -- (-12:2) circle(2pt) 
(0,0) -- (12:2) circle(2pt) 
(0,0) -- (36:2) circle (2pt);
\draw[thick] 
(0,0) -- (60:1) arc (-30:210:{1/sqrt(3)}) -- (0,0)
(0,0) -- (150:1) arc (60:300:{1/sqrt(3)}) -- (0,0)
(0,0) -- (240:1) arc (150:390:{1/sqrt(3)}) -- (0,0);
\draw[line width=3, red] 
(0,0) -- (12:1)
(0,0) -- (-12:1)
(0,0) -- (-36:1)
;
\draw[thick, fill] (0,0) circle (2pt);
\draw[thick, green, fill] 
(-36:1.33) circle (2pt) (-36:1.67) circle (2pt)
(-90:{3/sqrt(3)}) circle (2pt)
(12:1.5) circle (2pt) (12:1.75) circle (2pt)
(90:{2/sqrt(3)}) + (60:{1/sqrt(3)}) circle (2pt) + (120:{1/sqrt(3)}) circle (2pt)
(180:{3/sqrt(3)}) circle (2pt)
;
\draw[thick, ->, gray] (-0.3,0) arc (180:540:0.3);
\end{tikzpicture}
\ar[r,leftrightarrow]\&
\begin{tikzpicture}[baseline=-.5ex]
\draw[thick, fill, blue] 
(0,0) -- (-36:2) circle(2pt) node[below right] {\(*\)};
\draw[thick, fill] 
(0,0) -- (-12:2) circle(2pt) 
(0,0) -- (12:2) circle(2pt) 
(0,0) -- (36:2) circle (2pt);
\draw[thick] 
(0,0) -- (60:1) arc (-30:210:{1/sqrt(3)}) -- (0,0)
(0,0) -- (150:1) arc (60:300:{1/sqrt(3)}) -- (0,0)
(0,0) -- (240:1) arc (150:390:{1/sqrt(3)}) -- (0,0);
\draw[line width=3, red] 
(0,0) -- (-12:1)
;
\draw[thick, fill] (0,0) circle (2pt);
\draw[thick, red, fill] 
(-36:1.33) circle (2pt) (-36:1.67) circle (2pt)
(-90:{3/sqrt(3)}) circle (2pt)
(12:1.25) circle (2pt) 
(12:1.5) circle (2pt) (12:1.75) circle (2pt)
(90:{2/sqrt(3)}) + (60:{1/sqrt(3)}) circle (2pt) + (120:{1/sqrt(3)}) circle (2pt)
(180:{3/sqrt(3)}) circle (2pt)
;
\draw[thick, ->, gray] (-0.3,0) arc (180:540:0.3);
\end{tikzpicture}\\
\alpha = (\sfe_1^2\sfe_3^2\sfe_5^2\sfe_6\sfe_7)\iota_*(\alpha_{123})
\& (\sfh^1(\sfe_2), (2,1,3,0,2,1,1))
\end{tikzcd}
\)
}

\subcaptionbox{A loop class and the corresponding HE-configuration of type \(2\).\label{figure:star and HE2}}{
\begin{tikzcd}[column sep=4pc, ampersand replacement=\&]
\begin{tikzpicture}[baseline=-.5ex]
\draw[thick, fill, blue] 
(0,0) -- (-36:2) circle(2pt) node[below right] {\(*\)};
\draw[thick, fill] 
(0,0) -- (-12:2) circle(2pt) 
(0,0) -- (12:2) circle(2pt) 
(0,0) -- (36:2) circle (2pt);
\draw[thick] 
(0,0) -- (60:1) arc (-30:210:{1/sqrt(3)}) -- (0,0)
(0,0) -- (150:1) arc (60:300:{1/sqrt(3)}) -- (0,0)
(0,0) -- (240:1) arc (150:390:{1/sqrt(3)}) -- (0,0);
\draw[line width=3, red] 
(0,0) -- (300:1) arc (390:150:{1/sqrt(3)}) -- cycle;
\draw[thick, fill] (0,0) circle (2pt);
\draw[thick, green, fill] 
(-36:1.5) circle (2pt)
(12:1.33) circle (2pt) (12:1.67) circle (2pt)
(36:1.33) circle (2pt) (36:1.67) circle (2pt)
(180:{2/sqrt(3)}) + (90:{1/sqrt(3)}) circle (2pt) + (270:{1/sqrt(3)}) circle (2pt) + (180:{1/sqrt(3)}) circle (2pt)
(270:{3/sqrt(3)}) circle (2pt)
;
\draw[thick, ->, gray] (-0.3,0) arc (180:540:0.3);
\end{tikzpicture}
\ar[r,leftrightarrow]
\&
\begin{tikzpicture}[baseline=-.5ex]
\draw[thick, fill, blue] 
(0,0) -- (-36:2) circle(2pt) node[below right] {\(*\)};
\draw[thick, fill] 
(0,0) -- (-12:2) circle(2pt) 
(0,0) -- (12:2) circle(2pt) 
(0,0) -- (36:2) circle (2pt);
\draw[thick] 
(0,0) -- (60:1) arc (-30:210:{1/sqrt(3)}) -- (0,0)
(0,0) -- (150:1) arc (60:300:{1/sqrt(3)}) -- (0,0)
(0,0) -- (240:1) arc (150:390:{1/sqrt(3)}) -- (0,0);
\draw[line width=3, red] 
(0,0) -- (300:1);
\draw[thick, fill] (0,0) circle (2pt);
\draw[thick, red, fill] 
(-36:1.5) circle (2pt)
(12:1.33) circle (2pt) (12:1.67) circle (2pt)
(36:1.33) circle (2pt) (36:1.67) circle (2pt)
(180:{2/sqrt(3)}) + (90:{1/sqrt(3)}) circle (2pt) + (270:{1/sqrt(3)}) circle (2pt) + (180:{1/sqrt(3)}) circle (2pt)
(270:{3/sqrt(3)}) circle (2pt)
;
\draw[thick, ->, gray] (-0.3,0) arc (180:540:0.3);
\end{tikzpicture}\\
\beta=(\sfe_1\sfe_3^2\sfe_4^2\sfe_6^3\sfe_7)j_*(\beta_1)
\&
(\sfh^2(\sfe_7),(1,0,2,2,0,3,2))
\end{tikzcd}
}
\caption{Examples of the correspondence between HE-configurations and homology classes.}
\label{figure:HE and classes}
\end{figure}

\begin{prop}\label{prop:star and loop}
Let \(\ell\ge0\) and \(m\ge 1\). For each \(k\ge 0\), the union \(S^1_{\ell,m}(k)\sqcup L^1_{\ell,m}(k)\) forms an \(\bbF\)-basis for \(H_1(B_k\Gamma_{\ell,m})\).
\end{prop}
\begin{proof}
  Recall the number \( N_{\ell,m}(k) \) in \Cref{prop:elementary}. Let
  \(\bbF\langle S^1_{\ell,m}(k)\sqcup L^1_{\ell,m}(k)\rangle\) be the
  subspace of \(H_1(B_k\Gamma_{\ell,m})\) spanned by
  \(S^1_{\ell,m}(k)\sqcup L^1_{\ell,m}(k)\). Then we have
\begin{align*}
N_{\ell,m}(k)&=\dim_\bbF H_1(B_k\Gamma_{\ell,m})\tag{\Cref{equation:residual for elementary}}\\
&=\dim_\bbF \bbF\langle S^1_{\ell,m}(k)\sqcup L^1_{\ell,m}(k)\rangle\tag{\Cref{lem:standard star}}\\
&\le |S^1_{\ell,m}(k)\sqcup L^1_{\ell,m}(k)|\\
&=|\HE_{\ell,m}(k)|\tag{\Cref{lem:one-to-one}}\\
&=N_{\ell,m}(k). \tag{\Cref{prop:HE and H1}}
\end{align*}
Hence, \(\dim_\bbF \bbF\langle S^1_{\ell,m}(k)\sqcup L^1_{\ell,m}(k)\rangle=
|S^1_{\ell,m}(k)\sqcup L^1_{\ell,m}(k)|\), which implies the linear independence of \(S^1_{\ell,m}(k)\sqcup L^1_{\ell,m}(k)\).
\end{proof}

\subsection{Homology classes for bunches of grapes}\label{section:the map alpha}
Let \(\Gamma=(\sfT,\loops)\) be a bunch of grapes such that \(\sfT\)
has at least one edge. Fix an HE-configuration
\(\mathscr{X}=(W,\vec j, ((\sfh^\sfv,\vec a^\sfv)), \vec c)\in
\HE^i_\Gamma(k)\). Then we split \(W\) into the two subsets \(W_1\)
and \(W_2\) given by \(\sfv\in W_\epsilon\) if
\((\sfh^\sfv,\vec
a^\sfv)\in\SHE^\epsilon_{\ell(\sfv),m(\sfv)}(j_\sfv)\).

We assign to \( \mathscr{X} \) a configuration
\( \Psi(\mathscr{X}) \) of dots and half edges as
follows: \begin{enumerate}
\item For \(\sfv\in W_1\) and \(\sfh^\sfv=\sfh^1(\sfe_r^\sfv)\),
\begin{enumerate}
\item choose the smallest index \(s\) such that \(r<s\) with \(a^\sfv_s\ge 1\),
\item replace the dot on \(\sfe^\sfv_s\) closest to \(\sfv\) with a marked half edge, and mark the half edge on the pivot edge \(\sfe_1^\sfv\), and
\item regard the remaining dots as stabilizations.
\end{enumerate}
\item For \(\sfv\in W_2\) and \(\sfh^\sfv=\sfh^2(\sfe_r^\sfv)\),
\begin{enumerate}
\item mark the edge \(\sfe_r^\sfv\), and
\item regard all other dots as stabilizations.
\end{enumerate}
\item Assign a counterclockwise orientation at each vertex \(\sfv\in W\).
\end{enumerate}
Let \(p\in\mathbb{Z}[E]\) be the stabilization monomial in the
configuration \( \Psi(\mathscr{X}) \), i.e., the monomial obtained
by multiplying the edge containing each remaining dot. Then one can
easily check that the total degree of \(p\) is \(k-2|W_1|-|W_2|\).
Moreover, for each \(\sfv\in W\), there is a canonical embedding
\[
\iota^\sfv:B_2\Gamma_{0,3}\to B_2\Gamma\quad\text{ or }\quad
j^\sfv:B_1\Gamma_{1,0}\to B_1\Gamma
\]
according to whether \(\sfv\) is contained in \(W_1\) or \(W_2\).
Since these embeddings have pairwise disjoint images, we have the embeddings
\begin{align*}
\mathcal{I}&=\coprod_{\sfv\in W_1} \iota^\sfv : \prod_{\sfv\in W_1} B_2\Gamma_{0,3} \to B_{2|W_1|}\Gamma& &\text{and}&
\mathcal{J}&=\coprod_{\sfv\in W_2} j^\sfv : \prod_{\sfv\in W_2} B_1\Gamma_{1,0} \to B_{|W_2|}\Gamma
\end{align*}
defined by
\begin{align*}
\mathcal{I}((x^\sfv)_{\sfv\in W_1}) &= 
\coprod_{\sfv\in W_1} \iota^{\sfv}(x^\sfv)
& &\text{and}&
\mathcal{J}((y^\sfv)_{\sfv\in W_2}) &= \{y^\sfv: \sfv\in W_2\},
\end{align*}
respectively.

Finally, we define the map
\(\mathcal{K}\coloneqq\mathcal{I}\coprod \mathcal{J}:
\left(\prod_{\sfv\in W_1}B_2\Gamma_{0,3}\right)\times
\left(\prod_{\sfv\in W_2}B_1\Gamma_{1,0}\right) \to
B_{2|W_1|+|W_2|}\Gamma \) as the disjoint union of the maps
\( \mathcal{I} \) and \( \mathcal{J} \). Then the map \(\mathcal{K}\)
induces a group homomorphism between homology groups
\[
\mathcal{K}_*:H_{|W|}\left(
\left(\prod_{\sfv\in W_1}B_2\Gamma_{0,3}\right)\times
\left(\prod_{\sfv\in W_2}B_1\Gamma_{1,0}\right)
\right)
\to H_{|W|}(B_{2|W_1|+|W_2|}\Gamma).
\]
Note that due to the K\"unneth theorem and the fact that
\(H_i(B\Gamma_{\ell,m})=0\) for all \(i\ge2\), the domain of
\(\mathcal{K}_*\) can be expressed as a tensor product:
\[
H_{|W|}\left(
\left(\prod_{\sfv\in W_1}B_2\Gamma_{0,3}\right)\times
\left(\prod_{\sfv\in W_2}B_1\Gamma_{1,0}\right)
\right)
\cong \left(\bigotimes_{\sfv\in W_1} H_1(B_2\Gamma_{0,3}) \right)
\otimes
\left(\bigotimes_{\sfv\in W_2} H_1(B_1\Gamma_{1,0}) \right)
\cong \mathbb{Z}\langle \alpha_0\rangle,
\]
where the generator \(\alpha_0\) is given by the tensor product
\( \left(\bigotimes_{\sfv\in W_1} \alpha_{123}\right)\otimes
\left(\bigotimes_{\sfv\in W_2} \beta_{1}\right) \) of the generators.
Then the given HE-configuration \( \mathscr{X} \) provides a homology
class
\(\Psi(\mathscr{X})=p\cdot\mathcal{K}_*(\alpha_0)\in
H_{|W|}(B_k\Gamma)\).

\begin{exam}
  For example, the HE-configuration \(\mathscr{X}\) in
  \Cref{fig:image21} corresponds to the following diagram:
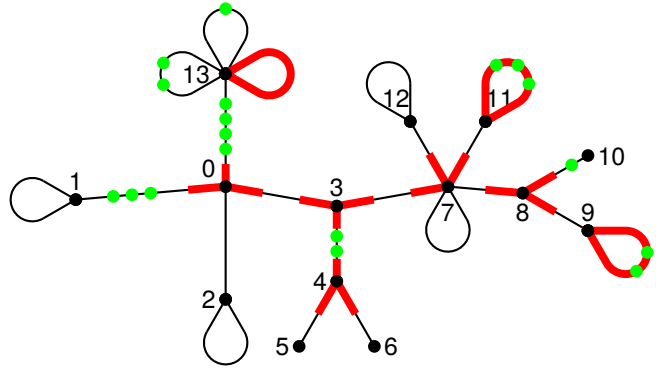
\[
\begin{tikzpicture}[baseline=-.5ex]
\draw[thick,fill] (0,0) circle (2pt) node (A) {} node[above] {\(\mathsf{1}\)} -- ++(5:2) circle(2pt) node(B) {} node[above left] {\(\mathsf{0}\)};
\draw[thick,fill] (B.center) +(0,1.5) circle (2pt) node (C) {} node[left=0.5ex] {\(\mathsf{13}\)};
\draw[thick,fill] (B.center) -- (C.center);
\draw[thick,fill] (B.center) -- +(0,-1.5) circle (2pt) node(D) {} node[left] {\(\mathsf{2}\)} +(0,0) -- ++(-10:1.5) circle(2pt) node(E) {} node[above] {\(\mathsf{3}\)} -- ++(10:1.5) circle(2pt) node(F) {} node[below=0.5ex] {\(\mathsf{7}\)} -- +(120:1) circle(2pt) node(G) {} node[xshift=-1.2ex,yshift=2ex] {\(\mathsf{12}\)} +(0,0) -- +(60:1) circle (2pt) node(H) {} node[xshift=1.2ex,yshift=2ex] {\(\mathsf{11}\)} +(0,0) -- ++(-5:1) circle (2pt) node(I) {} node[below] {\(\mathsf{8}\)} -- +(30:1) circle(2pt) node(J) {} node[right] {\(\mathsf{10}\)} +(0,0) -- +(-30:1) circle (2pt) node(K) {} node[above] {\(\mathsf{9}\)};
\draw[thick, fill] (E.center) -- ++(0,-1) circle (2pt) node[left] {\(\mathsf{4}\)} -- +(-60:1) circle (2pt) node[right] {\(\mathsf{6}\)} +(0,0) -- +(-120:1) circle (2pt) node[left] {\(\mathsf{5}\)};
\grape[180]{A};
\grape[0]{C}; \grape[90]{C}; \grape[180]{C};
\grape[-90]{D};
\grape[-90]{F};
\grape[120]{G};
\grape[60]{H};
\grape[-30]{K};

%half edges
\draw[line width=3,red] 
(B.center) -- +(90:0.3)
(B.center) -- +(5:-0.5)
(B.center) -- +(-10:0.5)
(C.center) -- +(30:0.5) arc (120:-120:{0.5/sqrt(3)}) -- cycle
(E.center) -- +(0,-0.3)
(E.center) -- +(10:0.5)
(E.center) ++(0,-1) -- +(-120:0.5) +(0,0) -- +(-60:0.5)
(E.center) ++(0,-1) -- +(90:0.3)
(E.center) -- +(-10:-0.5)
(F.center) -- +(60:0.5)
(F.center) -- +(120:0.5)
(F.center) -- +(10:-0.5)
(H.center) -- +(90:0.5) arc (180:-60:{0.5/sqrt(3)}) -- cycle
(I.center) -- +(-30:0.5)
(I.center) -- +(30:0.5)
(I.center) -- +(-5:-0.5)
(K.center) -- +(0:0.5) arc (90:-150:{0.5/sqrt(3)}) -- cycle;

\draw[thick, fill] 
(B.center) circle (2pt)
(C.center) circle (2pt)
(E.center) circle (2pt)
(E.center) ++(0,-1) circle (2pt)
(F.center) circle (2pt)
(H.center) circle (2pt)
(I.center) circle (2pt)
(K.center) circle (2pt);

% dots
\draw[thick,fill,green] 
(A.center) +(5:1) circle(2pt) +(5:0.75) circle (2pt) +(5:0.5) circle (2pt)
(C.center) +(90:{1.5/sqrt(3)}) circle (2pt) ++(180:{1/sqrt(3)}) +(150:{0.5/sqrt(3)}) circle (2pt) +(210:{0.5/sqrt(3)}) circle (2pt)
(E.center) +(0,-0.4) circle (2pt) +(0,-0.6) circle (2pt)
(H.center) ++(60:{1/sqrt(3)}) +(0:{0.5/sqrt(3)}) circle (2pt) +(60:{0.5/sqrt(3)}) circle (2pt) +(120:{0.5/sqrt(3)}) circle (2pt)
(I.center) +(30:0.75) circle (2pt)
(K.center) ++(-30:{1/sqrt(3)}) +(-60:{0.5/sqrt(3)}) circle (2pt) +(0:{0.5/sqrt(3)}) circle (2pt);
% additional dots
\draw[thick,fill,green] 
(B.center) +(90:0.5) circle(2pt) +(90:0.7) circle(2pt) +(90:0.9) circle(2pt) +(90:1.1) circle(2pt)
;
\end{tikzpicture}
\]
Here, we have
\begin{align*}
W_1&=\{\mathsf{0}, \mathsf{3},\mathsf{4},\mathsf{7},\mathsf{8}\},&
W_2&=\{\mathsf{9}, \mathsf{11}, \mathsf{13}\},
\end{align*}
and
\begin{align*}
p&=(\sfe^{\mathsf{1}})^3(\sfe^{\mathsf{4}})^2(\sfe^{\mathsf{9}}_2)^2 (\sfe^{\mathsf{10}})^1(\sfe^{\mathsf{11}}_2)^3 (\sfe^{\mathsf{13}})^4(\sfe^{\mathsf{13}}_3)^1 (\sfe^{\mathsf{13}}_4)^2,&
\mathcal{K}_*(\alpha_0)&\in H_{|W|}(B_{2|W_1|+|W_2|}\Gamma)=H_8(B_{13}\Gamma).
\end{align*}
Therefore, \(\Psi(\mathscr{X})=p\cdot\mathcal{K}_*(\alpha_0)\in H_8(B_{31}\Gamma)\).
\end{exam}

In summary, we obtain that the basis of the \(i\)-th homology group
\(H_i(B_k\Gamma)\) can be given by the set of HE-configurations on
\(\Gamma\).

\begin{thm}\label{thm:combinatorial description for homology cycles}
Let \(\Gamma\cong(\sfT,\loops)\) be a bunch of grapes such that \(\sfT\) has at least one edge.
Then for each \(i\ge 0\) and \(k\ge 0\), the map
\[
\Psi:\HE^i_\Gamma(k)\to H_i(B_k\Gamma)
\]
defined by \(\mathscr{X}\mapsto \Psi(\mathscr{X})\) is injective, and its image forms an \(\bbF\)-basis for \(H_i(B_k\Gamma)\).
\end{thm}

\appendix

\section{Bouquet of circles}\label{appendix}

In this paper, we have considered elementary bunches of graphs
\( \Gamma_{\ell,m} \) for \( m\ge1 \). In this appendix, we consider
the case where \( m=0 \), for completeness.

For \(\ell\ge 1\), let \(\Gamma_{\ell,0}\) be the elementary bunch of
grapes that is homeomorphic to the bouquet of \(\ell\) circles. We
define HE-configurations as in \Cref{def:HE} with \( m=0 \). We also
define SHE-configurations of type \( 1 \) as in \Cref{defn:SHE} with
\( m=0 \). However, we define an \emph{SHE-configuration} of size
\(j\) of \emph{type \( 2 \)} to be a pair
\((\sfh=\sfh^2(\sfe_r),\vec b = (b_1,\dots,b_\ell))\) for some
\(1\le r\le \ell\) such that
\begin{enumerate}
\item \(b_r\in\{0,1\}\) for all \(1\le r\le \ell\) and \(\norm{\vec b}=j\),
\item \(b_r=1\), \(b_{r-1}=0\) if \(\ell\ge 2\),
\item \(b_1=0\) if \(\ell=1\).
\end{enumerate}
Here, the indices of entries in \(\vec b\) are considered as elements in \(\bbZ/\ell\bbZ\). As before, we define \(\SHE_{\ell,0}(j)=\SHE^1_{\ell,0}\sqcup\SHE^2_{\ell,0}(j)\).

Note that if \(\ell=1\), then \(\SHE^2_{1,0}(j)\neq\varnothing\) if
and only if \(j=0\), in which case, we have
\(\SHE^2_{\ell,0}(0)=\{(\sfh^2(\sfe_1),(0))\}\). For \(\ell\ge 2\),
the definition of \( \SHE^2_{\ell,0}(j) \) is the same as the one in
\Cref{defn:SHE}, except that we replace the condition \(b_1=0\) by
\(b_{r-1}=0\).

\begin{lem}
For \(j\ge 0\), \(k\ge 0\), and \(\ell\ge 1\), we have
\begin{align*}
|\SHE^1_{\ell,0}(j)|&=(\ell-2)\binom{\ell-2}{j-1} - \binom{\ell-2}{j},&
|\HE^1_{\ell,0}(k)|&=\sum_{j\ge 0}|\SHE^1_{\ell,0}(j)|\binom {k}{j},\\
|\SHE^2_{\ell,0}(j)|&=\begin{cases}
1 & \ell=1, j=0;\\
0 & \ell=1, j\ge 1;\\
\ell\binom{\ell-2}{j-1} & \ell\ge 2,
\end{cases}&
|\HE^2_{\ell,0}(k)|&=\sum_{j\ge 0}|\SHE^2_{\ell,0}(j)|\binom {k}{j}.
\end{align*}
Therefore,
\[
|\SHE_{\ell,0}(k)|\binom k j=|\HE_{\ell,0}(k)|.
\]
\end{lem}
\begin{proof}
  The first two equalities follow from the proof of \Cref{prop:SHE},
  and the last equality follows from the other equalities. Hence, it
  remains to prove the third and fourth equalities.

  If \(\ell=1\), then \(\HE^2_{1,0}(k)=\{(\sfh^2(\sfe_1),(k))\}\) and
  so \(|\HE^2_{1,0}(k)|=1=|\SHE^2_{1,0}(0)|\binom k 0\) as desired,
  since \(\SHE^2_{1,0}(j)=\varnothing\) for any \(j\ge 1\). Suppose
  that \(\ell\ge 2\) and consider
  \((\sfh=\sfh^2(\sfe_r),\vec b)\in\SHE^2_{\ell,0}(j)\). Then
  \(b_{r-1}=0\), \(b_r=1\), and \(b_s\) is either \(0\) or \(1\), for
  each \( s\ne r \). Thus, the number of such \(\vec b\) is equal to
  \(\binom {\ell-2}{j-1}\). Since there are \(\ell\) choices for
  \(\sfh\), we have
  \(|\SHE^2_{\ell,0}(j)| = \ell\binom {\ell-2}{j-1}\). This shows the
  third equality.

  Finally, we prove the fourth equality. For each HE-configuration
  \((\sfh^2(\sfe_r), \vec a)\in \HE^2_{\ell,0}(k)\), we define the
  SHE-configuration \((\sfh^2(\sfe_r),\vec b)\) by
\[
b_s=\begin{cases}
0 & s=r-1 \text{ or }b_s=0;\\
1 & s=r;\\
1 & s\neq r-1 \text{ and } b_s\ge 1.
\end{cases}
\]

Conversely, for each SHE-configuration \((\sfh^2(\sfe_r),\vec b)\),
consider the vector \(\vec a=(a_1,\dots,a_\ell)\) of nonnegative
integers with \(\norm{\vec a}=k\) satisfying that \(a_s\ge 1\) if
\(b_s\ge 1\) and \(a_s=0\) if \(s\neq r-1\) and \(b_s=0\). Then it is
easy to check that \((\sfh^2(\sfe_r),\vec a)\in \HE^2_{\ell,0}(k)\).
Moreover, the number of such vectors \(\vec a\) is the number of
nonnegative integer solutions of \(x_1+\cdots+x_{j+1}=k-j\), which is
exactly \(\binom k j\).

Since the above two constructions are inverses to each other, we
obtain the fourth equality.
\end{proof}

For each \(\ell\ge 1\) and \(j,k\ge 0\), the SHE- and
HE-configurations can be depicted in \Cref{figure:bouquet HE and SHE}.

\begin{figure}
\subcaptionbox{An SHE-configuration and the corresponding HE-configuration for \(\ell=1\).}{
\(
\begin{tikzcd}[ampersand replacement=\&, row sep=0pc, column sep=1pc]
\begin{tikzpicture}[baseline=-.5ex]
\useasboundingbox (-1.5,-0.5) rectangle (1.5, 1.5);
\draw[thick, fill] (0,0) circle (2pt) node (A) {};
\draw[thick] (0,0) -- ++(120:0.5) arc (210:-30:{0.5/sqrt(3)}) -- cycle;
\draw (0,0) ++ (0,{1.5/sqrt(3)}) node[above] {\(\sfe_1\)};
\end{tikzpicture}
\&
\begin{tikzpicture}[baseline=-.5ex]
\useasboundingbox (-2.5,-0.5) rectangle (2.5, 1.5);
\draw[thick] (0,0) -- ++(120:0.5) arc (210:-30:{0.5/sqrt(3)}) -- cycle;
\draw[line width=3, red] (0,0) -- ++(120:0.5);
\draw (0,0) ++ (0,{1.5/sqrt(3)}) node[above] {\(\sfe_1\)};
\draw[thick, fill] (0,0) circle (2pt) node (A) {};
\end{tikzpicture}
\&
\begin{tikzpicture}[baseline=-.5ex]
\useasboundingbox (-2.5,-0.5) rectangle (2.5, 1.5);
\draw[thick] (0,0) -- ++(120:0.5) arc (210:-30:{0.5/sqrt(3)}) -- cycle;
\draw[line width=3, red] (0,0) -- ++(120:0.5);
\draw (0,0) ++ (0,{1.5/sqrt(3)}) node[above] {\(\sfe_1\)};
\draw[thick, fill] (0,0) circle (2pt) node (A) {};
\draw[thick, fill, green] (0,0) ++(0,{1/sqrt(3)}) 
+(30:{0.5/sqrt(3)}) circle (2pt)
+(90:{0.5/sqrt(3)}) circle (2pt)
+(150:{0.5/sqrt(3)}) circle (2pt);
\end{tikzpicture}\\
\Gamma \& 
(\sfh^2(\sfe_1),(0))\in \SHE^2_{1,0}(0) \&
(\sfh^2(\sfe_1),(4))\in \HE^2_{1,0}(4)
\end{tikzcd}
\)
}
\subcaptionbox{An SHE-configuration and the corresponding HE-configuration for \(\ell\ge2\).}{
\(
\begin{tikzcd}[ampersand replacement=\&, row sep=0pc, column sep=1pc]
\begin{tikzpicture}[baseline=-.5ex]
\useasboundingbox (-1.5,-1.5) rectangle (1.5, 1.5);
\draw[thick, fill] (0,0) circle (2pt) node (A) {};
\foreach \i in {0,90,180,270} {
\begin{scope}[rotate=\i]
\draw[thick, fill] (0,0) circle(2pt) -- +({30}:0.5) +(0,0) -- +({-30}:0.5);
\draw[thick] (0,0) +({30}:0.5) arc ({120}:{-120}:{0.5 / sqrt(3)});
\end{scope}
}
\draw (0,0) ++ (0:{1.5/sqrt(3)}) node[right] {\(\sfe_1\)};
\draw (0,0) ++ (90:{1.5/sqrt(3)}) node[above] {\(\sfe_2\)};
\draw (0,0) ++ (180:{1.5/sqrt(3)}) node[left] {\(\sfe_3\)};
\draw (0,0) ++ (270:{1.5/sqrt(3)}) node[below] {\(\sfe_4\)};
\end{tikzpicture}
\&
\begin{tikzpicture}[baseline=-.5ex]
\useasboundingbox (-2.5,-1.5) rectangle (2.5, 1.5);
\draw[thick, fill] (0,0) node (A) {};
\foreach \i in {0,90,180,270} {
\begin{scope}[rotate=\i]
\draw[thick, fill] (0,0) circle(2pt) -- +({30}:0.5) +(0,0) -- +({-30}:0.5);
\draw[thick] (0,0) +({30}:0.5) arc ({120}:{-120}:{0.5 / sqrt(3)});
\end{scope}
}
\draw (0,0) ++ (0:{1.5/sqrt(3)}) node[right] {\(\sfe_1\)};
\draw (0,0) ++ (90:{1.5/sqrt(3)}) node[above] {\(\sfe_2\)};
\draw (0,0) ++ (180:{1.5/sqrt(3)}) node[left] {\(\sfe_3\)};
\draw (0,0) ++ (270:{1.5/sqrt(3)}) node[below] {\(\sfe_4\)};
\draw[line width=3, red] (0,0) -- ++(120:0.5);
\draw[thick, fill] (0,0) circle (2pt);
\draw[thick, fill, green] (0,0) +(180:{1.5/sqrt(3)}) circle (2pt);
\end{tikzpicture}
\&
\begin{tikzpicture}[baseline=-.5ex]
\useasboundingbox (-2.5,-1.5) rectangle (2.5, 1.5);
\draw[thick, fill] (0,0) node (A) {};
\foreach \i in {0,90,180,270} {
\begin{scope}[rotate=\i]
\draw[thick, fill] (0,0) circle(2pt) -- +({30}:0.5) +(0,0) -- +({-30}:0.5);
\draw[thick] (0,0) +({30}:0.5) arc ({120}:{-120}:{0.5 / sqrt(3)});
\end{scope}
}
\draw (0,0) ++ (0:{1.5/sqrt(3)}) node[right] {\(\sfe_1\)};
\draw (0,0) ++ (90:{1.5/sqrt(3)}) node[above] {\(\sfe_2\)};
\draw (0,0) ++ (180:{1.5/sqrt(3)}) node[left] {\(\sfe_3\)};
\draw (0,0) ++ (270:{1.5/sqrt(3)}) node[below] {\(\sfe_4\)};
\draw[line width=3, red] (0,0) -- ++(120:0.5);
\draw[thick, fill] (0,0) circle (2pt);
\draw[thick, fill, green] (0,0) ++(0,{1/sqrt(3)}) 
+(30:{0.5/sqrt(3)}) circle (2pt)
+(90:{0.5/sqrt(3)}) circle (2pt)
+(150:{0.5/sqrt(3)}) circle (2pt)
(0,0) ++(180:{1/sqrt(3)}) 
+(150:{0.5/sqrt(3)}) circle (2pt)
+(210:{0.5/sqrt(3)}) circle (2pt);
\end{tikzpicture}\\
\Gamma \&
(\sfh^2(\sfe_2),(0,1,1,0))\in \SHE^2_{4,0}(2) \&
(\sfh^2(\sfe_2),(0,4,2,0))\in \HE^2_{4,0}(6)
\end{tikzcd}
\)
}
\caption{Examples of SHE- and HE-configurations on bouquet of circles.}
\label{figure:bouquet HE and SHE}
\end{figure}

Finally, for each HE-configuration \((\sfh^2(\sfe_i),\vec a)\in \HE^2_{\ell,0}(k)\) on \(\Gamma_{\ell,0}\), there is a natural way to assign a stabilized loop class as seen earlier in \Cref{figure:HE and classes}.

As before, we denote the sets of homology cycles corresponding to HE-configurations \(\HE^\epsilon_{\ell,0}(k)\) of type 1 and 2 by \(S^1_{\ell,0}(k)\) and \(L^1_{\ell,0}(k)\), respectively. Then we have the following proposition, which is an analog of \Cref{prop:star and loop}, and we omit the proof.

\begin{prop}
Let \(\ell\ge1\). For each \(k\ge 0\), the union \(S^1_{\ell,0}(k)\sqcup L^1_{\ell,0}(k)\) forms an \(\bbF\)-basis for \(H_1(B_k\Gamma_{\ell,0})\).
\end{prop}

\section{Relationship with discrete Morse theory}\label{section:comparison}

In this appendix, we introduce discrete configuration spaces of graphs
and compare the homology basis obtained from \emph{discrete Morse
  theory} with ours. The \emph{discrete configuration space}
\(UD_k\Gamma\) is the collection
\[
UD_k\Gamma=\{\{c_1,\dots,c_k\}:c_i\cap c_j=\varnothing \quad \forall i\neq j\}
\]
of sets of \(k\) pairwise disjoint cells \(c_1,\dots,c_k\), i.e., vertices and edges, of
\(\Gamma\) equipped with the natural cube complex structure. Hence,
for each cell \(c\) of \(UD_k\Gamma\), its dimension is the same as
the number of edges in \(c\).

\begin{thm}[\cite{Abrams2000, KKP2012}]\label{Abrams}
  Let \(k\ge 1\) be an integer and let \(\Gamma\) be a graph. Then the
  discrete \(k\)-configuration spaces \(UD_k\Gamma\) is \(K(\pi,1)\)
  and a deformation retract of \(B_k\Gamma\) if \(\Gamma\) is
  \emph{sufficiently subdivided} in the following sense:
\begin{enumerate}
\item Any path joining two distinct essential vertices has an edge length at least \(k-1\).
\item Any nontrivial cycle has an edge length at least \(k+1\).
\end{enumerate}
\end{thm}

Let \(\Gamma\) be a sufficiently subdivided graph for \(k\ge 2\), and
let \(\sfT\) be a maximal tree of \(\Gamma\). Starting from a leaf of
\(\sfT\), we label the vertices of \(\Gamma\) with integers
\( 0,1,2,\dots \), beginning with \(0\), as we follow the clockwise
direction along the boundary of the thickening of \(\sfT\). 
For vertices \(\sfv\) and \(\sfw\) of \(\sfT\), we denote
\(\sfv<\sfw\) if the label of \(\sfv\) is smaller than that of
\(\sfw\), and denote the vertices of an edge \(\sfe\) by
\(\sfv^+(\sfe) \) and \( \sfv^-(\sfe) \) if
\(\sfv^-(\sfe)<\sfv^+(\sfe)\).

Let \(c\) be a cell of \(UD_k\Gamma\). A vertex \(\sfv\in c\) is
\emph{unblocked} if \((c\setminus\{\sfv\})\cup\{\sfe(\sfv)\}\) is a
cell in \(UD_{k+1}\Gamma\), where \( \sfe(\sfv) \) is the edge with
endpoints \( \sfv \) and \( \sfw \) for the unique vertex \( \sfw \)
adjacent to \( \sfv \) with \( \sfv < \sfw \). If \(c\) has an
unblocked vertex, then we define
\(W_k(c)=(c\setminus\{\sfv\})\cup\{\sfe(\sfv)\}\) for the smallest
unblocked vertex \(\sfv\in c\).

For a cell \(c\) in \(UD_k\Gamma\), we say that 
\begin{enumerate}
\item \( c \) is \emph{collapsible} if it is in the image of \(W_{k-1}\), 
\item \( c \) is \emph{redundant} if it is not collapsible but
  contains an unblocked vertex, and
\item \( c \) is \emph{critical} if it is neither collapsible nor redundant.
\end{enumerate}
The set of critical \(i\)-cells in \(UD_k\Gamma\) is denoted by \(\crit_i(UD_k\Gamma,W)\).

One may define a linear map \(V:C^\cell_\bullet(UD_k\Gamma)\to C^\cell_\bullet(UD_k\Gamma)\) of degree \(1\) between cellular chain complexes as follows: \(V(c)=0\) for a collapsible or critical cell \(c\), and \(V(c)=W(c)\) or \(V(c)=-W(c)\) such that the term \(c\) is canceled in \(c+\partial W(c)\) for a redundant cell \(c\).

Let \(M_\bullet(UD_k\Gamma,W)\) be a free module over \(\bbF\) spanned by critical cells, and let 
\[
f=1+\partial V+V\partial:C^\cell_\bullet(UD_k\Gamma)\to C^\cell_\bullet(UD_k\Gamma).
\]
Then it is chain homotopic to the identity by definition. Indeed, the
iterated composition of \(f\) eventually stabilizes, and we denote it
by \(f^\infty\). It is known that the composition
\(\pi f^\infty\partial\) defines a differential on
\(M_\bullet(UD_k\Gamma,W)\), where
\(\pi:C^\cell_\bullet(UD_k\Gamma)\to M_\bullet(UD_k\Gamma,W)\) is the
projection map.

\begin{thm}\cite[Theorem~18]{Farley2005}
The complexes \((C^\cell_\bullet(UD_k\Gamma),\partial)\) and \((M_\bullet(UD_k\Gamma,W),\pi f^\infty\partial)\) have isomorphic homology groups.
\end{thm}

For a bunch of grapes \(\Gamma=(\sfT,\loops)\), let \(\Gamma'\) be its
sufficiently subdivided graph with \(k\ge 2\). For each cycle \(\sfC\)
of \(\Gamma\), let \(\sfC'\) denote the corresponding cycle of
\(\Gamma'\). Then, \( \sfT' \) has a unique edge incident to the
essential vertex of \( \sfC \) corresponding to the second half-edge
of \(\sfC\). By removing these edges, we obtain a maximal tree
\(\sfT'\) of \(\Gamma'\). We label the vertices of \( \sfT' \) as
described above. Note that the vertex \(\sfw_0\) with label \(0\) is
either a leaf of \(\sfT\) or a vertex obtained by a subdivision of a
cycle attached to a leaf of \(\sfT\). See \Cref{figure:maximal tree}
for an example when \(k=3\).

\begin{figure}[ht]
\newcounter{vertex}
\setcounter{vertex}{0}
\[
\begin{tikzpicture}[baseline=-.5ex,scale=1.5]
\begin{scope}[line width=10, line cap=round, gray, opacity=0.5]
\draw
(0,0) node (A) {} -- ++(5:2) node(B) {}
(B.center) -- +(0,1.5) node (C) {}
(B.center) -- +(0,-1.5) node(D) {} 
+(0,0) -- ++(-10:1.5) node(E) {}
+(0,0) -- ++(10:1.5) node(F) {}
+(0,0) -- +(120:1) node(G) {}
+(0,0) -- +(60:1) node(H) {} 
+(0,0) -- ++(-5:1) node(I) {} 
+(0,0) -- +(30:1) node(J) {} 
+(0,0) -- +(-30:1) node(K) {}
(E.center) -- ++(0,-1)
+(0,0) -- +(-60:1) 
+(0,0) -- +(-120:1);
\ddgrape[180]{A};
\ddgrape[0]{C}; \ddgrape[90]{C}; \ddgrape[180]{C};
\ddgrape[-90]{D};
\ddgrape[-90]{F};
\ddgrape[120]{G};
\ddgrape[60]{H};
\ddgrape[-30]{K};
\end{scope}
\begin{scope}[line width=9, line cap=round, white]
\draw
(0,0) node (A) {} -- ++(5:2) node(B) {}
(B.center) -- +(0,1.5) node (C) {}
(B.center) -- +(0,-1.5) node(D) {} 
+(0,0) -- ++(-10:1.5) node(E) {}
+(0,0) -- ++(10:1.5) node(F) {}
+(0,0) -- +(120:1) node(G) {}
+(0,0) -- +(60:1) node(H) {} 
+(0,0) -- ++(-5:1) node(I) {} 
+(0,0) -- +(30:1) node(J) {} 
+(0,0) -- +(-30:1) node(K) {}
(E.center) -- ++(0,-1)
+(0,0) -- +(-60:1) 
+(0,0) -- +(-120:1);
\ddgrape[180]{A};
\ddgrape[0]{C}; \ddgrape[90]{C}; \ddgrape[180]{C};
\ddgrape[-90]{D};
\ddgrape[-90]{F};
\ddgrape[120]{G};
\ddgrape[60]{H};
\ddgrape[-30]{K};
\end{scope}

\draw[thick,fill] 
(0,0) -- ++(5:2)
(B.center) -- +(0,-1.5)
+(0,0) -- ++(-10:1.5)
+(0,0) -- ++(10:1.5)
+(0,0) -- +(120:1)
+(0,0) -- +(60:1)
+(0,0) -- ++(-5:1)
+(0,0) -- +(30:1)
+(0,0) -- +(-30:1)
(E.center) -- ++(0,-1)
+(0,0) -- +(-60:1) 
+(0,0) -- +(-120:1);

\draw[thick,blue,<-] (B.center)+(0,2pt) -- (C.center);

\dgrape[180]{A};
\dgrape[0]{C}; \dgrape[90]{C}; \dgrape[180]{C};
\dgrape[-90]{D};
\dgrape[-90]{F};
\dgrape[120]{G};
\dgrape[60]{H};
\dgrape[-30]{K};

\draw[thick, fill] 
(A.center) circle (1.5pt)
(C.center) circle (1.5pt)
(D.center) circle (1.5pt)
(E.center) circle (1.5pt)
(E.center) ++(0,-1) circle (1.5pt)
(E.center) ++(0,-1) ++(-60:1) circle (1.5pt)
(E.center) ++(0,-1) ++(-120:1) circle (1.5pt)
(F.center) circle (1.5pt)
(G.center) circle (1.5pt)
(H.center) circle (1.5pt)
(I.center) circle (1.5pt)
(J.center) circle (1.5pt)
(K.center) circle (1.5pt);

\draw[thick,fill,blue]
(B.center) circle (1.5pt);

\draw[thick,fill]
(C.center) ++ (30:0.5) node[above right] {\(\vertex\)}
(C.center) ++ (0:{1.5/sqrt(3)}) node[right] {\(\vertex\)}
(C.center) ++ (-30:0.5) node[below] {\(\vertex\)}
(C.center) node[left] {\(\vertex\)}
(B.center) ++ (0,0.75) circle (1.5pt) node[left] {\(\vertex\)}
(B.center) node[above left] {\(\vertex\)}
(B.center) ++ (-10:0.75) circle (1.5pt) node[above] {\(\vertex\)}
(E.center) node[above] {\(\vertex\)}
(E.center) ++ (10:0.75) circle (1.5pt) node[above] {\(\vertex\)}
(F.center) node[left,yshift=1ex] {\(\vertex\)}
(F.center) ++ (120:0.5) circle (1.5pt) node[left] {\(\vertex\)}
(G.center) node[left] {\(\vertex\)}
(G.center) ++ (90:0.5) node[right] {\(\vertex\)}
(G.center) ++ (120:{1.5/sqrt(3)}) node[above left] {\(\vertex\)}
(G.center) ++ (150:0.5) node[below left] {\(\vertex\)}
(F.center) ++ (60:0.5) circle (1.5pt) node[right] {\(\vertex\)}
(H.center) node[left] {\(\vertex\)}
(H.center) ++ (30:0.5) node[below right] {\(\vertex\)}
(H.center) ++ (60:{1.5/sqrt(3)}) node[above right] {\(\vertex\)}
(H.center) ++ (90:0.5) node[left] {\(\vertex\)}
(F.center) ++ (-5:0.5) circle (1.5pt) node[above] {\(\vertex\)}
(I.center) node[above] {\(\vertex\)}
(I.center) ++ (30:0.5) circle (1.5pt) node[above] {\(\vertex\)}
(J.center) node[right] {\(\vertex\)}
(I.center) ++ (-30:0.5) circle (1.5pt) node[below] {\(\vertex\)}
(K.center) node[above] {\(\vertex\)}
(K.center) ++ (-60:0.5) node[below left] {\(\vertex\)}
(K.center) ++ (-30:{1.5/sqrt(3)}) node[below right] {\(\vertex\)}
(K.center) ++ (0:0.5) node[above] {\(\vertex\)}
(F.center) ++ (-120:0.5) node[left] {\(\vertex\)}
(F.center) ++ (-90:{1.5/sqrt(3)}) node[below] {\(\vertex\)}
(F.center) ++ (-60:0.5) node[right] {\(\vertex\)}
(E.center) ++ (0,-0.5) circle (1.5pt) node[right] {\(\vertex\)}
(E.center) ++ (0,-1) node[right] {\(\vertex\)}
(E.center) ++ (0,-1) ++ (-60:0.5) circle (1.5pt) node[right] {\(\vertex\)}
(E.center) ++ (0,-1) ++ (-60:1) node[below] {\(\vertex\)}
(E.center) ++ (0,-1) ++ (-120:0.5) circle (1.5pt) node[left] {\(\vertex\)}
(E.center) ++ (0,-1) ++ (-120:1) node[below] {\(\vertex\)}
(B.center) ++ (0,-0.75) circle (1.5pt) node[right] {\(\vertex\)}
(D.center) node[right] {\(\vertex\)}
(D.center) ++ (-120:0.5) node[above left] {\(\vertex\)}
(D.center) ++ (-90:{1.5/sqrt(3)}) node[below] {\(\vertex\)}
(D.center) ++ (-60:0.5) node[above right] {\(\vertex\)}
(A.center) ++ (5:1) circle (1.5pt) node[above] {\(\vertex\)}
(A.center) node[above] {\(\vertex\)}
(A.center) ++ (150:0.5) node[above right] {\(\vertex\)}
(A.center) ++ (180:{1.5/sqrt(3)}) node[left] {\(\vertex\)}
(A.center) ++ (210:0.5) node[below right] {\(\vertex\)}
(C.center) ++ (150:0.5) node[above left] {\(\vertex\)}
(C.center) ++ (180:{1.5/sqrt(3)}) node[left] {\(\vertex\)}
(C.center) ++ (210:0.5) node[below] {\(\vertex\)}
(C.center) ++ (60:0.5) node[above right] {\(\vertex\)}
(C.center) ++ (90:{1.5/sqrt(3)}) node[above] {\(\vertex\)}
(C.center) ++ (120:0.5) node[above left] {\(\vertex\)}
;

\draw[->] (C.center) ++ (0:{1/sqrt(3)}) ++ (30:0.75) arc (30:-30:0.75);
\end{tikzpicture}
\]
\caption{A labeled maximal tree \(\sfT'\) of \(\Gamma'\cong(\sfT,\loops)\) and the oriented root.
The edges of \( \Gamma' \) that are not contained in \(\sfT'\) are denoted by dashed lines.}
\label{figure:maximal tree}
\end{figure}

For each critical \(i\)-cell \(c\), we regard \(c\) as a configuration consisting of half-edges and points in \(\Gamma\) as follows:
\begin{enumerate}
\item The vertices in \(c\) are regarded as marked points.
\item Each edge in \(c\) adjacent to an essential vertex \(\sfv\) is regarded as the first or second half-edge of an edge adjacent to \(\sfv\) of \(\Gamma\).
\end{enumerate}

If \(\sfw_0\) is a leaf of \(\Gamma\), then it is easy to check that \(c\) is a valid \(\HE\)-configuration with respect to the oriented root \((\sfw_0',\sfe_0)\) for the unique edge \(\sfe_0\) of \(\sfT\) adjacent to \(\sfw_0\) and its endpoint \(\sfw_0'\) different from \(\sfw_0\).
See \Cref{figure:critical vs HE 1}.

Otherwise, suppose that \(\sfw_0\) is a vertex obtained by a
subdivision of a cycle \(\sfC_0\) attached to a leaf \(\sfv\) in
\(\sfT\). For the edge \(\sfe_0\) adjacent to \(\sfv\) with the
endpoint \(\sfv'\) different from \(\sfv\), we regard
\((\sfv',\sfe_0)\) as the oriented root for \(\Gamma\) as well. Then
\(c\) satisfies the condition of the \(\HE\)-configuration locally at
every vertex of \(\sfT\) possibly except for the vertex \(\sfv'\).
That is, if the restriction of \(c\) at \(\Gamma_{\sfv'}\) is not a
valid \(\HE\)-configuration, then \(c\) has an half-edge on
\(\sfe_0\). In this case, since \(c\) never has the first half-edge on
\(\sfC_0\) containing \(\sfw_0\), by replacing the half-edge on
\(\sfe_0\) with the first half-edge on \(\sfC_0\), we have a valid
\(\HE\)-configuration \(c'\). See \Cref{figure:critical vs HE 2}.

\begin{figure}[ht]
\centering
\begin{subfigure}{\textwidth}
\[
\begin{tikzcd}
\begin{tikzpicture}[baseline=-.5ex]
\draw[line width=3, red] (0,0) -- ++(0.5,0);
\draw[thick,fill] (0,-1) circle (2pt) node[below] {\(\mathsf{0}\)} -- ++(0,0.5) circle (2pt) -- ++(0,0.5) circle (2pt) ++(0.5,0) circle (2pt) -- ++(0.5,0) circle (2pt)
(-1,0) circle (2pt) -- ++(0.5,0) -- ++(0.5,0);
\draw[thick, fill] (0,0) -- +({-30+90}:0.5) +({30+90}:0.5) circle (2pt) (0,{1.5/sqrt(3)}) circle (2pt);
\draw[thick, dashed] +(0,0) -- +({+30+90}:0.5);
\draw[thick] (0,0) +({30+90}:0.5) arc ({120+90}:{-120+90}:{0.5 / sqrt(3)});
\draw[thick,fill,red] (-0.5,0) circle (2pt)
(60:0.5) circle (2pt);
\end{tikzpicture}
\ar[r,"\Xi"]&
\begin{tikzpicture}[baseline=-.5ex]
\draw[thick,fill,blue] (0,-1) circle (2pt) node[below] {\(*\)} -- ++(0,1);
\draw[line width=3, red] (0,0) -- ++(0.5,0);
\draw[thick,fill] (0,0) circle (2pt) ++(0.5,0) -- ++(0.5,0) circle (2pt)
(-1,0) circle (2pt) -- ++(0.5,0) -- ++(0.5,0);
\draw[thick, fill] (0,0) -- +({-30+90}:0.5) +({30+90}:0.5);
\draw[thick] +(0,0) -- +({+30+90}:0.5);
\draw[thick] (0,0) +({30+90}:0.5) arc ({120+90}:{-120+90}:{0.5 / sqrt(3)});
\draw[thick,fill,green] (-0.5,0) circle (2pt)
(60:0.5) circle (2pt);
\end{tikzpicture}&
\begin{tikzpicture}[baseline=-.5ex]
\draw[line width=3, red] (0,0) -- ++(120:0.5);
\draw[thick,fill] (0,-1) circle (2pt) node[below] {\(\mathsf{0}\)} -- ++(0,0.5) circle (2pt) -- ++(0,0.5) circle (2pt)
-- ++(0.5,0) -- ++(0.5,0) circle (2pt)
(-1,0) circle (2pt) -- ++(0.5,0) circle (2pt) -- ++(0.5,0);
\draw[thick, fill] (0,0) -- +({-30+90}:0.5) +({30+90}:0.5) circle (2pt) (0,{1.5/sqrt(3)}) circle (2pt);
\draw[thick] (0,0) +({30+90}:0.5) arc ({120+90}:{-120+90}:{0.5 / sqrt(3)});
\draw[thick,fill,red] (0.5,0) circle (2pt)
(60:0.5) circle (2pt);
\end{tikzpicture}
\ar[r,"\Xi"]&
\begin{tikzpicture}[baseline=-.5ex]
\draw[thick,fill,blue] (0,-1) circle (2pt) node[below] {\(*\)} -- ++(0,1);
\draw[line width=3, red] (0,0) -- ++(120:0.5);
\draw[thick,fill] (0,0) circle (2pt)
-- ++(0.5,0) -- ++(0.5,0) circle (2pt)
(-1,0) circle (2pt) -- ++(0.5,0) -- ++(0.5,0);
\draw[thick, fill] (0,0) -- +({-30+90}:0.5) +({30+90}:0.5);
\draw[thick] (0,0) +({30+90}:0.5) arc ({120+90}:{-120+90}:{0.5 / sqrt(3)});
\draw[thick,fill,green] (0.5,0) circle (2pt)
(60:0.5) circle (2pt);
\end{tikzpicture}
\end{tikzcd}
\]
\caption{The pivot edge has the smallest label.}
\label{figure:critical vs HE 1}
\end{subfigure}
\begin{subfigure}{\textwidth}
\[
\begin{tikzcd}
\begin{tikzpicture}[baseline=-.5ex]
\draw[thick,fill,blue] (0,-1) circle (2pt) node[below] {\(*\)} -- (0,0);
\draw[line width=3,red] (0,0) -- (60:0.5);
\draw[thick, fill, red] (0,-0.5) circle (2pt);
\draw[thick,fill] (0,-0.5) ++(0,0.5) circle (2pt) node (A) {};
\draw[thick, fill] (0,0) -- +({-30+0}:0.5) circle (2pt) +({30+0}:0.5) node[above, xshift=1ex] {\(\mathsf{0}\)} ({1.5/sqrt(3)},0) circle (2pt);
\draw[thick, dashed] (0,0) -- +({+30+0}:0.5);
\draw[thick] (0,0) +({30+0}:0.5) arc ({120+0}:{-120+0}:{0.5 / sqrt(3)});
\draw[thick, fill] (30:0.5) circle (2pt);
\draw[thick, fill] (0,0) +({-30+90}:0.5) +({30+90}:0.5) circle (2pt) (0,{1.5/sqrt(3)}) circle (2pt);
\draw[thick, dashed] (0,0) -- +({+30+90}:0.5);
\draw[thick] (0,0) +({30+90}:0.5) arc ({120+90}:{-120+90}:{0.5 / sqrt(3)});
\draw[thick, fill] (60:0.5) circle (2pt);
\draw[thick, fill] (0,0) -- +({-30+180}:0.5) +({30+180}:0.5) circle (2pt) ({-1.5/sqrt(3)},0) circle (2pt);
\draw[thick, dashed] (0,0) -- +({+30+180}:0.5);
\draw[thick] (0,0) +({30+180}:0.5) arc ({120+180}:{-120+180}:{0.5 / sqrt(3)});
\draw[thick, fill, red] (150:0.5) circle (2pt);
\end{tikzpicture}
\ar[r,"\Xi"]&
\begin{tikzpicture}[baseline=-.5ex]
\draw[thick,fill,blue] (0,-1) circle (2pt) node[below] {\(*\)} -- (0,0);
\draw[thick,fill] (0,0) -- +({-30+0}:0.5);
\draw[thick] (0,0) -- +({+30+0}:0.5);
\draw[thick] (0,0) +({30+0}:0.5) arc ({120+0}:{-120+0}:{0.5 / sqrt(3)});
\draw[thick] (0,0) -- +({+30+90}:0.5);
\draw[thick] (0,0) +({30+90}:0.5) arc ({120+90}:{-120+90}:{0.5 / sqrt(3)});
\draw[thick] (0,0) -- +({-30+180}:0.5);
\draw[thick] (0,0) -- +({+30+180}:0.5);
\draw[thick] (0,0) +({30+180}:0.5) arc ({120+180}:{-120+180}:{0.5 / sqrt(3)});
\draw[thick, fill, green] (0,-0.5) circle (2pt);
\draw[thick, fill, green] (150:0.5) circle (2pt);
\draw[line width=3,red] (0,0) -- (60:0.5);
\draw[thick, fill] (0,0) circle (2pt);
\end{tikzpicture}&
\begin{tikzpicture}[baseline=-.5ex]
\draw[thick,fill,blue] (0,-1) circle (2pt) node[below] {\(*\)} -- (0,-0.5);
\draw[line width=3,red] (0,-0.5) -- (0,0);
\draw[thick,fill] (0,-0.5) circle (2pt) ++(0,0.5) circle (2pt) node (A) {};
\draw[thick, fill] (0,0) -- +({-30+0}:0.5) circle (2pt) +({30+0}:0.5) node[above, xshift=1ex] {\(\mathsf{0}\)} ({1.5/sqrt(3)},0) circle (2pt);
\draw[thick, dashed] (0,0) -- +({+30+0}:0.5);
\draw[thick] (0,0) +({30+0}:0.5) arc ({120+0}:{-120+0}:{0.5 / sqrt(3)});
\draw[thick, fill, red] (30:0.5) circle (2pt);
\draw[thick, fill] (0,0) -- +({-30+90}:0.5) +({30+90}:0.5) circle (2pt) (0,{1.5/sqrt(3)}) circle (2pt);
\draw[thick, dashed] (0,0) -- +({+30+90}:0.5);
\draw[thick] (0,0) +({30+90}:0.5) arc ({120+90}:{-120+90}:{0.5 / sqrt(3)});
\draw[thick, fill, red] (60:0.5) circle (2pt);
\draw[thick, fill] (0,0) -- +({-30+180}:0.5) circle (2pt) +({30+180}:0.5) circle (2pt) ({-1.5/sqrt(3)},0) circle (2pt);
\draw[thick, dashed] (0,0) -- +({+30+180}:0.5);
\draw[thick] (0,0) +({30+180}:0.5) arc ({120+180}:{-120+180}:{0.5 / sqrt(3)});
\end{tikzpicture}
\ar[r,"\Xi"]&
\begin{tikzpicture}[baseline=-.5ex]
\draw[thick,fill,blue] (0,-1) circle (2pt) node[below] {\(*\)} -- (0,0);
\draw[line width=3,red] (0,0) -- (-30:0.5);
\draw[thick,fill] (0,0) circle (2pt) node (A) {};
\draw[thick] (0,0) -- +({+30+0}:0.5);
\draw[thick] (0,0) +({30+0}:0.5) arc ({120+0}:{-120+0}:{0.5 / sqrt(3)});
\draw[thick, fill, green] (30:0.5) circle (2pt);
\draw[thick, fill] (0,0) -- +({-30+90}:0.5);
\draw[thick] (0,0) -- +({+30+90}:0.5);
\draw[thick] (0,0) +({30+90}:0.5) arc ({120+90}:{-120+90}:{0.5 / sqrt(3)});
\draw[thick, fill, green] (60:0.5) circle (2pt);
\draw[thick, fill] (0,0) -- +({-30+180}:0.5);
\draw[thick] (0,0) -- +({+30+180}:0.5);
\draw[thick] (0,0) +({30+180}:0.5) arc ({120+180}:{-120+180}:{0.5 / sqrt(3)});
\end{tikzpicture}
\end{tikzcd}
\]
\caption{The pivot edge does not have the smallest label.}
\label{figure:critical vs HE 2}
\end{subfigure}
\caption{The correspondence between critical cells and \(\HE\)-configurations.}
\end{figure}

In summary, for each critical cell \(c\), we have a valid
\(\HE\)-configuration denoted by \(\Xi(c)\), and the following lemma,
whose proof is omitted, holds.

\begin{lem}\label{lemma:critical vs HE}
  Let \(\Gamma=(\sfT,\loops)\) be a bunch of grapes, and let
  \(\Gamma'\cong\Gamma\) be a sufficiently subdivided graph for \(k\).
  Then the function \(\Xi:\crit_i(UD_k\Gamma',W)\to\HE^i_\Gamma(k)\)
  is a bijection.
\end{lem}

The following corollary is a generalization of \cite[Theorem~3.7]{Farley2006} from trees to bunches of grapes.

\begin{cor}
  Using the notation from \Cref{lemma:critical vs HE}, the boundary map
  \(\pi f^\infty\partial\) on \(M_\bullet(UD_k\Gamma',W)\) vanishes.
\end{cor}
\begin{proof}
  By \Cref{thm:combinatorial description for homology cycles},
  \Cref{lemma:critical vs HE}, and the definition of
  \(M_i(UD_k\Gamma',W)\),
\[
\dim_\bbF H_i(B_k\Gamma) = \#\HE^i_\Gamma(k) = \#\crit_i(UD_k\Gamma',W) =\dim_\bbF M_i(UD_k\Gamma',W),
\]
which implies that the boundary map vanishes.
\end{proof}
\begin{remark}
  It is worth noting that the boundary map \(\pi f^\infty\partial\)
  depends on the choice of the function \(W\), which in turn depends
  on the choice of the maximal tree \(\sfT'\). In other words, it may
  not vanish for other choices of maximal trees.
\end{remark}

Let us denote the isomorphism \(M_i(UD_k\Gamma',W)\to H_i(B_k\Gamma)\)
by \(\Phi\). Then, for each critical cell \(c\), the two homology
cycles \(\Psi(\Xi(c))\) and \(\Phi(c)\) do not coincide in general,
but \(\Phi(c)\) is a linear combination of \(i\)-cycles coming from
\(\Psi\). See \Cref{figure:critical vs HE class} for an example. We
omit the details.

\begin{figure}[ht]
\centering
\begin{align*}
\Psi\left(
\begin{tikzpicture}[baseline=-.5ex]
\draw[thick,fill,blue] (0,-1) circle (2pt) node[below] {\(*\)} -- ++(0,1);
\draw[line width=3, red] (0,0) -- ++(0.5,0);
\draw[thick,fill] (0,0) circle (2pt) ++(0.5,0) -- ++(0.5,0) circle (2pt)
(-1,0) circle (2pt) -- ++(0.5,0) -- ++(0.5,0);
\draw[thick, fill] (0,0) -- +({-30+90}:0.5) +({30+90}:0.5);
\draw[thick] +(0,0) -- +({+30+90}:0.5);
\draw[thick] (0,0) +({30+90}:0.5) arc ({120+90}:{-120+90}:{0.5 / sqrt(3)});
\draw[thick,fill,green] (-0.5,0) circle (2pt)
(60:0.5) circle (2pt);
\end{tikzpicture}
\right)&=
\left(
\begin{tikzpicture}[baseline=-.5ex]
\draw[thick,fill,blue] (0,-1) circle (2pt) node[below] {\(*\)} -- ++(0,1);
\draw[line width=3, red] (0,0) -- ++(0.5,0) (0,0) -- ++(0,-0.5) (0,0) -- (60:0.5);
\draw[thick,fill] (0,0) circle (2pt) ++(0.5,0) -- ++(0.5,0) circle (2pt)
(-1,0) circle (2pt) -- ++(0.5,0) -- ++(0.5,0);
\draw[thick] +(0,0) -- +({+30+90}:0.5);
\draw[thick] (0,0) +({30+90}:0.5) arc ({120+90}:{-120+90}:{0.5 / sqrt(3)});
\draw[thick,fill,green] (-0.5,0) circle (2pt);
\draw[thick, ->, gray] (-0.3,0) arc (180:540:0.3);
\end{tikzpicture}
\right)
\\
\Phi\left(
\begin{tikzpicture}[baseline=-.5ex]
\draw[line width=3, red] (0,0) -- ++(0.5,0);
\draw[thick,fill] (0,-1) circle (2pt) node[below] {\(0\)} -- ++(0,0.5) circle (2pt) -- ++(0,0.5) circle (2pt) ++(0.5,0) circle (2pt) -- ++(0.5,0) circle (2pt)
(-1,0) circle (2pt) -- ++(0.5,0) -- ++(0.5,0) (0,{1.5/sqrt(3)}) circle (2pt);
\draw[thick, fill] (0,0) -- +({-30+90}:0.5) +({30+90}:0.5) circle (2pt);
\draw[thick, dashed] +(0,0) -- +({+30+90}:0.5);
\draw[thick] (0,0) +({30+90}:0.5) arc ({120+90}:{-120+90}:{0.5 / sqrt(3)});
\draw[thick,fill,red] (-0.5,0) circle (2pt)
(60:0.5) circle (2pt);
\end{tikzpicture}
\right)&=
\left(
\begin{tikzpicture}[baseline=-.5ex]
\draw[thick,fill] (0,-1) node[below] {\(0\)} -- ++(0,0.5) -- ++(0,0.5)
(0,0) -- ++(0.5,0) -- ++(0.5,0) circle (2pt)
(-1,0) circle (2pt) -- ++(0.5,0) -- ++(0.5,0) (0,{1.5/sqrt(3)}) circle (2pt);
\draw[thick, fill] (0,0) -- +({-30+90}:0.5) +({30+90}:0.5) circle (2pt);
\draw[thick, dashed] +(0,0) -- +({+30+90}:0.5);
\draw[thick] (0,0) +({30+90}:0.5) arc ({120+90}:{-120+90}:{0.5 / sqrt(3)});
\draw[thick,fill,red] (-0.5,0) circle (2pt)
(60:0.5) circle (2pt)
(0.5,0) circle (2pt)
;
\draw[thick,fill] (0,-1) circle (2pt) (0,-0.5) circle (2pt) (0,0) circle (2pt);
\draw[red, <-] (0.1,-0.1) -- ++(0.3,0) node[below] {\(C\)};
\draw[red, <-,rounded corners] (0.15,-0.5) -- (-15:{0.15/cos(15)}) -- ++(60:0.5) node[right] {\(B\)};
\draw[red, <-,rounded corners] (-0.15,-1) -- (-0.15,-0.15) -- (-0.5,-0.15) node[below] {\(A\)};
\end{tikzpicture}
\right)^{-1}
\cdot
\left(
\begin{tikzpicture}[baseline=-.5ex]
\draw[thick,fill] (0,-1) circle (2pt) node[below] {\(0\)} -- ++(0,0.5) circle (2pt) -- ++(0,0.5) circle (2pt)
-- ++(0.5,0) -- ++(0.5,0) circle (2pt)
(-1,0) circle (2pt) -- ++(0.5,0) -- ++(0.5,0) (0,{1.5/sqrt(3)}) circle (2pt);
\draw[thick, fill] (0,0) -- +({-30+90}:0.5) +({30+90}:0.5) circle (2pt);
\draw[thick, dashed] +(0,0) -- +({+30+90}:0.5);
\draw[thick] (0,0) +({30+90}:0.5) arc ({120+90}:{-120+90}:{0.5 / sqrt(3)});
\draw[thick,fill,red] (-0.5,0) circle (2pt) (0.5,0) circle (2pt)
(60:0.5) circle (2pt);
\draw[red, <-] (0.1,-0.1) -- ++(0.3,0);
\end{tikzpicture}
\right)\cdot
\left(
\begin{tikzpicture}[baseline=-.5ex]
\draw[thick,fill] (0,-1) circle (2pt) node[below] {\(0\)} -- ++(0,0.5) circle (2pt) -- ++(0,0.5)
(0,0) -- ++(0.5,0) circle (2pt) -- ++(0.5,0) circle (2pt)
(-1,0) circle (2pt) -- ++(0.5,0) -- ++(0.5,0) (0,{1.5/sqrt(3)}) circle (2pt);
\draw[thick, fill] (0,0) -- +({-30+90}:0.5) +({30+90}:0.5) circle (2pt);
\draw[thick, dashed] +(0,0) -- +({+30+90}:0.5);
\draw[thick] (0,0) +({30+90}:0.5) arc ({120+90}:{-120+90}:{0.5 / sqrt(3)});
\draw[thick,fill,red] (-0.5,0) circle (2pt) (0,0) circle (2pt)
(60:0.5) circle (2pt);
\draw[red, ->] (-0.15,-0.1) -- (-0.15,-1) node[left] {\(A\)};
\draw[red, <-,rounded corners] (0.15,-0.5) -- (-15:{0.15/cos(15)}) -- ++(60:0.5) node[right] {\(C\)};
\draw[red, ->] (-0.4,-0.1) node[below] {\(B\)} -- (-0.1,-0.1);
\end{tikzpicture}
\right)=
\left(
\begin{tikzpicture}[baseline=-.5ex]
\draw[thick,fill] (0,-1) circle (2pt) node[below] {\(0\)} -- ++(0,0.5) circle (2pt) -- ++(0,0.5) circle (2pt)
-- ++(0.5,0) -- ++(0.5,0) circle (2pt)
(-1,0) circle (2pt) -- ++(0.5,0) -- ++(0.5,0) (0,{1.5/sqrt(3)}) circle (2pt);
\draw[thick, fill] (0,0) -- +({-30+90}:0.5) +({30+90}:0.5) circle (2pt);
\draw[thick, dashed] +(0,0) -- +({+30+90}:0.5);
\draw[thick] (0,0) +({30+90}:0.5) arc ({120+90}:{-120+90}:{0.5 / sqrt(3)});
\draw[thick,fill,red] (-0.5,0) circle (2pt) (0.5,0) circle (2pt)
(60:0.5) circle (2pt);
\draw[red, ->, rounded corners] (0.4,-0.1) -- (0.1,-0.1) -- (0.1,-0.4) node[right] {\(A\)};
\draw[red, ->, rounded corners] (-30:{0.1/cos(30)}) ++(60:0.5) -- ++(60:-0.3) -- ++(0.3,0) node[above] {\(B\)};
\draw[red, ->, rounded corners] (-0.4,0.1) node[above] {\(C\)} -- (120:{0.1/cos(30)}) -- ++(60:0.4);
\draw[red, <-, rounded corners] (-0.4,-0.1) -- (-0.1,-0.1) -- (-0.1,-0.4) node[left] {\(D\)};
\end{tikzpicture}
\right)
\\
&=
\left(
\begin{tikzpicture}[baseline=-.5ex]
\draw[thick,fill,blue] (0,-1) circle (2pt) node[below] {\(*\)} -- ++(0,1);
\draw[line width=3, red] (0,0) -- ++(0.5,0) (0,0) -- ++(0,-0.5) (0,0) -- (60:0.5);
\draw[thick,fill] (0,0) circle (2pt) ++(0.5,0) -- ++(0.5,0) circle (2pt)
(-1,0) circle (2pt) -- ++(0.5,0) -- ++(0.5,0);
\draw[thick] +(0,0) -- +({+30+90}:0.5);
\draw[thick] (0,0) +({30+90}:0.5) arc ({120+90}:{-120+90}:{0.5 / sqrt(3)});
\draw[thick,fill,green] (-0.5,0) circle (2pt);
\draw[thick, <-, gray] (-0.3,0) arc (180:540:0.3);
\end{tikzpicture}
\right)+
\left(
\begin{tikzpicture}[baseline=-.5ex]
\draw[thick,fill,blue] (0,-1) circle (2pt) node[below] {\(*\)} -- ++(0,1);
\draw[line width=3, red] (0,0) -- ++(-0.5,0) (0,0) -- ++(0,-0.5) (0,0) -- (60:0.5);
\draw[thick,fill] (0,0) circle (2pt)
-- ++(0.5,0) -- ++(0.5,0) circle (2pt)
(-1,0) circle (2pt) -- ++(0.5,0) ++(0.5,0);
\draw[thick] +(0,0) -- +({+30+90}:0.5);
\draw[thick] (0,0) +({30+90}:0.5) arc ({120+90}:{-120+90}:{0.5 / sqrt(3)});
\draw[thick,fill,green] (0.5,0) circle (2pt);
\draw[thick, <-, gray] (-0.3,0) arc (180:540:0.3);
\end{tikzpicture}
\right)
\end{align*}
\caption{The comparison between the two bases.}
\label{figure:critical vs HE class}
\end{figure}

\bibliographystyle{alpha}

\end{document}